\documentclass[a4paper, 12pt]{amsart}
\usepackage{amsmath,amsfonts,amssymb,mathrsfs}
\usepackage{latexsym, amscd, epsfig}
\usepackage[utf8]{inputenc}
\usepackage{color}
\usepackage{amsmath}
\usepackage{amssymb}
\usepackage{amsthm}
\usepackage{graphicx}
\usepackage{esint}
\usepackage[colorlinks=true,linkcolor=blue]{hyperref}
\usepackage{verbatim}
\usepackage{accents}
\newcommand{\ubar}[1]{\underaccent{\bar}{#1}}
\usepackage{appendix}

\numberwithin{equation}{section}

\theoremstyle{plain}
\newtheorem{thm}{Theorem}
\newtheorem{prop}{Proposition}[section]
\newtheorem{lem}[prop]{Lemma}
\newtheorem{cor}[prop]{Corollary}

\newtheorem{rmk}{Remark}[section]

\newtheorem{problem}{Problem}

\newcommand {\R} {\mathbb{R}}

\newcommand {\p} {\partial}

\newcommand {\supp} {\text{supp}}

\DeclareMathOperator {\dist} {dist}

\def\Xint#1{\mathchoice
	{\XXint\displaystyle\textstyle{#1}}
	{\XXint\textstyle\scriptstyle{#1}}
	{\XXint\scriptstyle\scriptscriptstyle{#1}}
	{\XXint\scriptscriptstyle\scriptscriptstyle{#1}}
	\!\int}
\def\XXint#1#2#3{{\setbox0=\hbox{$#1{#2#3}{\int}$ }
		\vcenter{\hbox{$#2#3$ }}\kern-.6\wd0}}

\def\dashint{\Xint-}

\def\ba{\begin{array}}
	\def\ea{\end{array}}
\def\be{\begin{equation}}
	\def\ee{\end{equation}}

\def\f{\frac}
\def\p{\partial}

\def\cpsi{\check{\psi}}

\def\bn{{\bf n}}
\def\bu{{\bf u}}
\def\bp{{\bf p}}

\def\mcB{\mathcal{B}}

\def\mcD{\mathcal{D}}

\def\mcG{\mathcal{G}}

\def\mcJ{\mathcal{J}}
\def\mcK{\mathcal{K}}

\def\mcO{\mathcal{O}}

\def\mcR{\mathcal{R}}
\def\mcS{\mathcal{S}}

\def\mcU{\mathcal{U}}

\def\msA{\mathscr{A}}

\def\msG{\mathscr{G}}

\def\msR{\mathscr{R}}

\def\mfb{\mathfrak{b}}
\def\mfc{\mathfrak{c}}

\def\mfh{\mathfrak{h}}
\def\mfM{\mathfrak{M}}
\def\mfq{\mathfrak{q}}
\def\mfS{\mathfrak{S}}
\def\mft{\mathfrak{t}}
\def\mfT{\mathfrak{T}}

\def\baru{\bar{u}}
\def\barho{\bar{\rho}}
\def\barH{\bar{H}}

\def\unrho{\ubar{\rho}}
\def\unu{\ubar{u}}
\def\unH{\ubar{H}}

\usepackage{xspace}

\begin{document}
	\title[Variational Structure and Jet Flows]
	{Variational Structure and {Two-Dimensional} subsonic Jet Flows for Compressible Euler System with general incoming flows}
	
	\author{Yan Li}
	
	\address{School of Mathematical Sciences, Shanghai Jiao Tong University, 800 Dongchuan Road, Shanghai, 200240}
	\email{liyanly@sjtu.edu.cn}
	
	\author{Wenhui Shi}
	\address{School of Mathematics, Monash University,  9 Rainforest Walk, Clayton, VIC 3800, Australia}
	\email{shi@math1.rwth-aachen.de}

	\author{Lan Tang}
	\address{School of Mathematics and Statistics, Central China Normal University, Wuhan, Hubei 430079, China}
	\email{lantang@mail.ccnu.edu.cn}
	
	\author{Chunjing Xie}
	\address{School of mathematical Sciences, Institute of Natural Sciences,
		Ministry of Education Key Laboratory of Scientific and Engineering Computing,
		and CMA-Shanghai, Shanghai Jiao Tong University, 800 Dongchuan Road, Shanghai, China}
	\email{cjxie@sjtu.edu.cn}

	\begin{abstract}
		In this paper, we proved the well-posedness theory of compressible subsonic jet flows for {two-dimensional} steady Euler system with {\it general} incoming horizontal velocity {as long as the flux is larger than a critical value}. 
			One of the key observations is that the stream function formulation for {two-dimensional} compressible steady Euler system enjoys a variational structure even when the flows have nontrivial vorticity, so that the jet problem can be reformulated as a domain variation problem. {This variational structure helps to adapt} the framework developed by Alt, Caffarelli, and Friedman to study {the jet problem, which is a Bernoulli type free boundary problem. A major technical point to analyze the jet flows} is that the inhomogeneous terms in the rescaled equation near the free boundary are always small, even when the vorticity of the flows is big. 
	\end{abstract}
	
	\keywords{steady Euler equations, variational structure, jet, subsonic flows, free boundary,  vorticity.}
	\subjclass[2010]{%\AMSMOS
	  35Q31, 35R35, 35J20, 35J70, 35M32, 76N10}

	\thanks{Updated on \today}
	
	\maketitle

	\section{Introduction and main results}
	\subsection{Background and motivation}\label{sec:background} One of the most interesting problems in fluid dynamics is the study of flows through nozzles and the associated free boundary problems, such as shocks, jets, and cavities (\cite{CF48, Batchelor, Bersbook}). Historically lots of works have been done in the irrotataional (zero vorticity) case, for which we give a brief review.  The problem of finding steady irrotational subsonic flows in a two-dimensional infinitely long \emph{fixed} nozzle was first posed by Bers in \cite{Bersbook}. For this problem, properties of smooth solutions were studied by Gilbarg (\cite{Gilbarg1}), and the well-posedness theory when the mass flux is no larger than a critical number was first established in \cite{XX1}, cf. \cite{XX2, DXY, HWW} for further generalizations.  
	%It was also proved the existence of subsonic-sonic flows in two dimensional nozzles in \cite{XX1}.
	%The existence of weak solutions for subsonic-sonic flows was also established in \cite{XX1} if the flux of the flow tends to the critical number.  
	%These results  in \cite{XX1} were later generalized to the flows in three dimensional axisymmetric nozzles and multidimensional nozzles in  \cite{XX2} and  \cite{DXY, HWW}, respectively.
	Compared with the above nozzle flow problem, the jet problem, where the flow is partially bounded by fixed nozzles and partially free, is  more challenging. Early investigations of the jet problem relied on hodograph transformation and complex analysis techniques, which mainly deal with two-dimensional incompressible flows with restrictive assumptions on nozzles, cf. \cite{Chaplygin, Birkhoff, Gilbargjets} and references therein.
	%The investigation of the jet problem was pioneered by Chaplygin, who developed a hodograph method to show the existence of two dimensional steady irrotational subsonic jets emerging from a straight nozzle \cite{Chaplygin}. This method combined with complex analysis techniques was later used to solve incompressible jet problems with more general nozzles, cf. \cite{Birkhoff, Gilbargjets, Cook} and references therein. 
	A major breakthrough for the jet problem was made by Alt, Caffarelli, and Friedman in 1980s. They developed a systematic regularity theory for the free boundary problems which are equivalent to domain variation problems (\cite{AC81, ACF84, Friedman82}). With the aid of this general theory they established  the well-posedness for two-dimensional and three-dimensional axisymmetric jet and cavity problems for  steady irrotational incompressible and compressible subsonic flows (\cite{ACF83,ACF85}). 
	This approach is recently applied to investigate irrotational subsonic jet problem with different boundary conditions, cf. \cite{CDXiang}. An important progress for the existence of irrotational subsonic-sonic jet flows and jet flows in bounded domains has been made in  \cite{WX2,WX1} via a careful study of the so-called Chaplygin equation.

	Vorticity plays an important role in understanding fluid dynamics not only mathematically but also physically (\cite{Batchelor,Majda}). For example, when vorticity appears, the steady water waves (a class of typical free boundary problems) have many new features comparing with the irrotational flows, eg. the existence of stagnation points, see \cite{Strauss} and references therein. 
	{The jet flows with non-zero vorticity are not only generalization of irrotational flows, but also have their own physical significance. For example, a compressible rotational (non-zero vorticity) jet flow should be an important constituent of} the flow pattern with a transonic shock inside and a jet issued from the nozzle (\cite{CF48}). In the past two decades significant progress has been made on the stability of transonic shocks in a nozzle. It was shown in \cite{XY1,XY2} that the transonic shock problem for isentropic irrotational flows in a finite nozzle is in general ill-posed with prescribed exit pressure. However, in the class of rotational flows which are subsonic in the downstream region, transonic shocks in a divergent nozzle were shown to be stable under the perturbation of the exit pressure, see \cite{LXY, ChenSX, LXY1, WXX} in different settings. 
	%considering the isentropic or full Euler system in a divergent nozzle where the vorticity of the flows is not zero in the downstream subsonic region}, 
{Therefore, }in order to study the flow patterns with both transonic shock and jet, it is necessary to look at the compressible rotational jet flows.

For the rotational flows the steady {compressible} Euler system becomes much more complicated compared with that for irrotational flows. 
%\tbl{ This sentence did not say much. One can instead say "when passing to the non-zero vorticity case, new phenomenon... or situation are significantly different. Example and references."}
One of the main difficulties is that the steady Euler system for subsonic flows is a hyperbolic-elliptic coupled system. Fortunately, it is interesting that in \cite{XX3} the steady compressible Euler system for subsonic flows was reduced to a single second order quasilinear elliptic equation with inhomogeneous term, which represents the information of the hyperbolic mode.  This observation and a careful study for the associated quasilinear elliptic equation allow to establish the well-posedness theory for subsonic flows with non-zero small vorticity in an infinitely long fixed nozzle (\cite{XX3}), and later in the case with large vorticity (\cite{DXX,CHWX}). For more results in this direction, one may refer to \cite{DD, DD1, ChenXie1, CDX, DL, DX, CDXX} and references therein for various generalizations. 

There are only very few analytical results concerning the effect of vorticity on jet flows. The well-posedness of the
steady incompressible jet flows with small non-zero vorticity was {established} in \cite{Friedmanrot}.  This work crucially relied on the fact that  the steady incompressible Euler system can be reduced to a semilinear elliptic equation for the stream function which has a variational structure, hence the jet problem can be formulated as a domain variation problem, see \cite{ChengDu} for recent progress in this direction for incompressible jet flows.  As far as the jet problems for {compressible} flows with \emph{non-zero} vorticity are concerned, to the authors' knowledge, there are no rigorous analytical results available. A major obstacle is that it is unknown whether such problem enjoys some nice structure for the analysis on the associated free boundary problem. In this paper, we give an affirmative answer to this question and prove the existence and {uniqueness} of the steady compressible jet flows with general vorticity.

%\newpage

\subsection{The problem and main results} Two-dimensional steady isentropic compressible ideal flows are governed by the
following Euler system
\begin{equation}\label{a0}
	\left\{
	\begin{aligned}
		&\nabla\cdot(\rho\mathbf{u}) =0,\\
		&\rho\mathbf{u}\cdot \nabla\mathbf{u}+ \nabla p =0,
	\end{aligned}
	\right.
\end{equation}
where $\mathbf{u}=(u_1,u_2)$ denotes the flow velocity, $\rho$ is the density, and $p=p(\rho)$
is the pressure of the flow. Suppose that the flow is a polytropic gas, after nondimensionalization,  the equation of state can be written as  $ p(\rho)=
\frac{\rho^{\gamma}}{\gamma}$, where the constant $\gamma>1$ is called
the adiabatic exponent. The local sound speed and the Mach number of the flow are defined as
\begin{equation}\label{eq:sound}
	c(\rho)=\sqrt{p'(\rho)}=\rho^{\frac{\gamma-1}{2}}\quad \text{and}\quad M=\f{|\textbf{u}|}{c(\rho)},
\end{equation}
respectively.
The flow is called \emph{subsonic} if $M<1$, \emph{sonic} if $M=1$ and \emph{supersonic} if $M>1$.

In this paper we consider {a symmetric} nozzle in $\R^2$ bounded by two solid boundaries. Without loss of generality we assume that the nozzle is symmetric about $x_1$-axis. 
We denote the symmetry axis and the upper solid boundary of the nozzle by 
\begin{equation}\label{eq:nozzle}
	S_0:=\{(x_1, 0)| x_1\in \mathbb{R}\}\quad \text{and}\quad
	S_1:=\{(x_1, x_2)| x_1=\Theta(x_2), \ x_2\in [1,\bar H) \}
\end{equation}
respectively, 
{where $\Theta\in C^{1,\bar\alpha}([1,\barH])$ 
	($\bar\alpha\in(0,1)$)}
for some given $\barH>1$, and satisfies
\begin{equation}\label{eq:nozzle1}
	\Theta(1)=0\quad \text{and}\quad  \lim_{x_2\rightarrow \barH-} \Theta(x_2)= -\infty,
\end{equation}
i.e., the orifice of the nozzle is at $A:= (0,1)$ and the nozzle is asymptotically horizontal with height $\barH$ at upstream $x_1\rightarrow -\infty$ (cf. Figure 1). In order to study the uniqueness of solution, we also require that there exists   $h_*\in[1,\bar H)$ such that %the nozzle is monotone for  $x_2\in(h_*,\bar H)$, i.e.,
\begin{equation}\label{eq:nozzle2}
	\Theta'(x_2)\leq 0\quad\text{for } x_2\in (h_*, \bar{H}).
\end{equation}
The main goal of this paper is to study the following jet problem.
\begin{problem}\label{pb}
	Given a mass flux $Q>0$ and 
	{a positive horizontal velocity $\bar u=\bar u(x_2)$} of the flow at upstream as $x_1\rightarrow -\infty$, find $(\rho, \mathbf{u})$, the free boundary $\Gamma$ and the outer pressure $p_e$, which is assumed to be a constant, such that the following statements hold.
	\begin{enumerate}
		\item The free boundary $\Gamma$ joins the outlet of the nozzle as a continuous differentiable curve and tends asymptotically horizontal at downstream as $x_1\rightarrow \infty$. 
		\item The solution $(\rho, \mathbf{u})$ solves the Euler system \eqref{a0} in the flow region $\mcO$ bounded by $S_0$, $S_1$, and $\Gamma$. It takes the incoming data at upstream, i.e., 
		\begin{equation}\label{eq:asymp_up}
			{u_1(x_1,x_2)\to \bar u(x_2)}  
			\quad \text{as}\,\, x_1\to -\infty,
		\end{equation}
		and 
		\begin{equation}\label{Emassflux}
			\int_0^1 (\rho u_1)(0, x_2)\ dx_2=Q.
		\end{equation}
		Furthermore it satisfies the boundary conditions
		\begin{equation}\label{FBPbc}
			p(\rho)=p_e \text{ on}\,\, \Gamma \quad \text{and}\quad   \mathbf{u}\cdot \bn=0 \text{ on}\,\, S_1\cup\Gamma,
		\end{equation}
		where $\bf{n}$ is the unit normal along $S_1\cup \Gamma$.
	\end{enumerate}
\end{problem}
\begin{center}
	\includegraphics[height=5cm, width=10cm]{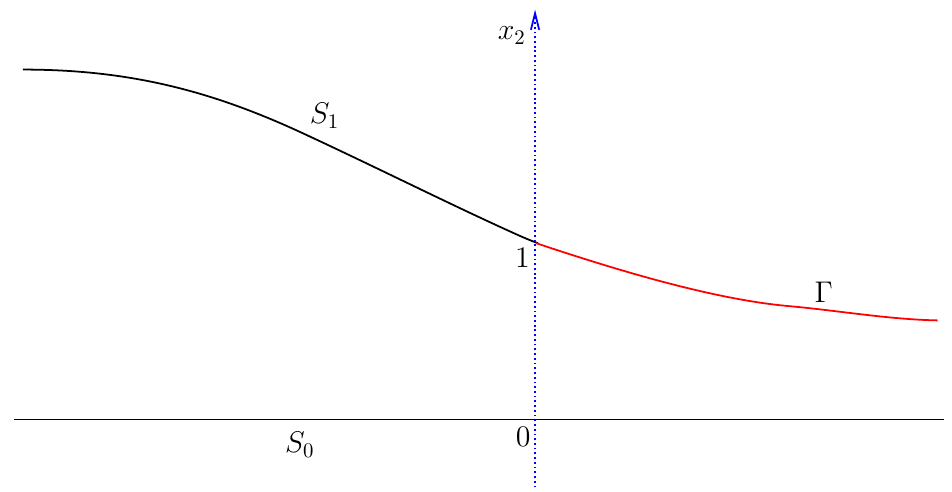}\\
	{\small Figure 1.  The jet problem}
\end{center}	

The main results in this paper can be stated as follows.
\begin{thm}\label{mainthm}
	Let $S_1$ defined in \eqref{eq:nozzle} satisfy \eqref{eq:nozzle1}--\eqref{eq:nozzle2} and $Q$ be a positive constant. If $\bar u\in C^{1,1}([0,\bar H])$ satisfies
	%Given a nozzle which satisfies \eqref{eq:nozzle}--\eqref{eq:nozzle1}. Given a positive horizontal velocity $\bar u\in C^{1,1}([0,\bar H])$ and mass flux $Q>0$ at the upstream. Suppose that
	\begin{equation}\label{cond:u0_eps0}
		\bar u(x_2)>0 \ \ \text{for } x_2\in[0,\bar H],\quad \bar u'(0)= 0,   \quad\text{and}\quad
		\bar u'(\barH)\geq0,
	\end{equation}
	then there exists $Q_c>0$ depending on $\bar u$, $\gamma$, and the nozzle, such that the following statements hold. 
	\begin{enumerate}
		\item[(i)] (Existence and properties of solutions) For any $Q>Q_c$, {there are functions  $\rho,\mathbf{u}\in C^{1,\alpha}(\mathcal{O})\cap C^0(\overline{\mathcal{O}})$ where $\mathcal O$ is the flow region, the free boundary $\Gamma$, and the outer pressure $p_e$ such that $(\rho,\mathbf u,\Gamma,p_e)$  solves Problem \ref{pb}. 
		Furthermore, the following properties hold.
		\begin{enumerate}
		\item The flow is globally uniformly subsonic  and has negative vertical velocity in the flow region $\mcO$, i.e.,
			\begin{equation}\label{Esubsonic}
				{\sup_{\overline{\mcO}}\frac{|\mathbf{u}|^2}{c^2(\rho)}<1}  
				\quad\text{and}\quad 
				u_2<0 \,\,\text{in}\,\,{\mcO}.
			\end{equation}
			\item (Smooth fit) The free boundary $\Gamma$ joins the outlet of the nozzle as a $C^1$ curve.
			\item The free boundary $\Gamma$ is given by a graph $x_1=\Upsilon(x_2)$, $x_2\in (\ubar H, 1]$ for some $\ubar H\in(0,1)$, where the function  $\Upsilon$ is $C^{2,\alpha}$ for any $\alpha\in (0,1)$ and $\lim_{x_2\rightarrow \ubar H+} \Upsilon(x_2)=\infty$. For $x_1$ sufficiently large, the free boundary  can also be written as $x_2=f(x_1)$ for some {$C^{2,\alpha}$} function $f$, which satisfies
			$$\lim_{x_1\rightarrow \infty}f(x_1)=\unH
			\quad\text{and}\quad  \lim_{x_1\rightarrow \infty}f'(x_1)=0.$$			
		\item (Upstream and downstream asymptotics) %The free boundary $\Gamma$ is asymptotically horizontal at downstream $x_1\rightarrow \infty$. Let $\ubar{H}$ be the asymptotic height of the free boundary at downstream $x_1\rightarrow \infty$. Then 
		For any $\alpha\in(0,1)$, there exist positive constants $\bar\rho$ and $\ubar{\rho}$, which are the upstream and downstream density respectively, and a positive function $\ubar {u}\in C^{1,\alpha}([0,\ubar H])$, which is the downstream horizontal velocity, such that
						\begin{equation}\label{upstreambehavior}
							\|(\rho,u_1, u_2)(x_1,\cdot)- (\barho, \bar{u}(\cdot), 0)\|_{C^{1,\alpha}_{loc}([0, \bar H))}\to 0 \quad \text{as}\,\, x_1\rightarrow -\infty
						\end{equation}
						and 
						\begin{equation}\label{downstreambehavior}
							\|(\rho,u_1, u_2)(x_1,\cdot)-(\ubar \rho, \ubar{u}(\cdot), 0)\|_{C^{1,\alpha}_{loc}([0, \ubar{H}))}\to 0 \quad \text{as}\,\, x_1\rightarrow \infty.
						\end{equation} 
						Moreover, the upstream density $\bar\rho$ and the downstream density $\ubar\rho$ satisfy 
						\begin{equation*}\label{eq:bu_rho}
							\barho=\frac{Q}{\int_0^{\bar H}\bar u(x_2)dx_2}
							\quad\text{and}\quad
							\unrho=(\gamma p_e)^{\frac1{\gamma}},
						\end{equation*}
						respectively; the downstream horizontal velocity $\ubar u$ and the downstream  height $\ubar{H}$ are also uniquely determined by $Q$, $\bar u$, $\gamma$, $\bar H$ and $p_e$.
				\end{enumerate}}
				
				\item[(ii)] (Uniqueness) The Euler flow which satisfies all properties in part (i) is unique.
				
				\item[(iii)] (Critical mass flux) {$Q_c$ is the critical mass flux} for the existence of 
				subsonic jet flow in the following sense: 
				{either
					\begin{equation*}
						\sup_{\overline{\mcO}}\frac{|\mathbf{u}|^2}{c^2(\rho)}\rightarrow
						1\quad\text{as}\ Q\rightarrow Q_c+,
					\end{equation*}
					or there is no $\sigma>0$ such that for all $Q\in(Q_c-\sigma,Q_c)$, there are Euler flows satisfying
					all properties in part (i) and
					\begin{equation*}%
						\sup_{Q\in(Q_c-\sigma,Q_c)}\sup_{\overline{\mcO}}\frac{\mathbf{|u|}^2}{c^2(\rho)}<1.
				\end{equation*}}
			\end{enumerate}
		\end{thm}

		A few remarks are in order.
		
		\begin{rmk}
			{One of the most interesting features of Theorem \ref{mainthm} is that the flows} can have arbitrary large vorticity. And there is neither the sign condition nor the smallness condition  on the vorticity of the incoming flows.
		\end{rmk}
		
		\begin{rmk}
			Together with the analysis in \cite{XX3}, all results 
			in this paper work for general equation of states $p=p(\rho)$ with $p'(\rho)>0$ and $p''(\rho)>0$.
		\end{rmk}
		
		\begin{rmk}
			The condition \eqref{eq:nozzle1} of the nozzle boundary $S_1$ defined in \eqref{eq:nozzle} is sufficient for proving the existence of solutions to the jet problem. The condition \eqref{eq:nozzle2} of $S_1$ at far field is only used for the uniqueness of the solution and the existence of the critical mass flux.
		\end{rmk}
		
		\begin{rmk}
			{The far field condition \eqref{eq:asymp_up} prescribes the horizontal velocity at the upstream. 
				If, instead of \eqref{eq:asymp_up},  the Bernoulli function (which is a hyperbolic mode for steady compressible Euler system) is prescribed at the upstream, as that has been done for the nozzle flows in \cite{XX3}, we can also prove the existence of compressible jet flows under suitable assumptions for the incoming Bernoulli function.}
		\end{rmk}
		
		\begin{rmk}
			Similar arguments also work for two-dimensional  non-isentropic flows and three-dimensional axisymmetric flows (\cite{LSX2}). Furthermore,
			the ideas developed in this paper can also be used to deal with the {cavity problem} for fluid with non-zero vorticity (\cite{LSX1}).
		\end{rmk}
		
		Here we outline the main ideas and the key  points in the proof of Theorem \ref{mainthm} as follows.
		
		With simple topological structure for streamlines, one can reformulate the Euler system in terms of the Bernoulli function and the vorticity (cf. Proposition \ref{Elemmaequivalent}). By {this equivalent reformulation},  the Euler system can be reduced into a single {second order quasilinear  equation} for the stream function with a complicated memory term representing the information of the hyperbolic mode. Moreover, the equation is elliptic if and only if the flow is subsonic (cf. Lemma \ref{lem:stream_equiv}).

		A key observation is that the quasilinear elliptic equation for stream function enjoys a variational structure as in the irrotational case (cf. Lemma \ref{lem:vari_psi}). As a consequence, one can formulate {the subsonic jet problem as a domain variation problem even when the flows have non-zero vorticity.} 
		%The inhomogeneous term produced by non-zero vorticity causes difficulty to apply the framework developed by Alt, Caffarelli and Friedman to prove the Lipschitz regularity of the free boundary. 

		To study the free boundary regularity, one needs only to consider the rescaled equation to analyze the linear decay and non-degeneracy of solutions near the free boundary. Another key observation of this paper is that the inhomogeneous terms become small terms in the rescaled equations so that the comparison principle can be used as a basic analysis tool even when the vorticity is large. In fact, we can still obtain the Lipschitz regularity of the solution to the domain variation problem even with the large vorticity of the flows. 
		
		The regularity of the free boundary away from the orifice is obtained with the aid of the regularity theory developed for the free boundary problem {to} inhomogeneous elliptic equations {in %by De Silva, Ferrari and Salsa 
			\cite{D11, DFS15} and related results.} To solve the jet problem,  a crucial step is to show that the solution obtained from the domain variation problem satisfies a monotonicity property,  which guarantees the equivalence of the Euler system and the domain variation formulation in terms of the stream function. This is {achieved} via comparison principles inspired by the proof in \cite{ACF85}. 
		
		The blowup technique plays an important role in the study of asymptotics of solutions near the free boundary. In our situation, blowup limits for solutions near the free boundary correspond to a Cauchy problem for the second order equation of the stream function. Note that {the coefficients}  for this equation depend on the regularity of vorticity in the upstream, which is usually not analytic. Hence the Cauchy-Kovalevskaya theorem, which has been used in \cite{ACF85} to deal with the problem for irrotational flows, cannot be directly used in this paper. 
		To overcome this difficulty, we adapt the unique continuation principle developed in \cite{KT01} to characterize the solutions of Cauchy problem for the associated quasilinear elliptic equation, which corresponds to flows with non-zero Lipschitz vorticity.
		
	 As long as one has the existence and uniqueness of the solution for the jet problem with sufficiently {large} incoming mass flux,  the existence of the critical mass flux can be established by the compactness arguments adapted from \cite{XX3}.
		
		The rest of this paper is organized as follows. In Section \ref{secstream}, the stream function formulations for both the Euler system and the jet problem are established. In Section \ref{secvar},  we give the variational formulation for the jet problem with truncations for domain and ellipticity. Section \ref{secreg} is devoted to the study for {regularity} of the truncated free boundary problem. The monotonicity property of the solution is established in Section \ref{secprop}, which is crucial for the equivalence of the stream function formulation and the Euler system. In order to remove the domain truncations later, we also prove some uniform estimates for {solutions of} the truncated problem in Section \ref{secprop}. In Section \ref{seccont},  the continuous fit and smooth fit of the free boundary are established. In Section \ref{secremove}, we remove the domain and subsonic truncations, and study the far fields asymptotic behavior of the solution. This {completes} the proof for the existence of solutions to the jet problem with {large} mass flux. The uniqueness of subsonic jet problem is proved in Section \ref{secunique}. Finally, the existence of a critical mass flux is established in Section \ref{SEcritical}. Some technical details are included in the appendix.
		
		In this paper, for convention the repeated indices mean the summation.
		
		\section{Stream function formulation and subsonic truncation}\label{secstream}
		In this section, we introduce the stream function formulation to reduce the Euler system into a single {second order} quasilinear equation, which is elliptic in the subsonic region and becomes singular elliptic at the sonic state, cf. Lemma \ref{lem:stream_equiv}. We also reformulate Problem \ref{pb} into a Bernoulli type free boundary problem for the quasilinear equation satisfied by the stream function, cf. Problem \ref{Pb2}. In order to deal with the possible degeneracy of the equation near the sonic state, a subsonic truncation is introduced so that the modified  equation is always uniformly elliptic.
		\subsection{The equation for the stream function}
		First, motivated by the analysis in  \cite{XX3}, one has the following equivalent formulation for the compressible Euler system.
		\begin{prop}\label{Elemmaequivalent}
			Let ${\tilde\mcO} \subset \mathbb{R}^2$ be the domain bounded by two streamlines $S_0=\{(x_1,0)| x_1\in \R\}$ and
			\[
			{\tilde S_1} :=\{(x_1, x_2)|x_1={\tilde \Theta(x_2)}, \ubar H< x_2< \bar H\},\]
			where  $0<\ubar H< \bar H<\infty$ and $\tilde \Theta:(\ubar H, \bar H)\rightarrow \R$  is a $C^1$ function with
			\[
			\lim_{x_2\to \bar H-}\tilde\Theta(x_2) =-\infty\quad\text{and}\quad  \lim_{x_2\to \ubar H+} \tilde\Theta(x_2) =\infty.
			\]
			Let $\rho:\overline{\tilde \mcO}\rightarrow (0,\infty)$ and $\bu=(u_1,u_2):\overline{\tilde\mcO}\rightarrow \R^2$ be $C^{1,1}$ in $\tilde\mcO$ and continuous up to $\p\tilde\mcO$ except finitely many points. Suppose that  $\bu$ satisfies the slip boundary condition $\bu\cdot \bn=0$ on $\partial\tilde\mcO$,   
			and $(\rho, \bu)$ satisfies the upstream asymptotics \eqref{upstreambehavior} with a positive constant $\bar \rho$ and a positive function $\bar u\in C^{1,1}([0,\bar H])$. 
			Moreover, suppose that 
			\begin{equation}\label{Ephorizontalvelocity}
				u_2< 0\quad \text{ in } \tilde\mcO.
			\end{equation}
			%and the following asymptotic behavior
			%\begin{equation}\label{EcequivstreamEuler}
			%	u_1,\,\,\rho,\,\, \text{and}\,\, \p_{x_2}u_2\,\, \text{are bounded,
				%		while}\,\, u_2,\,\,\p_{x_1}u_2,\,\,\text{and}\,\,\p_{x_2}\rho\rightarrow
			%	0,\,\,\text{as}\,\, x_1\rightarrow -\infty.
			%\end{equation}
			Then $(\rho, \bu)$ solves the Euler system \eqref{a0} in $\tilde\mcO$ if and only if $(\rho, \bu)$ satisfies
			\begin{equation}\label{eq:euler}
				\left\{
				\begin{aligned}
					&\nabla\cdot(\rho \bu)=0, \\
					&\bu \cdot \nabla \mathscr{B}(\rho, \bu) =0, \\
					&\bu\cdot \nabla \left(\frac{\omega}{\rho}\right) =0,
				\end{aligned}
				\right.
			\end{equation}
			where
			\begin{align*}
				\mathscr{B}(\rho, \bu):= \frac{|\bu|^2}{2}+h(\rho),\quad \omega:=\p_{x_1}u_2-\p_{x_2}u_1,\quad\text{and}\quad h(\rho):=\frac{\rho^{\gamma-1}}{\gamma-1}
			\end{align*}
			are the Bernoulli function, the vorticity, and the enthalpy of the flow, respectively.
		\end{prop}
	
	\begin{rmk}
		Clearly, the regularity requirement for a  classical  $(\rho, \bu)$ of the compressible Euler system \eqref{a0} is that $(\rho, \bu)\in C^{1}$ rather than $\bu\in C^{1,1}$. The assumption $\bu \in C^{1,1}$ makes sure that the last equation in \eqref{eq:euler} is satisfied in the classical sense.  In Proposition \ref{prop:equiv_sol}, we show that if $(\rho, \bu) \in C^{1,\alpha}$ is induced by the stream function associated with \eqref{eq:euler},  $(\rho, \bu)$ is still a classical solution of the original Euler system \eqref{a0}.
	\end{rmk}
		
		%\tre{What kind of regularity of $\bu$ is required in this proposition? I don't think we are able to get $\bu\in C^{1,1}$ from the steam function formulation (which corresponds to $\psi\in C^{2,1}$ or equivalently $W^{3,\infty}$ in the flow region)? This is simply because that $\psi$ satisfies the elliptic equation $\nabla\cdot (A\nabla \psi)=f$ with $f\in C^{0,1}$. Thus morally $\psi''$ satisfies an elliptic equation of the type $\nabla\cdot (A\nabla\psi)=f''$, where the RHS is in $W^{-1,\infty}$. This is the boarder-line case for the elliptic regularity, i.e. one perhaps cannot simply expect $\psi''\in C^{0,1}$. This regularity assumption is related to the Helmholtz decomposition used in the proof. Can one weaken the regularity assumption there?}
		\begin{proof}[Proof of Proposition \ref{Elemmaequivalent}]
			It is straightforward that {if $(\rho, \bu)$ solves
				Euler system \eqref{a0} in $\tilde\mcO$ with $\rho\in C^1(\tilde \mcO)$ and $\bu\in C^{1,1}(\tilde \mcO)$,} then it satisfies
			\eqref{eq:euler}.
			
			On the other hand, it follows from the first and the third equation in \eqref{eq:euler} that
			\begin{equation*}
				\partial_{x_2}(u_1\partial_{x_1}u_1+u_{2}\partial_{x_2}u_1 +\p_{x_1}h(\rho))-\partial_{x_1}(u_1\partial_{x_1}u_2+u_2\partial_{x_2}u_{2}+\p_{x_2}h(\rho))=0.
			\end{equation*}
			Therefore, there exists a function $\Phi$ such that
			\begin{equation*}
				\p_{x_1}\Phi=u_1\partial_{x_1}u_1+u_{2}\partial_{x_2}u_1 +\p_{x_1}h(\rho), \quad
				\p_{x_2}\Phi=u_1\partial_{x_1}u_2+u_2\partial_{x_2}u_{2}+\p_{x_2}h(\rho).
			\end{equation*}
			Thus the Bernoulli's law (the second equation in \eqref{eq:euler}) is equivalent to
			\begin{equation*}\label{EtransportPotential}
				(u_1,u_2)\cdot\nabla\Phi=0.
			\end{equation*}
			This implies that $\Phi$ is a constant along each  streamline.  {Note that the far field behavior \eqref{upstreambehavior} implies} $\p_{x_2}\Phi\rightarrow 0$ as $x_1\rightarrow -\infty$.  Hence one has
			\begin{equation}\label{EupstramPotentialderivative}
				\Phi\rightarrow C \quad \text{as}\,\, x_1\rightarrow -\infty.
			\end{equation}
			On the other hand, it follows from (\ref{Ephorizontalvelocity}) that through
			each point in $\tilde\mcO$, there is one and only one streamline {which satisfies}
			\begin{equation*}
				\left\{
				\begin{array}{ll}
					\displaystyle\frac{dx_1}{ds}=u_1(x_1(s),x_2(s)),\\ \\
					\displaystyle\frac{dx_2}{ds}=u_2(x_1(s),x_2(s)),
				\end{array}
				\right.
			\end{equation*}
			and can be defined globally in the domain. Furthermore, it follows from the continuity equation (the first equation in \eqref{eq:euler}) that any streamline through some point in $\tilde\mcO$ cannot touch $\partial\tilde\mcO$. 
			{Indeed, suppose there exists a streamline through $(-x_1^0,x_2^0)$ in $\tilde\mcO$ (without loss of generality assume that  $x_1^0>0$ is sufficiently large) which intersects $S_0$. 	
				Then by virtue of the continuity equation and the slip boundary condition along each streamline, one has
				$$0=\int_{0}^{x_2^0}\rho u(-x_1^0,s)ds.$$
				\begin{center}
					\includegraphics[height=4cm]{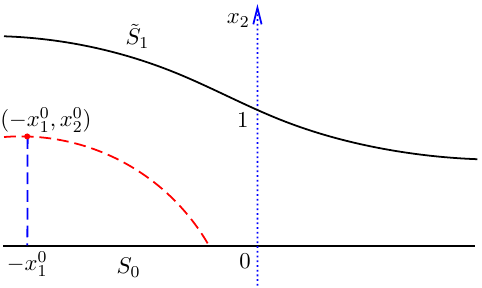}\hspace{3em}
					\includegraphics[height=4cm]{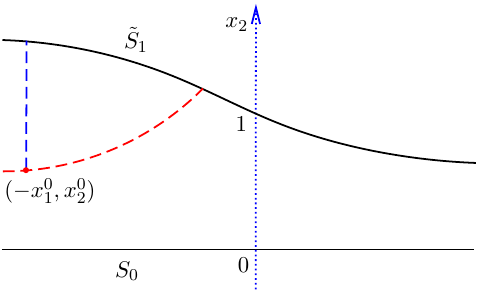}\\
					{\small Figure 2.}
				\end{center}	
				However, in view of the asymptotic behavior \eqref{upstreambehavior}, the term on the right-hand side of the above equation is positive. This contradiction implies that a streamline inside the domain cannot touch the boundary $S_0$. Similarly, one can show that a streamline through a point inside the domain cannot intersect with $\tilde S_1$.}	
			Therefore, one can conclude from \eqref{EupstramPotentialderivative} that $\Phi\equiv C$ in the whole domain $\tilde\mcO$. This implies that $\p_{x_1}\Phi=\p_{x_2}\Phi\equiv 0$ in $\tilde\mcO$, i.e.,
			\begin{equation}\label{label_10}
				u_1\partial_{x_1}u_1+u_{2}\partial_{x_2}u_1 +\p_{x_1}h(\rho)=0 \quad \text{and}\quad
				u_1\partial_{x_1}u_2+u_2\partial_{x_2}u_{2}+\p_{x_2}h(\rho)=0
			\end{equation}
			hold globally in $\tilde\mcO$. The above two equations together with the continuity equation are exactly the original Euler system \eqref{a0}.
			%If the straight line $\{x_2=x_2^0\}$ intersects with $\tilde S_1$ at a point $(\tilde{\Theta}(x_2^0), x_2^0)$, then it follows from the continuity equation and the slip boundary condition that 
			%	\begin{equation*}
				%		0=\int^{\tilde\Theta(x_2^0)}_{x_1^0}(\rho u_2)(s,x_2^0)\ ds.
				%	\end{equation*}
			% 	This contradicts \eqref{Ephorizontalvelocity}. 
			%If the straight line $\{x_2=x_2^0\}$ does not intersect with $\tilde S_1$, then it follows from the asymptotic behavior \eqref{upstreambehavior} that 
			%\[
			%0=\int_{-\infty}^{x_1^0}(\rho u_2)(s,x_2^0)\ ds +\int_{\bar{H}}^{x_2^0}\bar{\rho} \bar{u}(x_2)dx_2<0.
			%\]
			%This contradiction implies that a streamline inside the domain cannot touch the boundary $\tilde S_1$. 
		\end{proof}

		It follows from the continuity equation that there is a stream function $\psi$ satisfying
		\begin{align}\label{eq:psi_uv}
			\nabla \psi=(-\rho u_2, \rho u_1).
		\end{align}
		The Bernoulli's law (the second equation of \eqref{eq:euler}) shows that $\mathscr{B}(\rho,\mathbf{u})$ is conserved along each streamline. In the following proposition we show that $\mathscr{B}(\rho,\mathbf{u})$ can be expressed as a function of the stream function $\psi$ under the assumptions of Proposition \ref{Elemmaequivalent}.
		
		\begin{prop}\label{prop:Bernoulli}
			{Let a horizontal velocity $\bar u \in C^{1,1}([0,\bar H])$ satisfy \eqref{cond:u0_eps0} and a mass flux $Q>0$. 
				Suppose that $(\rho, \mathbf{u})$ is a solution to the Euler system \eqref{a0} and satisfies} the assumptions of Proposition \ref{Elemmaequivalent}. Then there is a function $\mathcal{B}:[0,Q]\rightarrow \R$, $\mathcal{B}\in C^{1,1}([0,Q])$ such that 
			\begin{align*}
				\mathscr{B}(\rho, \mathbf{u})=h(\rho)+\frac{|\nabla\psi|^2}{2\rho^2}=\mathcal{B}(\psi) \ \text{ in } \tilde\mcO.
			\end{align*}
			Denote
			\begin{equation}\label{eq:u0_eps0}
				\kappa_0:=\|{\bar u}'\|_{L^\infty([0, \bar H])}+\|{\bar u}''\|_{L^\infty([0,\bar H])}.
			\end{equation}
			Then 
			\begin{equation}\label{eq:u0_eps0_B}
				\|\mathcal{B}'\|_{L^\infty([0, Q])}\leq \frac{\kappa_0}{\bar\rho} \quad{and}\quad 
				\|\mathcal{B}''\|_{L^\infty([0,Q])}\leq \frac{\kappa_0}{\bar\rho^2\min_{x_2\in[0,\bar H]}\bar u(x_2)}, 
			\end{equation} 
			where
			\begin{equation}\label{eq:rhobar}
				\bar \rho =\frac{Q}{\int_0^{\bar H} \bar u(s) ds}.
			\end{equation}
		\end{prop}
		\begin{proof}
			In view of Proposition \ref{Elemmaequivalent} the Bernoulli function $\mathscr{B}(\rho,\mathbf{u})$ is conserved along each streamline, which is globally well-defined in $\tilde\mcO$. In particular, $\mathscr{B}(\rho, \mathbf{u})$ is uniquely determined by its value on the upstream. Let $\mfh(\psi;\bar\rho):[0,Q]\rightarrow [0,\bar H]$ be the position of the streamline at upstream where the stream function has the value $\psi$, i.e.,
			\begin{align}\label{eq:kappa}
				\psi=\bar\rho\int_0^{\mfh(\psi;\bar \rho)}\bar u(s)\ ds.
			\end{align}
			Note that $\mfh$ is well-defined as $\bar u>0$.  Here the upstream density $\bar\rho$ is uniquely determined by $Q$ and $\bar u$ from the upstream asymptotics \eqref{upstreambehavior}, i.e. $\bar \rho$ satisfies \eqref{eq:rhobar}. Since the Bernoulli function at the upstream is 
			\begin{equation}\label{def:B_upstream}
				B(x_2):=\lim_{x_1\rightarrow -\infty} \mathscr{B}(\rho,\mathbf{u})(x_1,x_2)=h(\bar\rho)+\frac{\bar u^2(x_2)}{2} \quad\text{for } x_2\in[0,\bar H],
			\end{equation}
			then we have
			\begin{align}\label{eq:psi}
				h(\rho)+\frac{|\nabla\psi|^2}{2\rho^2}=\mathcal{B}(\psi) \ \text{ in } \tilde\mcO,
			\end{align}
			where the function $\mathcal B$ is defined as
			\begin{equation}\label{defB}
				\mathcal{B}(z):=B(\mfh(z;\bar\rho))=h(\bar\rho)+\frac{\bar u^2(\mfh(z;\bar\rho))}{2}, \quad z\in [0,Q].
			\end{equation} 
			{The straightforward computations give}
			\begin{align}\label{eq:dB}
				\mathcal{B}'(z)=\frac{\bar u'(\mfh(z;\bar\rho))}{\bar\rho} \quad\text{and}\quad  \mathcal{B}''(z)=\frac{\bar u''(\mfh(z;\bar\rho))}{\bar u(\mfh(z;\bar\rho))\bar\rho^2}.
			\end{align}
			Hence \eqref{eq:u0_eps0_B} follows directly from \eqref{eq:dB} and the definition of $\kappa_0$ in \eqref{eq:u0_eps0}.
		\end{proof}
		For later purpose we extend the Bernoulli function $\mathcal{B}$ from $[0,Q]$ to $\R$ as follows: 
		firstly in view of \eqref{cond:u0_eps0}, $\bar u$ can be extended to a $C^{1,1}$ function defined on $\R$, which is still denoted by $\bar u$, such that 
		\begin{equation}\label{label_11}
		\bar u>0 \text{ on } \R, \quad {\bar u'= 0 \text{ on } (-\infty,0],} 
		\quad \bar u'\geq 0  \text{ on } [\bar H,\infty).
		\end{equation}
		Furthermore, the extension can be made such that 
		\begin{equation}\label{label_12}
		0<\bar u_\ast=\inf_{\R}\bar u\leq \bar u^*\leq \sup_{\R} \bar u =:\tilde u^\ast <\infty \quad\text{with} \quad \bar u_*:=\inf_{[0,\bar H]}\bar u, \quad \bar u^*:=\sup_{[0,\bar H]}\bar u,
		\end{equation}
		where $\tilde u^\ast$ depends on $\bar u^*$ and $\|\bar u\|_{C^{1,1}([0,\bar H])}$, and 
		\begin{equation}\label{label_13}
				\|\bar u\|_{C^{1,1}(\R)}\leq C\|\bar u\|_{C^{1,1}([0,\bar H])}.
		\end{equation}
		%$\xi(x_2)$ for $x_2>\bar H$, where $\xi(x_2)$ is a smooth decreasing function in $[\bar H,\infty)$ such that 
		%$$|\xi|\leq 1 \text{ in }[\bar H,\infty),\quad \xi=0 \text{ for } x_2\in [\bar H+1,\infty) \quad\text{and}\quad \xi'(\bar H)=\frac{\bar u'(\bar H)}{\bar u(\bar H)}.?$$
		Consequently, using \eqref{defB} one naturally gets an extension of $\mathcal B$ to a $C^{1,1}$ function in $\R$ (still denoted by $\mathcal{B}$), which satisfies  
		\begin{equation}\label{eq:sign_B}
			{\mathcal{B}'(z)= 0 \text{ on } (-\infty, 0]} 
			\quad \text{ and } \quad \mathcal{B}'(z)\geq 0 \text{ on } [Q,\infty). 
		\end{equation}
		The function $\mathcal{B}$ is bounded from above and below, i.e.,
		\begin{align}\label{eq:B}
			0<{B}_*\leq \mathcal{B}(z)\leq {B}^*<\infty,\quad z\in \R,
		\end{align}
		where $B_*$ and $B^*$ are defined as 
		\begin{equation}\label{defB*}
			B_*:=h(\bar\rho)+\frac12(\bar u_*)^2,\quad  B^*:=h(\bar\rho)+\frac12(\tilde u^*)^2
		\end{equation}
		with 
		\begin{equation*}\label{eq:B_derivative}
			%\bar \rho =\frac{Q}{\int_0^{\bar H} \bar u(s) ds}, \quad\text{and}\quad 
			\|\mathcal{B}'\|_{L^\infty(\R)}+\|\mathcal{B}''\|_{L^\infty(\R)}\leq C (\|\mathcal{B}'\|_{L^\infty([0, Q])}+\|\mathcal{B}''\|_{L^\infty([0, Q])}).
		\end{equation*}

		Before formulating the Euler system into a quasilinear PDE for the stream function $\psi$, let us digress for the study on the flow state with a given Bernoulli constant. 
		For the flow state with given Bernoulli constant $s$, the flow
		density $\rho$ and flow speed $q$ satisfy
		\begin{equation*}
			h(\rho)+\frac{q^2}{2}=s,
		\end{equation*}
		where we recall $h(\rho)=\frac{\rho^{\gamma-1}}{\gamma-1}$ is the enthalpy of the flow.
		Therefore, the speed $q$ satisfies
		\begin{equation*}
			q=\mfq(\rho,s)= \sqrt{2(s-h(\rho))}.
		\end{equation*}
		For each {fixed $s$}, we denote the critical density and the maximum density by
		\begin{equation}\label{defrhoc}
			\varrho_c(s):=\left\{\frac{2(\gamma-1)}{\gamma+1} s\right\}^{\frac{1}{\gamma-1}}\quad\text{and}\quad   \varrho^*(s) :=\left\{(\gamma-1)s\right\}^{\frac{1}{\gamma-1}},
		\end{equation}
		respectively. For states with given Bernoulli constant $s$, note that {$s-h(\rho)\geq 0$ for $\rho\leq \varrho^\ast(s)$. Thus the flow state $q=\mfq(\rho,s)$ is well-defined when $\rho\leq \varrho^\ast(s)$.} The flow is subsonic (i.e. $\mfq(\rho,s)<c(\rho)$, where $c(\rho)$ is the sound speed defined in \eqref{eq:sound}) if and only if $\varrho_c(s)<\rho\leq \varrho^\ast(s)$. At the critical density one has $\mfq(\varrho_c(s),s)= c(\varrho_c(s))$. We denote the square of the momentum and {the square of}  critical momentum by 
		\begin{equation}\label{defF}
			\mft(\rho,s):=\rho^2 \mfq(\rho,s)^2=2\rho^2(s-h(\rho)) \ \text{ and }\ \mft_c(s):=\mft(\varrho_c(s),s)=\left\{\frac{2(\gamma-1)}{\gamma+1} s\right\}^{\frac{\gamma+1}{\gamma-1}}.
		\end{equation}
		
		\begin{rmk}\label{rmk:Q}
			%Let $Q$ be the mass flux and $\bar u$ be the horizonal velocity at the upstream as in Proposition \ref{prop:Bernoulli}. 
			Suppose that the flow $(\rho, \mathbf{u})$ has the mass flux $Q$ and satisfies the asymptotics \eqref{upstreambehavior} at the upstream with a constant density $\bar \rho$. Then necessarily $\bar \rho$ must be defined by \eqref{eq:rhobar}.  
			The flow is subsonic at the upstream if $\tilde  u^* < c(\bar \rho)=\bar\rho^{\frac{\gamma-1}{2}}$, where $\tilde u^*$ is defined in \eqref{label_12}. This can be guaranteed by letting 
			\begin{equation}\label{def:Q_*}
			Q>\tilde Q,\quad \text{where}\quad
			\tilde Q:=(\tilde u^*)^{\frac2{\gamma-1}}\int_0^{\bar H}\bar u(s)\ ds\geq Q_*:=(\bar u^*)^{\frac2{\gamma-1}}\int_0^{\bar H}\bar u(s)\ ds.
			\end{equation}
		   %From now on we assume the mass flux $Q$ satisfies \eqref{def:Q_*} unless otherwise stated.

			An immediate consequence of \eqref{def:Q_*} is that the upper and lower bounds of the Bernoulli function $\mathcal{B}$ are comparable, i.e. 
			\begin{equation}\label{def:t_bound}
			B_\ast \leq \mathcal{B}(z)\leq B^\ast \leq \frac{\gamma+1}{2}B_\ast,
			\end{equation}
			where $B_\ast$ and $B^\ast$ are defined in \eqref{defB*}. Furthermore, $B_\ast$ (thus $B^\ast$) is comparable to $\bar\rho^{\gamma-1}$, i.e.
			\begin{equation}\label{eq:rhobar_B_*}
				(\gamma-1)^{-\frac{1}{\gamma-1}}\bar \rho \leq B_\ast^{\frac{1}{\gamma-1}}\leq \left(\frac{\gamma+1}{2(\gamma-1)}\right)^{\frac{1}{\gamma-1}}\bar\rho.
			\end{equation}
			As $Q=\bar\rho \|\bar u\|_{L^1([0,\bar{H}])}$, the above inequality can also be reformulated as 
			\begin{align}\label{label_7}
				(\gamma-1)^{-\frac{1}{\gamma-1}}\frac{Q}{\|\bar u\|_{L^1([0,\bar H])}}\leq B_\ast^{\frac{1}{\gamma-1}}\leq\left(\frac{\gamma+1}{2(\gamma-1)}\right)^{\frac{1}{\gamma-1}}\frac{Q}{\|\bar u\|_{L^1([0,\bar H])}}.
			\end{align}
		\end{rmk}
		
		Now we have the following lemma on the representation of density $\rho$ in terms of the stream function $\psi$ in the subsonic region.
		\begin{lem}\label{lem:density}
			Suppose the density function $\rho$ and the stream function $\psi$ satisfy the Bernoulli's law \eqref{eq:psi}.
			Then the following statements hold. 
			\begin{itemize}
				\item [(i)] The density function $\rho$ can be expressed  as a function of $|\nabla \psi|^2$ and $\psi$ in the subsonic region, i.e.
				\begin{equation}\label{eq:rho}
					\rho=\frac{1}{g(|\nabla \psi|^2, \psi)},\quad  \text{if } \rho\in (\varrho_c(\mcB(\psi)), {\varrho^*(\mcB(\psi))]},
				\end{equation}
				where $\varrho_c$ and $\varrho^*$ are functions defined in \eqref{defrhoc}, 
				and 
				$$g:\{(t,z)| 0\leq t<\mft_c(\mathcal{B}(z)), \ z\in \R\}\rightarrow \R$$
				{is a function  smooth in $t$ and $C^{1,1}$ in $z$ with $\mft_c$ defined in \eqref{defF}.} Furthermore, 
				\begin{align}\label{eq:upper_lower_g}
					\frac{1}{\varrho^*(B^*)}=:g_*\leq g(t,z)\leq g^*:= \frac{1}{\varrho_c(B_*)}.
				\end{align}
				\item [(ii)] The function $g$ satisfies the identity
				\begin{equation}\label{eq:claim1_g}
					g(t,z)^2\p_z g(t,z)= -2\mathcal{B}'(z)\p_t g(t,z), \quad t\in [0, \mft_c(\mathcal{B}(z))), \ z\in \R.
				\end{equation}
			\end{itemize}
			
			%Furthermore, there are positive constants $g_\ast, g^\ast$ depending on $\bar\rho$ and $\gamma$ such that $g_\ast \leq g(|\nabla\psi|^2,\psi)\leq g^\ast$.
		\end{lem}
		\begin{proof}
			%	Set
			%	\begin{equation}\label{defF}
				%		\mathcal{F}(\varrho, z):=2\varrho^2(\mathcal{B}(z)-h(\varrho)).
				%	\end{equation}
			\emph{(i)}.	From the expression \eqref{defF}, {the straightforward computations give}
			\begin{align}\label{eq:dF}
				\p_\rho \mft(\rho,s)=4\rho\left(s-\frac{\gamma+1}{2}h(\rho)\right).
			\end{align}
			Now, one can see that with $\varrho_c$ and $\varrho^*$ defined in \eqref{defrhoc}, the following statements hold:
			\begin{itemize}
				\item[(a)] $\rho\mapsto\mft(\rho,s)$ achieves its maximum $\mft_c(s)$ at $\rho=\varrho_{c}(s)$;
				%\item[(b)] $\mft(\rho,s)> 0$ if and only if $0< \rho< \varrho^\ast (s)$;
				\item[(b)] $\p_\rho\mft(\rho,s)<0$ when $\varrho_{c}(s)< \rho < \varrho^\ast(s)$.
			\end{itemize}
			Thus by the inverse function theorem, for each fixed $s>0$ and $\rho\in (\varrho_c(s), \varrho^\ast (s)]$, one can express $\rho$ as a function of $t:= \mft(\rho,s)\in [0,\mft_c(s))$ and $s$, i.e. $\rho=\rho(t,s)$. Let
			\begin{equation} \label{eq:branch}
				g(t,z):=\frac{1}{\rho(t,\mathcal{B}(z))}, \quad t\in [0, \mft_c(\mcB(z)))
			\end{equation}
			where $\mathcal{B}$ is the function defined in \eqref{defB}. The function $g$ is smooth in $t$ by the inverse function theorem and $C^{1,1}$ in $z$ by the $C^{1,1}$ regularity of $\mathcal{B}$ and the smooth dependence of $\rho$ on $s$. In view of Bernoulli's law \eqref{eq:psi} we have \eqref{eq:rho}.
			Consequently \eqref{eq:upper_lower_g} holds. 
			
			\emph{(ii)}. Let $\varrho(t,z):=\rho(t,\mathcal{B}(z))$.  Then from \eqref{defF} and Bernoulli's law \eqref{eq:psi} one has
			\begin{align}\label{eq:claim}
				t= 2\varrho(t,z)^2 (\mathcal{B}(z)-h(\varrho(t,z))).
			\end{align}
			Differentiating \eqref{eq:claim} with respect to $t$ and $z$ yields
			\begin{align}\label{eq:drhot}
				\frac{1}{2}=\frac{\p_t\varrho}{\varrho}\left(2\varrho^2\mathcal{B}(z)-2\varrho^2h(\varrho)-\varrho^3h'(\varrho)\right)\stackrel{\eqref{eq:claim}}{=}\frac{\p_t \varrho}{\varrho}\left(t-\varrho^{\gamma+1}\right)
			\end{align}
			and
			\begin{align}\label{eq40-1}
				-\varrho^2\mathcal{B}'(z)=\frac{\p_z \varrho}{\varrho}\left(2\varrho^2\mathcal{B}(z)-2\varrho^2h(\varrho)-\varrho^3h'(\varrho)\right)\stackrel{\eqref{eq:claim}}{=}\frac{\p_z\varrho}{\varrho}\left(t-\varrho^{\gamma+1}\right),
			\end{align}
			respectively.
			Combining  \eqref{eq:drhot} and \eqref{eq40-1} one has  
			\[
			\p_z \varrho(t,z)=-2\varrho(t,z)^2\p_t\varrho(t,z) \mathcal{B}'(z).
			\]
			This together with \eqref{eq:branch} yields \eqref{eq:claim1_g}.
		\end{proof}
		
		Finally, we derive the equation of the stream function in the subsonic region. 
		\begin{lem}\label{lem:stream_equiv}
			Let $(\rho,\bu)$ be a solution to the Euler system \eqref{eq:euler}. 
			Assume $(\rho, \bu)$ satisfies the assumptions in Proposition \ref{Elemmaequivalent}. Then in the subsonic region $|\nabla\psi|^2<\mft_c(\mathcal{B}(\psi))$, the stream function $\psi$ solves
			\begin{align}\label{eq:psi_main}
				\nabla\cdot \left(g(|\nabla \psi|^2, \psi)\nabla \psi \right)=\frac{\mathcal{B}'(\psi)}{g(|\nabla\psi|^2,\psi)},
			\end{align}
			where $g$ is defined in \eqref{eq:branch} and $\mathcal{B}$ is the Bernoulli function {%at upstream 
				defined in \eqref{defB}.}
			The equation \eqref{eq:psi_main} is elliptic if and only if $|\nabla\psi|^2< \mft_c(\mcB(\psi))$.
		\end{lem}
		\begin{proof}
			Let $X(s;x)$ be the streamlines satisfying
			\begin{equation*}
				\left\{
				\begin{aligned}
					& \frac{dX}{ds} =\mathbf{u}(X(s;x)),\\
					& X(0; x)=x.
				\end{aligned}
				\right.
			\end{equation*}
			It follows from the third equation in \eqref{eq:euler} that $\omega/\rho$ is a constant along each streamline. Hence $\omega/\rho$ can be determined by the associated data in the upstream as long as the streamlines of the flows have simple topological structure, which is guaranteed by the assumption $u_2<0$. If $x\in \{\psi=z\}$, then
			\begin{align*}
				\frac{\omega}{\rho}(x)=\lim_{s \rightarrow -\infty}\frac{\omega}{\rho}(X(s;x)) =-\frac{\bar u'(\mfh(z;\bar \rho))}{\bar\rho}.
			\end{align*}
			Expressing the vorticity $\omega$ in terms of the stream function $\psi$ and using \eqref{eq:dB} one has
			\begin{align}\label{label_9}
				-\nabla\cdot \left(\frac{\nabla \psi}{\rho} \right)=\omega = -\mathcal{B}'(\psi)\rho.
			\end{align}
			In view of \eqref{eq:rho} the above equation can be rewritten into \eqref{eq:psi_main}.

			The equation \eqref{eq:psi_main} can be written in the nondivergence form as follows
			\begin{equation*}
				a^{ij}(\nabla\psi,\psi) \p_{ij}\psi + \p_z g(|\nabla \psi|^2, \psi) |\nabla\psi|^2=  
				\frac{\mathcal{B}'(\psi)}{g(|\nabla \psi|^2, \psi)},
			\end{equation*}
			where the matrix
			$$(a^{ij})=g(|\nabla \psi|^2, \psi)I_2+2\p_t g(|\nabla \psi|^2, \psi)\nabla\psi\otimes \nabla \psi$$
			is symmetric with the eigenvalues
			\begin{align*}
				\beta_0  =g(|\nabla \psi|^2, \psi)\quad \text{and}\quad
				\beta_1 =g(|\nabla \psi|^2, \psi)+2\p_t g(|\nabla \psi|^2, \psi)|\nabla\psi|^2.
			\end{align*}
			%It follows from \eqref{eq:rho} and \eqref{eq:B}  that   if $t\in {[0, \mft_c(\mcB(z)))}$ with $\mft_c$ defined in \eqref{defF}, then
			%	\begin{align}\label{eq:upper_lower_g}
				%		\frac{1}{\varrho^*(B^*)}=:g_*\leq g(t,z)\leq g^*:= \frac{1}{\varrho_c(B_*)}.
				%\end{align}
				Differentiating the identity $t=\mft(\frac{1}{g(t,z)}, \mathcal{B}(z))$ gives
				\begin{equation}\label{eq:dt_g}
					\p_t g(t,z)=-\frac{g^2(t,z)}{\p_\rho \mft(\frac{1}{g(t,z)}, \mathcal{B}(z))}\quad\text{for}\quad {0\leq t} <\mft_c(\mcB(z)).
				\end{equation}
				Therefore, one has
				\begin{equation}\label{eq:ptg}
					\p_t g(t,z)\geq 0 \quad \text{and}\,\, \lim_{t \rightarrow\mft_c(\mcB(z))-}\p_t g(t,z)=+\infty.
				\end{equation}
				Thus {$\beta_0$ has uniform upper and lower bounds% depending only on $\gamma$ and $B_\ast$
				}, and $\beta_1$ has a uniform lower bound but blows up when $|\nabla\psi|^2$ approaches $\mft_c(\mcB(\psi))$. Therefore, the equation \eqref{eq:psi_main} is elliptic as long as $|\nabla \psi|^2 < \mft_c(\mcB(\psi))$, and  is singular when $|\nabla \psi|^2=\mft_c(\mcB(\psi))$. This completes the proof of the lemma.
			\end{proof}
			
			\subsection{Reformulation for the jet flows in terms of the stream function}\label{sec:setup}
			%We assume the nozzle $\mcN$ is symmetric about the $x$-axis and with height $H_0$. Assume the incoming flow has the horizontal velocity $\bar u$ and constant density $\bar \rho$. Let $Q=\bar\rho \int_0^{H_0}\bar u(s) ds$ be the incoming mass.
			Let
			%\begin{equation}\label{fluxBC}
			%Q:=\int_0^{\bar H} \bar\rho \bar u(s) \ ds,
			%\end{equation}
			%be the total flux of the flow and
			\begin{equation}\label{defLambda}
				\Lambda:=
				{\rho_e\sqrt{2B(\barH)-2h(\rho_e)}}
			\end{equation}
			be the constant momentum on the free boundary with $\rho_e :=(\gamma p_e)^{1/\gamma}$, where {$B$ is defined in \eqref{def:B_upstream} and} $p_e$ is the pressure on the free boundary as in Problem \ref{pb}. 
			%Assume {that} \eqref{upstreambehavior} and \eqref{Ephorizontalvelocity} hold. Then Problem \ref{pb} can be solved as long as the following problem in terms of the stream function $\psi$ is solved.
			From the previous analysis, {Problem \ref{pb} can be solved as long as the following
			problem in terms of the stream function is solved.}
			%if $(\rho,u)\in C^1(\mathcal O)\times C^{1,1}(\mathcal O)$ solves Problem \ref{pb}, where $\mathcal O$ is the flow region, then $\psi$ determined by \eqref{eq:psi_uv} solves the following Problem \ref{Pb2}.
			
			\begin{problem}\label{Pb2}
				Given {a mass flux $Q>0$ and an incoming horizontal velocity $\bar u\in C^{1,1}([0,\bar H])$ satisfying \eqref{cond:u0_eps0}.} 
				%, where $\tilde Q$ is defined in \eqref{def:Q_*}
				One looks for a triple $(\psi, \Gamma_\psi, \Lambda)$ satisfying 
				{$\psi\in C^{2,\alpha}(\{\psi<Q\})\cap C^{1}(\overline{\{\psi<Q\}})$ for any  $\alpha\in(0,1)$,} 
				$\partial_{x_1}\psi>0$ in $\{\psi<Q\}$ and 
				\begin{equation}\label{eq_pb2}
					\left\{
					\begin{aligned}
						&\nabla\cdot\left(g(|\nabla \psi|^2,\psi)\nabla \psi\right)=\frac{\mathcal{B}'(\psi)}{ g(|\nabla\psi|^2,\psi)} &&\text{ in } \{\psi<Q\},\\
						&\psi =0 &&\text{ on } S_0,\\
						&\psi =Q &&\text{ on } S_1 \cup \Gamma_\psi,\\
						&|\nabla \psi| =\Lambda &&\text{ on } \Gamma_\psi,
						%|\nabla \psi|^2 &<\mft_c(\mcB(\psi)) \text{ on } \{\psi<Q\},
					\end{aligned}
					\right.
				\end{equation}
				where the free boundary $\Gamma_\psi:=\p\{\psi<Q\}\backslash S_1$.  Furthermore, the free boundary $\Gamma_\psi$ and the flow 
				%$(\rho, \mathbf{u})=(\frac{1}{g(|\nabla \psi|^2, \psi)}, g(|\nabla \psi|^2, \psi)\nabla^\perp\psi)$ 
				\begin{align}\label{def:rhou_psi}
					(\rho,\mathbf u)=\left(\frac{1}{g(|\nabla \psi|^2, \psi)},\, g(|\nabla \psi|^2, \psi)\partial_{x_2} \psi,\, - g(|\nabla \psi|^2, \psi)\partial_{x_1}\psi\right)	
				\end{align}
				 are expected to satisfy the following properties.
				\begin{enumerate}
					\item The flow is subsonic in the flow region, i.e. $|\nabla\psi|^2<\mft_c(\mathcal{B}(\psi))$ in $\{\psi<Q\}$, where $\mft_c$ is defined in \eqref{defF} and the Bernoulli function $\mathcal{B}$ is defined in \eqref{defB}.
					{\item The flow satisfies the asymptotic behavior \eqref{upstreambehavior} at upstream $x_1\to-\infty$, where the upstream density $\bar\rho$ is given by \eqref{eq:rhobar}.
					\item The free boundary $\Gamma_\psi$ is given by a graph $x_1=\Upsilon(x_2)$, $x_2\in (\ubar{H}, 1]$ for some $\ubar{H}\in (0,1)$, where the function $\Upsilon$ is $C^{2,\alpha}$ for any $\alpha\in(0,1)$.}
					\item The free boundary $\Gamma_\psi$ fits the nozzle at $A=(0,1)$ continuous differentiably, i.e. $\Upsilon(1)=\Theta(1)$ and $\Upsilon'(1)=\Theta'(1)$.
					\item  For $x_1$ sufficiently large, the free boundary is also an $x_1$-graph, i.e., it can be written as $x_2=f(x_1)$ for some {$C^{2,\alpha}$} function $f$. Furthermore, at downstream $x_1\rightarrow \infty$, one has
					\[
					\lim_{x_1\rightarrow \infty}f(x_1)=\unH,\quad  \lim_{x_1\rightarrow \infty}f'(x_1)=0,\quad\text{and}\quad \lim_{x_1\rightarrow \infty}(\rho,\mathbf{u})= (\unrho,\unu,0)
					\]
					for some {positive} constant $\unrho$ and positive function $\unu=\unu(x_2)$.
				\end{enumerate}
			\end{problem}
	
			In fact, we have the following crucial observation, which asserts that $(\rho,\mathbf u)$ defined in \eqref{def:rhou_psi} is a solution of the Euler system even if $(\rho,\mathbf u)\in C^{1,\alpha}(\{\psi<Q\})$.
			\begin{prop}\label{prop:equiv_sol}
				{Given a mass flux $Q>0$} and an incoming horizontal velocity $\bar u\in C^{1,1}([0,\bar H])$ satisfying \eqref{cond:u0_eps0}.
				Assume that a function $\psi\in C^{2,\alpha}(\{\psi<Q\})\cap C^{1}(\overline{\{\psi<Q\}})$ is a {subsonic} solution of \eqref{eq_pb2}, where $\alpha\in(0,1)$. Then the flow $(\rho,\mathbf u)$ defined in \eqref{def:rhou_psi} solves the Euler system \eqref{a0} in the flow region $\{\psi<Q\}$.
			\end{prop}
			\begin{proof}
				Due to the regularity of the function $\psi$,  the functions $\rho$ and $\mathbf u:=(u_1,u_2)$ defined in \eqref{def:rhou_psi} belong to $C^{1,\alpha}(\{\psi<Q\})\cap C(\overline{\{\psi<Q\}})$. Since $(\rho,\mathbf u)$ satisfies \eqref{eq:psi_uv}, it follows from %\eqref{eq:psi}
				the definition of the function $g$ in Lemma \ref{lem:density} (i) that in the flow region
				\begin{align*}\label{label_8}
					\frac{u_1^2+u_2^2}{2}+h(\rho)=	\frac{|\nabla\psi|^2}{2\rho^2}+h(\rho)=\mathcal{B}(\psi).
				\end{align*}
				Differentiating the above equality with respect to $x_1$ and $x_2$, respectively, one gets
				\begin{equation}\label{eq1}
					u_1\partial_{x_1}u_1+u_{2}\partial_{x_1}u_2 +\p_{x_1}h(\rho)=\mcB'(\psi)\partial_{x_1}\psi,
				\end{equation}
				and
				\begin{align}\label{eq2}
					u_1\partial_{x_2}u_1+u_2\partial_{x_2}u_{2}+\p_{x_2}h(\rho)=\mcB'(\psi)\partial_{x_2}\psi.
				\end{align}
				Note that by the quasilinear equation in \eqref{eq_pb2} one has  $\mathcal{B}'(\psi)=-(\p_{x_1}u_2-\p_{x_2}u_1)/\rho$. Thus the equalities \eqref{eq1} and \eqref{eq2} together with \eqref{eq:psi_uv} yield \eqref{label_10}. Note that $(\rho,\mathbf u)$ satisfies the continuity equation (the first equation in \eqref{a0}). Then the Euler system \eqref{a0} follows from \eqref{label_10} and the continuity equation.
				This finishes the proof of the proposition.
				%Thus even though $\rho$ and $\mathbf u$ are not $C^{1,1}$ functions as required in Proposition \ref{Elemmaequivalent}, we can still derive the original Euler system from the stream function formulation.
			\end{proof}
			
			\subsection{Subsonic truncation}\label{sec:truncation}
			One of the major difficulties to solve the equation \eqref{eq:psi_main} is that \eqref{eq:psi_main} becomes degenerate as the flows approach the sonic state. As in \cite{ACF85} our strategy here is to use a subsonic truncation so that after the truncation the equation \eqref{eq:psi_main} is always elliptic.
			
			Let $\varpi:\R\rightarrow [0,1]$ be a smooth nonincreasing {function such that
				\begin{equation*}
					\varpi(s)=\left\{
					\begin{aligned}
						&1 \quad \text{if}\,\, s\leq-1,\\
						& 0\quad \text{if}\,\, s\geq-1/2,
					\end{aligned}
					\right.\quad \text{ and }\quad  |\varpi'|
					+|\varpi''|\leq 8.
				\end{equation*}
				For $\epsilon\in (0,1/4)$, 
				let $\varpi_\epsilon(s):=\varpi((s-1)/\epsilon)$. We define $g_\epsilon:[0,\infty)\times \R\rightarrow \R$,
				\begin{equation}\label{eq:truncation_g}
					g_\epsilon(t,z):= g(t, z)\varpi_\epsilon(t/\mft_c(\mcB(z)))+ (1-\varpi_\epsilon(t/\mft_c(\mcB(z))))g^* ,
				\end{equation}
				where} $\mft_c$ is defined in \eqref{defF} and $g^*$ is the upper bound for $g$ in \eqref{eq:upper_lower_g}. The properties of $g_\epsilon$ are summarized in the following lemma.
			
			\begin{lem}\label{lem:truncation_g}
				Let $g$ be the function defined in Lemma \ref{lem:density}, and let $g_\epsilon$ be the subsonic truncation of $g$ defined in \eqref{eq:truncation_g}. Then the function $g_\epsilon(t, z)$ is smooth with respect to $t$ and {$C^{1,1}$} with respect to $z$. Furthermore, under the assumption \eqref{def:Q_*}, there exist  positive constants $c_\ast$ and $c^\ast$ depending only on $\gamma$, such that for all $(t,z)\in [0,\infty)\times \R$
				\begin{align}
					c_\ast B_*^{-\frac{1}{\gamma-1}} &\leq g_\epsilon(t,z)\leq c^\ast B_*^{-\frac{1}{\gamma-1}},\,\label{eq:g}\\
					%				0 <\p_t g_\epsilon(t,z) &\leq \frac{c^\ast }{\epsilon }B_\ast^{-\frac{2\gamma}{\gamma+1}} \quad \forall 0\leq t<\mft_c(B^\ast), \quad \p_t g_\epsilon (t,z)=0 \quad \forall t\geq \mft_c(B^\ast)\label{eq:deri_g}\\
					c_\ast B_*^{-\frac{1}{\gamma-1}}&\leq g_\epsilon(t, z)+2\partial_t g_\epsilon(t, z) t\leq c^\ast \epsilon^{-1}B_*^{-\frac{1}{\gamma-1}},\label{beta1eps}\\
					|\p_zg_\epsilon(t,z)|&\leq c^\ast\kappa_0\epsilon^{-1}B_\ast^{-\frac{\gamma+1}{\gamma-1}},\quad  |t\p_z g_\epsilon(t,z)|\leq c^\ast \kappa_0\epsilon^{-1}.
					\label{eq:dzg}
				\end{align}
				%\item [(ii)] If $0\leq t\leq(1-\epsilon)\mft_c(\mcB(z))$,  then
				%\begin{align}\label{eq:claim1}
				%	\p_z g_\epsilon(t,z)=\p_z g(t,z)= -2\frac{\mathcal{B}'(z)\p_t g(t,z)}{ g(t,z)^2}.
				%\end{align}
				%\item[(ii)] There exists $C>0$ depending only on $B_\ast$ and $\gamma$ such that 
				%\begin{align}
				%	|t\p_z g_\epsilon(t,z)|\leq  \frac{c^\ast}{\epsilon}, \quad \forall t\geq 0.
				%		\label{eq:dzg}
				%	\end{align}
			%	In particular, we have $\p_z g_\epsilon(t,z)=0$ for all $t\geq\mcT^*$.
			%	\begin{align}
				%	&|\p_{zz} g_\epsilon(t,z)|\leq C\epsilon^{-2}(|\mathcal{B}'(z)|^2+|\mathcal B''(z)|)\label{eq:dzzg};
				%	|\p_z g_\epsilon(t,z)|=0, \quad  \forall t\geq\mcT^*. \label{eq:dgz0}
				%\end{align}
			\end{lem}
			\begin{proof}
				It follows from Lemma \ref{lem:density} and the definition of $\varpi_\epsilon$ that $g_\epsilon$ is smooth with respect to $t$ and {$C^{1,1}$} with respect to $z$.
				
				\emph{(i)}. Clearly, \eqref{eq:g} follows directly from the upper and lower bound for $g$ in \eqref{eq:upper_lower_g} and {the} definition of $g_\epsilon$ in \eqref{eq:truncation_g}.
				
				To show \eqref{beta1eps},  we first claim that if $0\leq t\leq\left(1-\frac{\epsilon}2\right) \mft_c(\mcB(z))$ then \begin{equation}\label{eq:g_dt_bound}
					0<\p_t g(t,z)\leq \frac{C_\gamma  g(t,z)}{\epsilon \mft_c(B_\ast)},
				\end{equation}
				{where $C_\gamma$ represents a positive constant depending only on $\gamma$, and $\mft_c$ is defined in \eqref{defF}.} In view of the expression of $\p_t g(t,z)$ in \eqref{eq:dt_g} it suffices to estimate $\p_{\rho}\mft(\rho,s)$ at $\rho=\rho(t,s)$ for $0\leq t\leq\left(1-\frac{\epsilon}2\right) \mft_c(s)$. In fact, it follows from \eqref{defrhoc} that one has  $s=\frac{\gamma+1}{2}h(\varrho_c(s))$. Hence for $0\leq t\leq\left(1-\frac{\epsilon}2\right) \mft_c(s)$  and $\rho\geq \varrho_c(s)$ it holds that
				\begin{align*}
					\partial_{\rho}\mft(\rho, s)&\stackrel{\eqref{eq:dF}}{=}2(\gamma+1)\rho\left(h(\varrho_c(s))-h(\rho)\right)\stackrel{\eqref{defF}}{=}2(\gamma+1)\rho\left(\frac{
						\mft(\rho,s)}{2\rho^2}-\frac{\mft_c(s)}{2\varrho_c(s)^2}\right)\\
					&\leq (\gamma+1)\frac{\mft(\rho,s)-\mft_c(s)}{\rho}\leq -\frac{(\gamma+1)\epsilon \mft_c(s)}{2\rho}.
				\end{align*}
				Thus from \eqref{eq:dt_g} and \eqref{def:t_bound} we conclude that 
				\begin{equation*}
					0<\p_tg(t,z)=-\frac{g^2(t,z)}{\p_\rho\mft(\rho, s)}\Big|_{\rho=\frac{1}{g(t,z)}, s=\mathcal{B}(z)}\leq \frac{2g(t,z)}{(\gamma+1)\epsilon\mft_c(\mathcal{B}(z))}\leq \frac{C_\gamma  g(t,z)}{\epsilon \mft_c(B_\ast)},
				\end{equation*}
				which is the claimed inequality \eqref{eq:g_dt_bound}. 
				With \eqref{eq:g_dt_bound} at hand, we thus have
				\begin{equation}\label{eq:g_dt}
					\begin{split}
						0\leq\p_tg_\epsilon(t,z) &= \partial_t g (t, z) \varpi_\epsilon\left(\frac{t}{\mft_c(\mcB(z))}\right)+(g(t, z)-g^*) \varpi'_\epsilon\left(\frac{t}{\mft_c(\mcB(z))}\right)\frac{1}{\mft_c(\mcB(z))}\\
						&\leq \frac{C_\gamma  g(t,z)}{\epsilon \mft_c(B_\ast)}+\frac{8g^\ast}{\epsilon \mft_c(B_\ast)} \leq \frac{C_\gamma g^\ast}{\epsilon \mft_c(B_\ast)}.
					\end{split}
				\end{equation}
				Note that $\p_tg_\epsilon(t,z)$ is supported on $\{(t,z)| 0\leq t\leq (1-\epsilon/2)\mft_c(\mathcal{B}(z)), z\in \R\}$. Thus
				\begin{align*}
					0\leq t\p_tg_\epsilon(t,z)\leq \frac{C_\gamma g^\ast \mft_c(\mcB(z))}{\epsilon \mft_c(B_\ast)}\leq \frac{C_\gamma g^\ast \mft_c(B^\ast)}{\epsilon \mft_c(B_\ast)}.
				\end{align*}
				%the upper and lower bound of $g$ and $\mft_c(\mcB(z))$ (cf. \eqref{eq:upper_lower_g} and \eqref{def:t_bound}) that  \eqref{eq:deri_g} holds true. Note that $\p_tg_\epsilon(t,z)$ is supported on $\{(t,z): 0\leq t\leq (1-\epsilon/2)\mft_c(\mathcal{B}(z)), z\in \R\}$. Thus letting $\mcT^\ast:= \sup_{z\geq 0}\mft_c(\mcB(z))$, which depends only on $B_\ast$ and $\gamma$, we have that $\p_tg_\epsilon(t,z)=0$ for $t\geq \mcT^\ast$. 
				Since $B_\ast$ and $B^\ast$ are equivalent, cf. \eqref{def:t_bound}, then the estimate \eqref{beta1eps} follows directly from  \eqref{eq:g}.
				
				{\emph{(ii).}  Direct computations yield
					\begin{equation}\label{gm_dz}
						\begin{aligned}
							&\p_z g_\epsilon (t,z)= \p_z g(t,z)\varpi_\epsilon\left(\frac{t}{\mft_c(\mcB(z))}\right) + (g^\ast-g(t,z))\varpi_\epsilon'\left(\frac{t}{\mft_c(\mcB(z))}\right) \frac{t\mft_c'(\mcB(z))\mcB'(z)}{\mft_c(\mcB(z))^2}\\
							\stackrel{\eqref{eq:claim1_g}}{=} & \left\{ -2\frac{\p_t g(t,z)}{ g(t,z)^2} \varpi_\epsilon\left(\frac{t}{\mft_c(\mcB(z))}\right)+ (g^\ast-g(t,z))\varpi_\epsilon'\left(\frac{t}{\mft_c(\mcB(z))}\right) \frac{t\mft_c'(\mcB(z))}{\mft_c(\mcB(z))^2}\right\}\mcB'(z).
						\end{aligned}
					\end{equation}
					
					Note that both sums in the expression of $\p_tg_\epsilon$, cf. \eqref{eq:g_dt}, are nonnegative. Thus we get from  the expression of $\p_zg_\epsilon$ in \eqref{gm_dz} that  
					\begin{equation}\label{eq:dzgm_dtgm}\begin{split}
							|\p_zg_\epsilon(t,z)|\leq& \max\left\{\frac{2}{g^2(t,z)},\frac{t\mathfrak t_c'(\mathcal B(z))}{\mathfrak t_c(\mathcal B(z))}\right\}|\mathcal B'(z)| \p_tg_\epsilon(t,z).
							%\leq\frac{C_\gamma\kappa_0}{\epsilon}B_*^{-\frac{\gamma+1}{\gamma-1}}.
						\end{split}
					\end{equation}
					Note  from \eqref{eq:u0_eps0_B} and \eqref{eq:rhobar_B_*} that  
					\begin{equation}\label{eq:Bdz}
						|\mathcal{B}'(z)|%\leq \frac{\kappa_0}{\bar \rho}
						\leq C_\gamma\kappa_0B_\ast^{-\frac{1}{\gamma-1}}.
					\end{equation}
					This together with the support condition of $\p_tg_\epsilon$,
					\begin{align}\label{lable_2}
						\mft_c(\mcB(z))\sim B_\ast^{\frac{\gamma+1}{\gamma-1}},\quad \mft'_c(\mcB(z))\sim B_\ast^{\frac{2}{\gamma-1}}
					\end{align}
					as well as \eqref{eq:g_dt} and \eqref{eq:g} yields the first inequality in \eqref{eq:dzg}. 
					To show the second inequality in \eqref{eq:dzg}, first we note that from \eqref{eq:dzgm_dtgm} and the support of $\p_tg_\epsilon$ one has
					\begin{equation}\label{label_1}
						\p_zg_\epsilon(t,z)=0 \quad\text{for } t\geq \mft_c(\mathcal{B}(z)).
					\end{equation}
					Then it follows from the first inequality in \eqref{eq:dzg} that 
					%\begin{align*}
					%\left|t\p_zg_\epsilon(t,z)\right|\leq C_\gamma \left|\frac{t\p_tg(t,z)}{g(t,z)^2} + \frac{g^\ast \mft'_c(\mathcal{B}(z))}{\epsilon}\right||\mathcal{B}'(z)|\leq \frac{C_\gamma B_\ast^{\frac{1}{\gamma-1}}}{\epsilon}|\mathcal{B}'(z)|.
					%\end{align*}
					\begin{align*}
						\left|t\p_zg_\epsilon(t,z)\right|\leq \mft_c(\mathcal{B}(z))|\p_zg_\epsilon(t,z)|\leq \frac{C_\gamma\kappa_0}{\epsilon}.
					\end{align*}
					This finishes the proof of the lemma.}
			\end{proof}
			
			\section{Variational formulation for the free boundary problem}\label{secvar}
			\label{sec:variational}
			One of the key observations in this paper is that the equation \eqref{eq:psi_main} can be written as the Euler-Lagrange equation for a Lagrangian functional,
			and that the jet problem is equivalent to a domain variation problem. 
			
			{In Sections \ref{secvar}--\ref{seccont}, we will always assume that the incoming mass flux $Q$ satisfies  $Q>\tilde Q$, where $\tilde Q$ is defined in \eqref{def:Q_*}. }
			%We will show by the first variation that the minimizers solve a uniformly elliptic Bernoulli type free boundary problem, in particular, it solves the free boundary problem considered in Section \ref{sec:setup} away from the subsonic truncated region (i.e. when $|\nabla\psi|^2\leq \mcT_\epsilon(\psi)$).
			
			{Let $\Omega$ be  the domain bounded by $S_0$ and $S_1\cup ([0,\infty)\times \{1\})$.} Noticing that $\Omega$ is unbounded, we thus make a further approximation by considering the problems in a family of truncated domains $\Omega_{\mu, R}:=\Omega\cap \{-\mu<x_1<R\}$,
			where $\mu$ and $R$ are two large positive numbers. To define the energy functional, we set
			\begin{align}\label{Gepsilon}
				G_\epsilon(t,z):=\frac{1}{2}\int_0^t g_\epsilon(\tau, z) d\tau +\frac{1}{\gamma}\left( g_\epsilon(0,z)^{-\gamma}- g_\epsilon(0,Q)^{-\gamma}\right)
			\end{align}
			and
			\begin{equation}\label{Phiepsilon}
				\Phi_\epsilon(t,z):=-G_\epsilon(t,z)+2\p_t G_\epsilon(t,z)t.
			\end{equation}
			Given  $\psi^\sharp_{\mu,R}\in C(\p\Omega_{\mu,R})\cap H^1(\Omega_{\mu,R})$ with $0\leq \psi^\sharp_{\mu,R}\leq Q$, we look for the minimizer of the problem
			\begin{align}\label{eq:mini_truncated}
				{\inf_{\varphi\in \mathcal{K}_{\psi^\sharp_{\mu,R}}}} J^\epsilon_{\mu, R, \Lambda}(\varphi),
			\end{align}
			where 
			\begin{align}\label{eq:energy}
				J^\epsilon_{\mu, R,\Lambda} (\varphi)&:=\int_{\Omega_{\mu, R}} G_\epsilon(|\nabla \varphi|^2, \varphi) + \lambda_\epsilon^2 \chi_{\{\varphi<Q\}} \ dx,\quad \lambda_\epsilon:=\lambda_\epsilon(\Lambda) := \sqrt{\Phi_\epsilon(\Lambda^2, Q)}
			\end{align}
			with $\Lambda$ defined in \eqref{defLambda}, and where
			\begin{equation}\label{admissibletrun}
				\mathcal{K}_{\psi^\sharp_{\mu,R}}:=\{\varphi\in H^1(\Omega_{\mu,R})| \varphi=\psi^\sharp_{\mu,R} \text{ on } \p\Omega_{\mu,R}\}.
			\end{equation}
			Note that $\Phi_\epsilon(\Lambda^2,Q)>0$. 
			Indeed, with the aid of ellipticity condition  \eqref{beta1eps}, a straightforward computation shows that $t\mapsto \Phi_\epsilon(t,z)$ is monotone increasing:
			\begin{equation}\label{Phi_dt}
				0<\frac{1}{2} c_\ast B_\ast^{-\frac{1}{\gamma-1}}\leq \p_t\Phi_\epsilon(t,z)=\frac{1}{2}g_\epsilon (t,z)+\p_t g_\epsilon (t,z)t\leq c^\ast\epsilon^{-1}B_\ast^{-\frac{1}{\gamma-1}}.
			\end{equation}
			Since $\Phi_\epsilon(0,Q)=-G_\epsilon(0,Q)=0$, one immediately has $\Phi_\epsilon(t,Q)>0$ for $t>0$.
			
			The existence and the H\"{o}lder regularity of minimizers for \eqref{eq:mini_truncated} will be shown in Lemmas \ref{lem:existence} and  \ref{lem:holder}, respectively. Assuming the existence and continuity of a minimizer $\psi$, we now derive the Euler-Lagrange equation for the variational problem \eqref{eq:mini_truncated} in the open set $\{\psi<Q\}$.
			%Here $\psi_0$ is a given boundary value, which we assume to satisfy $\psi_0\leq Q$.
			
			\begin{lem}\label{lem:vari_psi}
				Let $\psi$ be a minimizer of \eqref{eq:mini_truncated}. Assume that $\Omega_{\mu,R}\cap\{\psi< Q\}$ is open. Then $\psi$ is a weak solution to
				\begin{align}\label{ELpde}
					\nabla\cdot(g_\epsilon(|\nabla \psi|^2, \psi)\nabla \psi) - \p_zG_\epsilon (|\nabla\psi|^2,\psi)=0\quad\text{in}\,\,  \Omega_{\mu,R}\cap\{\psi< Q\}.
				\end{align}
				{Furthermore, if $|\nabla \psi|^2\leq(1-\epsilon)\mft_c(\mcB(\psi))$,} it holds that
				\begin{align}\label{dzG_expression}
					\p_zG_\epsilon(|\nabla\psi|^2,\psi)=\frac{\mathcal{B}'(\psi)}{g(|\nabla\psi|^2, \psi)}.
				\end{align}
			\end{lem}
			
			\begin{proof}
				Let $\eta$ be a smooth function compactly supported in $\Omega_{\mu,R}\cap\{\psi< Q\}$. Then
				\begin{align*}
					\frac{d}{d\vartheta}J^\epsilon_{\mu, R,\Lambda}(\psi+\vartheta\eta)\big|_{\vartheta=0}=\int_{\Omega_{\mu,R}}2\p_t G_\epsilon(|\nabla\psi|^2,\psi)\nabla \psi\cdot \nabla \eta + \p_zG_\epsilon(|\nabla \psi|^2,\psi)\eta \ dx.
				\end{align*}
				By the definition of $G_\epsilon$ in \eqref{Gepsilon}, minimizers of \eqref{eq:mini_truncated} satisfy the equation \eqref{ELpde}.
				
				Next, by the definition of $G_\epsilon$ in \eqref{Gepsilon}, one has
				$$\p_z G_\epsilon(t,z)=\frac{1}{2}\int_0^t \p_z g_\epsilon(\tau, z) d\tau-g_\epsilon(0,z)^{-(\gamma+1)}\p_zg_\epsilon(0,z).$$
				Note that 
				\begin{equation}\label{gm_g}
				g_\epsilon(t,z)=g(t,z) \quad\text{in } \mcR_\epsilon:=\{(t,z)| 0\leq t\leq(1-\epsilon)\mft_c(\mcB(z)), z\in \R\}.
				\end{equation} 
				Thus using \eqref{eq:claim1_g} one has
				\begin{align*}
					%\label{defdzG}
					\p_zG_\epsilon(t,z)=\mathcal{B}'(z)\varrho(t,z)-\mathcal{B}'(z)\varrho(0,z)-g(0,z)^{-(\gamma+1)}\p_zg(0,z)\quad\text{in } \mcR_\epsilon.
				\end{align*}
				It follows from \eqref{eq:drhot}  that $\p_t\varrho (0,z)=-\frac{1}{2}\varrho(0,z)^{-\gamma}$. This together with \eqref{eq:claim1_g} gives
				\begin{align}
					\label{eq:dg_0}
					\p_zg(0,z)=-\mathcal{B}'(z) \varrho(0,z)^{-\gamma}=-\mathcal{B}'(z)g(0,z)^{\gamma} .
				\end{align}
				Substituting  \eqref{eq:dg_0} into the expression of $\p_zG_\epsilon$ one gets
				\begin{align}\label{eq:dzG}
					\p_zG_\epsilon(t,z)=\mathcal{B}'(z)\varrho(t,z)=\frac{\mathcal{B}'(z)}{g(t,z)} \quad \text{for }(t,z)\in \mcR_\epsilon.
				\end{align}
				This completes the proof for the lemma.
			\end{proof}
			
			Next we show that {minimizers} of \eqref{eq:mini_truncated} satisfy the free boundary condition in the following domain variation sense.
			\begin{lem}\label{lem:fbcondition}
				Let $\psi$ be a {minimizer} of \eqref{eq:mini_truncated}. Assume that $\Omega_{\mu,R}\cap\{\psi< Q\}$ is open. Then
				\begin{align*}
					\Phi_\epsilon(|\nabla\psi|^2,\psi)=\lambda_\epsilon^2 \quad\text{on }\Gamma_{\psi}:=\p\{\psi<Q\}\cap \Omega_{\mu,R}
				\end{align*}
				in the sense that
				\begin{align*}
					\lim_{s\to 0+} \int_{\p\{\psi<Q-s\}} \Big[\Phi_\epsilon(|\nabla\psi|^2,\psi)-\lambda_\epsilon^2\Big](\eta\cdot \nu) \ d \mathcal{H} ^1=0 \quad \text{for any } \eta\in C^\infty_0(\Omega_{\mu,R};\R^2),
				\end{align*}
				where $\mathcal{H}^1$ is the one-dimensional Hausdorff measure.
			\end{lem}
			
			\begin{rmk}[Relation between $\lambda_\epsilon$ and $\Lambda$]\label{rmk:fbcondition}
				If the free boundary $\Gamma_\psi$ is smooth and $\psi$ is smooth near $\Gamma_\psi$, then it follows  from the monotonicity of $t \mapsto \Phi_\epsilon(t,z)$ for each $z$ and the definition of $\lambda_\epsilon$ in \eqref{eq:energy}  that
				$$|\nabla\psi|=\Lambda \quad \text{on }\Gamma_\psi.$$
				Moreover, since 
				$$\lambda_\epsilon^2=\Phi_\epsilon(\Lambda^2, Q) =\Phi_\epsilon(0, Q)+\int_0^{\Lambda^2}\p_t\Phi_\epsilon(s,Q)ds,$$
				then in view of \eqref{Phi_dt} and $\Phi_\epsilon(0,Q)=0$ one has
				\begin{align*}
					\frac{1}{2}c_\ast\Lambda^2	\leq B_\ast^{\frac{1}{\gamma-1}}\lambda_\epsilon^2\leq c^\ast \epsilon^{-1}\Lambda^2.
				\end{align*}
				Using Remark \ref{rmk:Q} we conclude that there exists a constant $C=C(\gamma,\epsilon, \|\bar u\|_{L^1([0,\bar H])})>0$ such that
				\begin{align}\label{ld_Ld}
					C^{-1}\leq \frac{Q^{\frac{1}{2}}\lambda_\epsilon}{\Lambda}\leq C.
				\end{align}
			\end{rmk}
			
			\begin{proof}[Proof of Lemma \ref{lem:fbcondition}]
				Let $\eta\in C_0^\infty(\Omega_{\mu, R}; \mathbb{R}^2)$ and  $\tau_\vartheta(x):=x+\vartheta\eta(x)$. If $|\vartheta|$ is sufficiently small, then $\tau_\vartheta$ is a diffeomorphism of $\Omega_{\mu,R}$. Let $\psi_\vartheta(y):=\psi(\tau_\vartheta^{-1}(y))$. Since $\psi_\vartheta\in \mathcal{K}_{\psi^\sharp_{\mu,R}}$ and $\psi$ is a minimizer, one has
				\begin{equation}\label{est_m1}
					\begin{aligned}
						0\leq &J^\epsilon_{\mu, R,\Lambda}(\psi_\vartheta)-J^\epsilon_{\mu,R,\Lambda}(\psi)\\
						=& \int_{\Omega_{\mu, R}}\left(G_\epsilon(|\nabla \psi (\nabla \tau_\vartheta)^{-1}|^2, \psi)+\lambda_\epsilon^2 \chi_{\{\psi<Q\}} \right)\det (\nabla \tau_\vartheta)  \\ &\quad -\int_{\Omega_{\mu, R}}\left(G_\epsilon(|\nabla \psi|^2,\psi)+\lambda_\epsilon^2\chi_{\{\psi<Q\}}\right)\\
						= &\ \vartheta\int_{\Omega_{\mu, R}}\left(G_\epsilon(|\nabla\psi|^2,\psi)+\lambda_\epsilon^2\chi_{\{\psi<Q\}}\right)\nabla\cdot \eta\\
						&\quad -2\vartheta\int_{\Omega_{\mu, R}} \p_t G_\epsilon(|\nabla \psi|^2,\psi)\nabla \psi \nabla \eta \nabla  \psi +o(\vartheta).
					\end{aligned}
				\end{equation}
				Dividing $\vartheta$ on both sides of \eqref{est_m1} and passing to the limit $\vartheta\rightarrow 0$ yield
				\begin{align}\label{eq_51.5}
					\int_{\Omega_{\mu, R}}\left(G_\epsilon(|\nabla\psi|^2,\psi)+\lambda_\epsilon^2\chi_{\{\psi<Q\}}\right)\nabla\cdot \eta  -2\int_{\Omega_{\mu, R}} \p_tG_\epsilon(|\nabla \psi|^2,\psi)\nabla \psi \nabla \eta \nabla  \psi =0.
				\end{align}
				Note that
				\begin{equation*}
					\begin{aligned}
						&\left(G_\epsilon(|\nabla\psi|^2,\psi)+\lambda_\epsilon^2\right)\nabla\cdot \eta\\
						=&\ \nabla\cdot \left[\left(G_\epsilon(|\nabla\psi|^2,\psi)+\lambda_\epsilon^2\right)\eta\right]-2\p_tG_\epsilon \nabla\psi D^2\psi\eta - \p_zG_\epsilon \nabla\psi\cdot \eta
					\end{aligned}
				\end{equation*}
				and
				\[
				\nabla \psi \nabla\eta\nabla\psi+ \nabla\psi D^2\psi \eta= \nabla(\eta\cdot \nabla\psi)\cdot \nabla\psi.
				\]
				Using the divergence theorem for \eqref{eq_51.5} we have
				\begin{align*}
					0&=\lim_{s\to 0+}\int_{\Omega_{\mu, R}\cap\{\psi<Q-s\}}\left[\left(G_\epsilon(|\nabla\psi|^2,\psi)+\lambda_\epsilon^2\right)\nabla\cdot \eta -2 \p_t G_\epsilon(|\nabla \psi|^2,\psi)\nabla \psi \nabla \eta \nabla  \psi\right] \\
					&=\lim_{s\to 0+}\int_{\p\{\psi<Q-s\}} \left[G_\epsilon(|\nabla\psi|^2,\psi) - 2\p_t G_\epsilon(|\nabla\psi|^2,\psi)  |\nabla\psi|^2+\lambda_\epsilon^2\right](\eta\cdot \nu)\\
					&\quad +\lim_{s\to 0+}\int_{\Omega_{\mu, R}\cap\{\psi<Q-s\}}\left[2\nabla\cdot(\p_t G_\epsilon(|\nabla\psi|^2,\psi)  \nabla \psi)-\p_zG_\epsilon (|\nabla\psi|^2,\psi) \right](\nabla\psi\cdot\eta)\\
					&=\lim_{s\to 0+}\int_{\p\{\psi<Q-s\}} \left[-\Phi_\epsilon(|\nabla\psi|^2,\psi)+\lambda_\epsilon^2\right](\eta\cdot \nu),
				\end{align*}
				where the equation  \eqref{ELpde} for $\psi$ in the open set $\{\psi<Q\}$ has been used to get the last equality. This finishes the proof of the lemma.
			\end{proof}

			\section{The existence and regularity for the free boundary problem}\label{secreg}
			In this section we study the existence and  regularity of the minimizers, as well as the regularity of the free boundary away from the nozzle.

			\subsection{Existence of minimizers}
			For the ease of notations in the rest of this section, let
			\begin{equation}\label{label_3}
				\mcD:=\Omega_{\mu, R} =\Omega\cap\{-\mu<x_1<R\},
				\quad
				\mcG(\bp,z):=G_\epsilon(|\bp|^2,z), 
				\quad \mcJ:=J^\epsilon_{\mu,R,\Lambda},
				\quad \lambda :=\lambda_\epsilon.
			\end{equation}
			Note that $\mcD$ is a bounded Lipschitz domain in $\R^2$ and is contained in the infinite strip $\R\times [0,\bar H]$, {$\mcG:\R^2\times \R\rightarrow \R$ is smooth in $\bp$ and $C^{1,1}$} in $z$ (with further properties summarized in Proposition \ref{prop:assumption} below), and $\lambda$ is a positive constant. The minimization problem \eqref{eq:mini_truncated}--\eqref{admissibletrun} can be rewritten as finding minimizers of
			\begin{equation}\label{eq:energy_new}
				\mcJ(\varphi):=\int_{\mcD} \mcG(\nabla\varphi, \varphi) + \lambda^2 \chi_{\{\varphi<Q\}}\ dx,
			\end{equation}
			%where $G$ satisfies \eqref{eq:convex}--\eqref{eq:upper_pzzG},
			over the admissible set
			\begin{equation}\label{admissibleK}
				\mathcal{K}_{\psi^\sharp}:=\{\varphi\in H^1(\mcD)| \varphi=\psi^\sharp \text{ on } \p \mcD\}.
			\end{equation}
			Here $\psi^\sharp:\overline{\mcD}\rightarrow \R$ is given, $\psi^\sharp\in H^1(\mcD)\cap C(\p\mcD)$ and $0\leq \psi^\sharp\leq Q$ on $\p \mcD$.
			%$\mcG$ is extended to be zero on $\R^n\times (-\infty, 0)$.
			%Let $m:=\min_{\p D}\psi_0$.

			%For simplicity we write $D:=D_{\mu,R}$, $G=G_\epsilon$ and $g=g_\epsilon$ in the sequel.
			Properties of $\mcG$ are summarized in the following proposition.
			
			\begin{prop}\label{prop:assumption}
				Let $G_\epsilon$ be defined in \eqref{Gepsilon} and $\mcG$ be defined in \eqref{label_3}, then the following properties hold.
				\begin{enumerate}
					\item[(i)] {There exist  positive constants  $\mfb_\ast=c_\ast B_\ast^{-\frac{1}{\gamma-1}}$ and $\mfb^\ast=c^\ast B_\ast^{-\frac{1}{\gamma-1}}$ with $c_\ast$ and $c^\ast$ depending only on $\gamma$,  such that}
					\begin{align}
						\mfb_*|\bp|^2 &\leq p_i \partial_{p_i} \mcG (\bp,z) \leq \mfb^*|\bp|^2,\label{eq:convex}\\
						\mfb_* |\mathbf\xi |^2 &\leq {\xi_i\partial_{p_ip_j} \mcG(\bp,z)\xi_j}
						\leq \mfb^* \epsilon^{-1} |\mathbf\xi|^2 \quad\text{for all }\mathbf\xi\in \R^2.\label{eq:convex0}
					\end{align}
					\item[(ii)] One has
					\begin{equation}\label{supportG}
						\begin{split}
							&\p_z\mcG(\bp,z)= 0 \quad\text{in } \{(\bp, z): \bp\in\R^2, z\in (-\infty,0]\},\\
							&\p_z\mcG(\bp,z)\geq 0 \quad \text{in } \{(\bp, z): \bp\in\R^2, z\in [Q,\infty)\}.
						\end{split}
					\end{equation}
					\item[(iii)] There exist constants
					\begin{equation}\label{delta}
						\delta':=\epsilon^{-1}C_\gamma \kappa_0, \quad 
						%\delta:= \delta' B_\ast^{-\frac{1}{\gamma-1}}\left(B_\ast ^{-\frac{\gamma-1}{2(\gamma+1)}}+\kappa_0 B_\ast^{-1}+\bar u_\ast^{-1}\right),
						\delta:= \delta' B_\ast^{-\frac{1}{\gamma-1}}\left(
						B_\ast^{-\frac12} +\kappa_0 B_\ast^{-1}+\bar u_\ast^{-1}\right),	
					\end{equation}
					where $\kappa_0$ is the constant defined in \eqref{eq:u0_eps0} and $C_\gamma>0$ is a constant depending only on $\gamma$, such that
					\begin{align}
						&\epsilon^{-1}|\p_z \mcG(\bp,z)|+|\bp\cdot \p_{\bp z} \mcG(\bp,z)|\leq \delta', \quad |\p_{\bp z}\mcG(\bp,z)|+|\p_{zz}\mcG(\bp,z)|\leq \delta, \label{eq:upper_pzzG}\\
						&\mcG(\mathbf 0,Q)=0,\quad \mcG(\bp,z)\geq {\frac{\mfb_*}2}|\bp|^2-
						C_\gamma\kappa_0%\delta'\epsilon
						{\min\{Q, (Q-z)_+\}}.\label{eq:com_energy}
					\end{align}
				\end{enumerate}
			\end{prop}
			\begin{proof}
				\emph{(i).} In view of the definition of $G_\epsilon$ in \eqref{Gepsilon}, straightforward computations yield
				\begin{equation*}
					\begin{split}
						%G(0,z)&= \gamma^{-1}\left(g_\epsilon(0,z)^{-\gamma}-g_\epsilon(0,Q)^{-\gamma}\right),\\
						%\p_p G(0,z)&=g_\epsilon(0,z)\cdot 0=0,\\
						p_i \p_{p_i} \mcG(\bp,z) &= g_\epsilon(|\bp|^2,z)|\bp|^2,\\
						\p_{p_ip_j} \mcG(\bp,z) &= g_\epsilon(|\bp|^2, z)\delta_{ij} + 2\p_t g_\epsilon(|\bp|^2,z) p_ip_j.
					\end{split}
				\end{equation*}
				Thus from \eqref{eq:g} and \eqref{beta1eps} one immediately gets \eqref{eq:convex}--\eqref{eq:convex0}. 
				
				\emph{(ii).} From the definition of $\mcG$ we have
				\begin{equation}\label{mcG_dz}
					\begin{split}
						\p_z \mcG(\bp,z)&=\frac{1}{2}\int_0^{|\bp|^2} \p_z g_\epsilon( \tau,z)d\tau - \left(g_\epsilon^{-\gamma-1}\p_z g_\epsilon\right)(0,z) \\
						&\stackrel{\eqref{eq:dg_0}}{=}\frac{1}{2}\int_0^{|\bp|^2} \p_z g_\epsilon( \tau,z)d\tau + \frac{\mathcal{B}'(z)}{ g_\epsilon(0, z)}.
						%\\
						%&\stackrel{\eqref{eq:dzg}}{\leq} C|\bar u'(\mfh(z))|\int_0^{|p|^2} \left(1+\p_\tau %g_\epsilon(\tau,z)\right) d\tau +\frac{\bar u'(\mfh(z))}{\bar\rho g_\epsilon(0, z)}.
					\end{split}
				\end{equation}
				This, together with explicit expression for $\p_zg_\epsilon$ in \eqref{gm_dz} and an integration by parts yields that
				\begin{equation*}\label{dzG}
					\begin{split}
						\p_z \mcG(\bp,z)&=\mathcal{B}'(z)\int_0^{|\bp|^2}\left(
						-1+\frac{g(\tau,z)}2
						(g^\ast - g(\tau,z)) \frac{\tau \mft_c'(\mathcal{B}(z))}{\mft_c(\mathcal{B}(z))} \right)\frac{\varpi'_\epsilon(\tau/\mft_c(\mathcal{B}(z)))}{\mft_c(\mathcal{B}(z))g(\tau,z)} \ d\tau\\
						&+\frac{\mathcal{B}'(z)}{g(|\bp|^2, z)}\varpi_\epsilon(|\bp|^2/\mft_c(\mathcal{B}(z))).
					\end{split}
				\end{equation*}
				We claim that for $\tau\in [0,\mft_c(\mathcal{B}(z)))$ one has
				\begin{align*}
					0\leq \frac{g(\tau,z)}2(g^\ast-g(\tau,z)) \frac{\tau \mft_c'(\mathcal{B}(z))}{\mft_c(\mathcal{B}(z))} \leq 1.
				\end{align*}
				Indeed, from the expression of $\mft_c(s)$ in \eqref{defF} and recalling \eqref{def:t_bound} we have that 
				\begin{align*}
					0\leq \frac{\tau \mft_c'(\mathcal{B}(z))}{\mft_c(\mathcal{B}(z))}\leq 2(c_\gamma B^\ast)^{\frac{2}{\gamma-1}}\leq 2\left(\frac{\gamma+1}2\right)^{\frac{2}{\gamma-1}}(c_\gamma B_\ast)^{\frac{2}{\gamma-1}}
				\end{align*}
				where $c_\gamma:=\frac{2(\gamma-1)}{\gamma+1}$. On the other hand,
				\begin{align*}
					0\leq 4g(\tau,z)(g^\ast-g(\tau,z))\leq (g^\ast)^2 \stackrel{\eqref{eq:upper_lower_g}}{=} \frac{1}{\varrho_c^2(B_\ast)} \stackrel{\eqref{defrhoc}}{=}(c_\gamma B_\ast)^{-\frac{2}{\gamma-1}}.
				\end{align*}
				Thus the claim holds after combining the above two inequalities.
				With the claim at hand we infer from $\varpi'\leq 0$, $\varpi\geq 0$ and $g\geq 0$ that  $\p_z \mcG(\bp, z)$ is a positive multiple of $\mathcal{B}'(z)$. {In view of \eqref{eq:sign_B} we}  conclude that \eqref{supportG} holds true. 
				
				\emph{(iii).} \emph{Proof for \eqref{eq:upper_pzzG}:} It follows from \emph{(ii)} as well as \eqref{eq:dB} and \eqref{eq:upper_lower_g} that 
				\begin{align*}
					|\p_z\mcG(\bp,z)|\leq 16|\mathcal{B}'(z)|\frac{\epsilon \mft_c(B^\ast)}{\epsilon \mft_c(B_\ast)g_\ast}\leq C_\gamma \kappa_0.
				\end{align*}
				A direct differentiation of $\p_z\mcG$ gives
				\begin{align*}
					\p_{p_iz}\mcG(\bp,z) = p_i \p_zg_\epsilon (|\bp|^2,z).
				\end{align*}
				Thus by the second inequality in  \eqref{eq:dzg},
				\begin{align*}
					|\bp\cdot \p_{\bp z}\mcG(\bp,z)|=|\bp|^2|\p_zg_\epsilon(|\bp|^2,z)|\leq c^\ast\kappa_0\epsilon^{-1}.
				\end{align*}
				This gives the first inequality in \eqref{eq:upper_pzzG}.
				
				By the first inequality in \eqref{eq:dzg}, \eqref{lable_2} and \eqref{label_1}, we get 
				\begin{align*}
					|\p_{\bp z}\mcG(\bp,z)|=|\bp||\p_zg_\epsilon(|\bp|^2,z)|\leq 
					C_\gamma\kappa_0\epsilon^{-1} B_\ast^{-\frac{\gamma+1}{2(\gamma-1)}}.
				\end{align*}

				In the end, using the expression for 
				$\p_z\mcG$ in \eqref{mcG_dz} and \eqref{eq:dg_0} one has
				\begin{align*}
					\p_{zz}\mcG(\bp,z)=\frac{1}{2}\int_0^{|\bp|^2} \p_{zz}g_\epsilon(\tau,z)\ d\tau + (\mathcal{B}'(z))^2 g_\epsilon(0,z)^{\gamma-2}+ \frac{\mathcal{B}''(z)}{g_\epsilon(0,z)}.
				\end{align*}
				The last two terms above are bounded by $C_\gamma(\kappa_0^2B_*^{-\gamma/(\gamma-1)}+\kappa_0\bar u_*^{-1}B_*^{-1/(\gamma-1)})$ in view of \eqref{eq:u0_eps0_B}, \eqref{eq:rhobar_B_*} and \eqref{eq:g}. To estimate the first term, we note that from \eqref{gm_dz} and that  $\p_z g (t,z)=  2\mathcal{B}' (z)\p_t(\frac{1}{g(t,z)})$, cf. \eqref{eq:claim1_g}, we have
				\begin{align*}
					\p_{zz}g_\epsilon(\tau,z) =&2\left(\mathcal{B}''(z)\p_t\left(\frac{1}{g(\tau,z)}\right) + \mathcal{B}'(z) \p_{tz} \left(\frac{1}{g(\tau,z)}\right) \right) \varpi_\epsilon\left(\frac{\tau}{\mft_c(\mathcal{B}(z))}\right) \\
					&+  2\p_z g(\tau,z)\frac{d}{dz}\varpi_\epsilon\left(\frac{\tau}{\mft_c(\mathcal{B}(z))}\right) +(g(\tau,z)-g^\ast)\frac{d^2}{dz^2}\varpi_\epsilon\left(\frac{\tau}{\mft_c(\mathcal{B}(z))}\right).
					%\left[\varpi''_\epsilon(\tau/\mft_c(\mathcal{B}(z)))(\mft'_c(\mathcal{B}(z)) %\mathcal{B}'(z))^2 \right.\\
					%	&\left. \varpi_\epsilon'(\tau/\mft_c(\mathcal{B}(z))) (\mft_c''(\mathcal{B}(z))%(\mathcal{B}'(z))^2+ \mft_c'(\mathcal{B}(z)) \mathcal{B}^{''}(z))\right].
				\end{align*}
				We plug it into the expression of $\p_{zz}\mcG$. For the term involving $\p_{tz}(\frac{1}{g})$, an integration by parts yields
				\begin{align*}
					\int_0^{|\bp|^2} \p_{tz}\left(\frac{1}{g(\tau,z)}\right)\varpi_\epsilon\left(\frac{\tau}{\mft_c(\mathcal{B}(z))}\right) \ d\tau = -\frac{\p_zg}{g^2}(|\bp|^2, z)\varpi_\epsilon \left(\frac{|\bp|^2}{\mft_c(\mathcal{B}(z))}\right)+\frac{\p_zg}{g^2}({0},z)\\
					+\int_0^{|\bp|^2}\frac{\p_zg}{g^2}(\tau,z) \frac{\varpi'_\epsilon(\tau/\mft_c(\mathcal{B}(z))) }{ \mft_c(\mathcal{B}(z))} \ d\tau.
				\end{align*}
				For the term involving $\varpi''_\epsilon$, we estimate it by using
				$$|\varpi''_\epsilon(t/\mft_c(\mathcal{B}(z))) |\leq 8\epsilon^{-2},\quad |\supp(\varpi''_\epsilon(\cdot /\mft_c(\mathcal{B}(z))))|\leq \epsilon \mft_c(\mathcal{B}(z)).$$ 
				%\begin{align*}
				%	\p_{zz}g_\epsilon& = \left(\p_z(\frac{\mathcal{B}'}{g^2})\p_t g + \frac{\mathcal{B}'}{g^2}\p_{tz}g\right) \varpi_\epsilon - 2\p_z g\varpi_\epsilon' \mft_c'\mathcal{B}'\\
				%	&\qquad -(g^\ast-g)\left(\varpi''_\epsilon(\mft'_c \mathcal{B}')^2 -\varpi_\epsilon' (\mft_c''(\mathcal{B}')^2+ \mft_c' \mathcal{B}^{''})\right).
				%\end{align*}
				%We plug this into the expression for $\p_{zz}\mcG$ and apply an integration by parts in $t$ for those terms involving $\p_{tz}g$ and $\varpi''_\epsilon$. 
				Thus we infer from \eqref{eq:dB}, \eqref{eq:dzg} and \eqref{eq:upper_lower_g} that 
				\begin{align*}
					|\p_{zz}\mcG|\leq C_\gamma \kappa_0\epsilon^{-1}\left( \kappa_0 B_\ast^{-\frac{\gamma}{\gamma-1}}+ \bar u_\ast^{-1}B_\ast^{-\frac{1}{\gamma-1}}\right).
				\end{align*}
				Thus we obtain the second inequality in \eqref{eq:upper_pzzG}.
				
				\emph{Proof for \eqref{eq:com_energy}:} In view of the definition of $\mcG$ in \eqref{label_3} and the definition of $G_\epsilon$ in \eqref{Gepsilon}, we get $\mcG(\mathbf0,Q)=0$. The second inequality in \eqref{eq:com_energy} follows directly from the strong convexity in $\bp$ and the estimates of derivatives in $z$ in \eqref{eq:upper_pzzG}. More precisely, the strong convexity property \eqref{eq:convex0} gives that 
				\begin{align*}
					\mcG(\bp,z)\geq \mcG(\mathbf 0,z)+ \frac{\mfb_\ast}{2}|\bp|^2,
				\end{align*}
				where we have also used that $\nabla_{\mathbf p} \mcG(\mathbf 0,z)=0$ (since  $\nabla_{\mathbf p} \mcG(\bp,z)=2{\mathbf p}\p_tg_\epsilon(|\bp|^2,z)$). To estimate $\mcG(\mathbf0,z)$ we use a first order Taylor expansion at $z=Q$. 
				{Then for $z\in [0,Q]$, one has 
					\begin{align*}
						|\mcG(\mathbf0,z)|=|\mcG(\mathbf0,Q)+ \p_z\mcG(0,\xi_z)(z-Q)|\leq C_\gamma\kappa_0 (Q-z),
					\end{align*}
					where in the last inequality we have used $|\p_z\mcG(\bp, z)|\leq C_\gamma\kappa_0$ from \eqref{eq:upper_pzzG}. Moreover, we infer from %the support condition
					\eqref{supportG} that 
					$\mathcal G(\mathbf 0,z)=\mathcal G(\mathbf 0,0)\geq -C_\gamma\kappa_0 Q$ %need \partial_x\mcG=0 for z\leq 0
					for $z<0$, and $\mathcal G(\mathbf 0,z)\geq \mathcal G(\mathbf 0,Q)=0$ for $z>Q$.} 
				Combining the above estimates together we obtain the second inequality in \eqref{eq:com_energy}.
			\end{proof}

			With Proposition \ref{prop:assumption} at hand, the existence of minimizers for the functional \eqref{eq:energy_new} over the admissible set \eqref{admissibleK} follows from standard theory for calculus of variations.
			\begin{lem}\label{lem:existence}
				Assume $\mcG$ satisfies \eqref{eq:convex0} and \eqref{eq:com_energy}. Then the functional \eqref{eq:energy_new} over the admissible set \eqref{admissibleK} has a minimizer.
			\end{lem}
			\begin{proof}
				First, in view of \eqref{eq:com_energy} it is not hard to find a function $\tilde\psi\in \mathcal{K}_{\psi^\sharp}$ with $\mathcal{J}(\tilde\psi)<C<\infty$.
				Let $\{\psi_k\}\subset\mathcal{K}_{\psi^\sharp}$ be a minimizing sequence. By \eqref{eq:com_energy} one has
				\begin{align*}
					{\frac{\mfb_*}2}
					\int_{\mathcal{D}}|\nabla\psi_k|^2 \leq \mathcal{J}(\psi_k)+ C_\gamma\kappa_0 Q|\mathcal{D}|<C.
				\end{align*}
				Then it follows from the Poincar\'e inequality that $\psi_k-\psi^\sharp$ is uniformly bounded in $H^1(\mathcal{D})$.
				Therefore,  there exists $\psi\in H^1(\mathcal{D})$ and a subsequence (relabeled) satisfying
				\begin{equation*}
					\psi_k\to \psi\quad \text{in } L^2(\mcD)\cap L^2(\p \mcD),\quad \psi_k\rightarrow \psi \text{ a.e. in }\mcD, \quad \text{and}\quad \nabla \psi_k \rightharpoonup  \nabla\psi\quad \text{in } L^2(\mcD).
				\end{equation*}
				Since $\mcG$ is convex with respect to $\bp$,  using Fatou's lemma one obtains 
				\begin{equation*}
					\begin{aligned}
						\mcJ(\psi)=&\int_{\mcD}\mcG(\nabla \psi, \psi) +\lambda^2 \chi_{\{\psi<Q\}}\ dx \\
						\leq & \liminf_{k\rightarrow \infty} \int_{\mcD}\mcG(\nabla \psi_k, \psi_k) +\lambda^2 \chi_{\{\psi_k<Q\}} \ dx = \liminf_{k\to \infty} \mcJ (\psi_k).
					\end{aligned}
				\end{equation*}
				This means that $\psi\in\mathcal{K}_{\psi^\sharp}$ is a minimizer.
			\end{proof}
			
			%\begin{rmk}
			%Later we will show that for $\Lambda$ sufficiently small, or $\Lambda^2<\min_{c\in [0,Q]}t_\epsilon(c)$ and the nozzle $N$ is concave, the solution satisfies $|\nabla \psi|^2<t_\epsilon(\psi)$ in $\{\psi<Q\}$ (i.e. we are away from where the truncation is active). Thus the equation \eqref{eq:fbp} is satisfied in $\{\psi<Q\}$ by Lemma \ref{lem:vari_psi}.
			%\end{rmk}

			\subsection{$L^{\infty}$-estimate and H\"{o}lder regularity}
			The main goal of this subsection is to establish the  $L^{\infty}$-bounds and the H\"{o}lder regularity  for the minimizers. 
			Firstly, one can show that  the minimizers are supersolutions of the elliptic equation
			\begin{equation}\label{basic_eq_psi}
				\partial_i(\p_{p_i} \mcG(\nabla\psi, \psi))- \p_z \mcG(\nabla\psi,\psi) = 0.
			\end{equation}
			\begin{lem}\label{lem:supersol}
				Let $\psi$ be a minimizer of \eqref{eq:energy_new} over the admissible set \eqref{admissibleK}.
				Then $\psi$ satisfies
				\begin{equation}\label{eq:supersol}
					\int_{\mcD}(\p_{p_i} \mcG(\nabla\psi, \psi)\partial_i \zeta + \p_z \mcG(\nabla\psi,\psi) \zeta)\geq 0,\quad\text{for all }\zeta\geq 0,\,\zeta\in C_0^\infty(\mcD).
				\end{equation}
				%In particular, if $\p_zG(p,z)\leq 0$, then $\psi$ is a supersolution in the sense that
				%\begin{equation*}
				%\int_{D}g(|\nabla \psi|^2,\psi)\nabla \psi\cdot \nabla \zeta\geq 0, \ \text{for all }\zeta\geq 0, \ \zeta\in C_0^\infty(D).
				%\end{equation*}
			\end{lem}
			\begin{proof}
				Since $\psi$ is a minimizer,
				for any $\vartheta>0$,
				\begin{align*}
					0\leq & \ \vartheta^{-1}\left(\mcJ(\psi+\vartheta\zeta)-\mcJ(\psi)\right)\leq \vartheta^{-1}\int_{\mcD}(\mcG(\nabla\psi+\vartheta\nabla\zeta,\psi+\vartheta\zeta) - \mcG(\nabla\psi,\psi))\\
					=&  \vartheta^{-1}\int_{\mcD} \left(\mcG(\nabla\psi+\vartheta\nabla\zeta,\psi) - \mcG(\nabla\psi,\psi)\right)\\
					& +\vartheta^{-1}\int_{\mcD} \left(\mcG(\nabla\psi+\vartheta\nabla\zeta,\psi+\vartheta\zeta) - \mcG(\nabla\psi+\vartheta\nabla\zeta,\psi)\right),
				\end{align*}
				where the second inequality follows from the simple fact  $\{\psi+\vartheta\zeta<Q\}\subset \{\psi< Q\}$.
				%The convexity of $p\mapsto G(p,z)$ implies that
				%\begin{align*}
				%\vartheta^{-1}\int_{D} G(\nabla\psi+\vartheta\nabla\zeta,\psi) - G(\nabla\psi,\psi)\leq %\int_{D}\p_pG(\nabla\psi+\vartheta\nabla\zeta,\psi)\cdot \nabla \zeta.
				%\end{align*}
				Thus taking $\vartheta\rightarrow 0$ yields \eqref{eq:supersol}. 
			\end{proof}

			Next, one has the following $L^\infty$-estimate for the minimizers.
			\begin{lem}\label{lem:upper_lower_bd}
				If $\psi$ is a minimizer of \eqref{eq:energy_new} over the admissible set \eqref{admissibleK}, then
				$$0 \leq \psi\leq Q.$$
			\end{lem}
			\begin{proof}
				%Denote $\hat\psi^*=\min\{0, Q-\psi\}$.
				Set $\psi^\vartheta:=\psi+\vartheta\min\{0, Q-\psi\}$ for $\vartheta>0$. Since $\psi^\sharp\leq Q$ on $\p \mcD$, one has $\psi^\vartheta =\psi^\sharp$ on $\p \mcD$ and thus $\psi^\vartheta\in \mathcal{K}_{\psi^\sharp}$. Since $\{\psi^\vartheta<Q\}=\{\psi< Q\}$ and  $\mcG$ is convex with respect to $\bp$, it follows from the same argument as in Lemma \ref{lem:supersol} that
				\begin{align*}
					0&\leq \vartheta^{-1}\left(\mcJ(\psi^\vartheta)-\mcJ(\psi)\right)\\
					&\leq \int_{\mcD} \p_{p_i} \mcG(\nabla\psi^\vartheta,\psi)\partial_i (\min\{0, Q-\psi\}) + \vartheta^{-1}\int_{\mcD} (\mcG(\nabla \psi^\vartheta,\psi^\vartheta)-\mcG(\nabla\psi^\vartheta,\psi)).
				\end{align*}
				Letting $\vartheta\rightarrow 0$, there holds 
				\begin{align*}
					0\leq &\int_{\mcD\cap \{\psi>Q \}}(-\p_{p_i} \mcG(\nabla\psi, \psi)\partial_i\psi + \p_z \mcG(\nabla\psi,\psi)(Q-\psi))\\
					\stackrel{\eqref{supportG}}{\leq} &\int_{\mcD\cap\{\psi>Q\}} -\p_{p_i} \mcG(\nabla\psi, \psi)\partial_i\psi\stackrel{\eqref{eq:convex}}{\leq} -\mfb_*\int_{\mcD\cap\{\psi>Q\}}|\nabla\psi|^2.
				\end{align*}
				This implies $\psi\leq Q$.
				
				The proof for the lower bound is similar. Set $\cpsi^\vartheta:=\psi-\vartheta\min\{0,\psi\}$ with $\vartheta>0$. It is straightforward to check that for $\vartheta\in (0,1)$, $\{\cpsi^\vartheta <Q\}=\{\psi< Q \}$. Thus
				\begin{align*}
					0\leq& \vartheta^{-1}\left(\mcJ(\cpsi^\vartheta)-\mcJ(\psi)\right)\\
					\leq & \int_{\mcD} -\p_{p_i} \mcG(\nabla\cpsi^\vartheta, \psi)\partial_i(\min\{0,\psi\})
					+ \vartheta^{-1}\int_{\mcD} (\mcG(\nabla \cpsi^\vartheta,\cpsi^\vartheta)-\mcG(\nabla\cpsi^\vartheta,\psi)).
				\end{align*}
				Letting $\vartheta\rightarrow 0$ gives
				\begin{align*}
					\int_{\mcD\cap \{\psi<0\}}(-\p_{p_i} \mcG(\nabla\psi, \psi)\partial_i \psi - \p_z\mcG(\nabla\psi,\psi)\psi) \geq 0.
				\end{align*}
				It follows from \eqref{supportG} and \eqref{eq:convex} that the measure of the set $\{\psi<0\}$ is zero. Hence  $\psi\geq 0$. This finishes the proof of the lemma.
			\end{proof}
			
			Next we prove a  comparison principle for the variational solutions.
			\begin{lem}\label{lem:comparison}  Given a domain $D\subset \mathbb{R}\times (a, a+h)$ for some $a\in \R$ and $h>0$, let $\psi\in H^1(D)$ be a supersolution of \eqref{basic_eq_psi} in the sense of \eqref{eq:supersol} and $\phi\in H^1(D)$ be a solution of \eqref{basic_eq_psi}, i.e., 
				\begin{equation}\label{eq:equ_phi}
					\int_{D} (\p_{p_i}\mcG(\nabla \phi,\phi)\p_{i} \zeta + \p_z \mcG(\nabla \phi,\phi)\zeta)=0,  \quad\text{for all }\zeta\in C^\infty_0(D).
				\end{equation}
				Assume that $\phi\leq \psi$ on $\p D$. Then  $\phi\leq \psi$ in $D$ as long as $h\leq h_0$ for some $h_0=h_0(\gamma,\epsilon,\bar u)>0$.%Proposition \ref{prop:assumption}.
			\end{lem}
			\begin{proof}
				Let $\eta:=(\phi-\psi)^+$. Then $\eta=0$ on $\p D$.
				%Assume $\sup_{\mcD}\eta>0$. 
				%and let $\mcD_\eta^+:=\{x\in \mcD:\eta(x)>0\}$.
				Substituting $\zeta=\eta$ in \eqref{eq:supersol} and using \eqref{eq:equ_phi} one has
				\begin{equation*}
					%\label{eq:73.5}
					\int_{D}\left(\p_{p_i} \mcG(\nabla\phi,\phi)-\p_{p_i} \mcG(\nabla\psi, \psi)\right)\p_i \eta + \int_{D}\left(\p_z \mcG(\nabla \phi, \phi)-\p_z \mcG(\nabla\psi, \psi)\right)\eta \leq 0.
				\end{equation*}
				The convexity of $\mathcal{G}$ in \eqref{eq:convex0} and the triangle inequality yield
				\begin{equation*}
					%\label{est_star1}
					\int_{D}\left(\p_{p_i} \mcG(\nabla\phi,\phi)-\p_{p_i} \mcG(\nabla\psi, \psi)\right)\p_i \eta\geq  \mfb_* \int_{D}|\nabla \eta|^2 + \int_{D}\left(\p_{p_i} \mcG(\nabla\psi, \phi)-\p_{p_i} \mcG(\nabla \psi,\psi)\right) \partial_{i} \eta,
					%&\geq \beta\int_{D}|\nabla \eta|^2 -\int_{D}A_0|\nabla \eta|^2 |\eta|.
				\end{equation*}
				where the last integral in the above inequality can be estimated from \eqref{eq:upper_pzzG} as
				\begin{equation*}
					%\label{est_star2}
					\begin{split}
						&\int_{D}\left(\p_{p_i} \mcG(\nabla\psi, \phi)-\p_{p_i} \mcG(\nabla \psi,\psi)\right)\p_i \eta \\
						= & \int_{D} \int_0^1\p_{zp_i} \mcG(\nabla\psi, \psi+ s(\phi-\psi)) \p_i\eta \ ds \ (\phi-\psi) \geq \ -\delta\int_{D} |\nabla\eta| \eta.
					\end{split}
				\end{equation*}
				Similarly, using \eqref{eq:upper_pzzG} and the triangle inequality one has
				\begin{equation*}
					\int_{D}\left(\p_z \mcG(\nabla \phi, \phi)-\p_z \mcG(\nabla\psi, \psi)\right)\eta\geq - \delta \int_{D}(|\nabla\eta|\eta + \eta^2).
				\end{equation*}
				Combining the above estimates together yields
				\begin{equation*}
					\int_{D}|\nabla \eta|^2 \leq \frac{2\delta}{\mfb_*} \int_{D}\left(|\nabla \eta|\eta+ \eta^2 \right).
				\end{equation*}
				{Applying Cauchy-Schwarz and Poincar\'e  inequalities} to $\eta(x_1,\cdot)$ for each $x_1$ (notice that $\eta=0$ on $\p D$ and $D\subset \R\times (a,a+ h)$), one has
				\begin{equation*}
					\int_{D}|\nabla \eta|^2 \leq \frac{1}{2} \int_{D}|\nabla \eta|^2+ C h^2\frac{\delta}{\mfb_\ast}\left(1+\frac{\delta}{\mfb_\ast}\right)\int_{D} |\nabla\eta|^2,
				\end{equation*}
				where $C>0$ is a universal constant. 
				We infer from the definitions of $\delta$ and $\mfb_\ast$ in Proposition \ref{prop:assumption} that $\frac{\delta}{\mfb_\ast}(1+\frac{\delta}{\mfb_\ast})$ has an upper bound depending on $\gamma,\, \epsilon$, and $\bar u$.
				Thus if $h\leq h_0$ for some $h_0=h_0(\gamma,\epsilon,\bar u)>0$  sufficiently small, then necessarily $\nabla\eta=0$ in $D$, which implies that $\eta=0$ in $D$.  This completes the proof of the lemma.
			\end{proof}

			Now we are in position to prove the H\"older regularity for the {minimizers.}  The proof follows from standard Morrey type estimates. We remark that at this stage we do not aim to obtain the explicit dependence of the $C^{0,\alpha}$ norm on parameters $Q$ and $\lambda$. The dependence will be more clear in the Lipschitz estimate of the minimizers in Section \ref{sec:Lipz}.
			
			\begin{lem}\label{lem:holder}
				Let $\psi$ be a minimizer of \eqref{eq:energy_new} over the admissible set \eqref{admissibleK}. Then $\psi\in C^{0,\alpha}_{loc}(\mcD)$ for any $\alpha\in (0,1)$. Moreover,
				\[
				\|\psi\|_{C^{0,\alpha}(K)}\leq C(\gamma,\bar u, \epsilon, \mfb_*, \delta',\,\delta, Q, \lambda,\alpha,K)\quad \text{for any }  K\Subset \mcD.
				\]
			\end{lem}
			\begin{proof}
				Given any $B_r\subset \mcD$ with $r\leq h_0$, where  $h_0$ is the same as in Lemma \ref{lem:comparison}. Let $\phi\in H^1(B_r)$ be the energy minimizer for the functional 
				\[
				\mathfrak{J}(\phi)=
				\int _{B_r} \mcG(\nabla\phi, {\phi})\ dx
				\]
				among $\mathfrak{K}=\{\phi \in H^1(B_r)| \phi=\psi\,\,\text{on}\,\, \partial B_r\}$.
				Then $\phi$ is a solution of the following problem
				\begin{equation}\label{eq:aux1}
					\left\{
					\begin{aligned}
						&{\p_{x_i}(\p_{p_i}\mcG(\nabla \phi,\phi))-\p_z\mcG(\nabla \phi, \phi)}=0 &&\text{ in }B_r, \\
						&\phi= \psi &&\text{ on } \partial B_r.
					\end{aligned}
					\right.
				\end{equation}
				It follows from the interior regularity for elliptic PDE (cf. for example \cite{HanLin}) that 
				\[
				\|\nabla\phi\|_{L^2(B_{\hat r})}\leq C \frac{\hat r}{r}\|\nabla \phi\|_{L^2(B_{r})}+ C r^2\quad \text{for any } \hat r\in (0,r),
				\]
				for some $C=C(\mfb_*, \delta')$. 
				Moreover, by Lemmas \ref{lem:upper_lower_bd} and  \ref{lem:comparison} one has $0\leq\phi\leq Q$. Define 
				\[
				\tilde{\psi}:=\left\{
				\begin{aligned}
					& \phi\quad && \text{in } B_r,\\
					& \psi\quad && \text{in } \mcD\backslash B_r.
				\end{aligned}\right.
				\]
				Clearly, $\tilde{\psi}\in H^1(\mcD)$. 
				Since $\psi$ is a minimizer, one has $\mcJ(\psi)\leq \mcJ(\tilde{\psi})$.
				This implies
				\begin{align*}
					%\label{eq:aux2}
					\int_{B_r}(\mcG(\nabla\psi, \psi)-\mcG(\nabla\phi,\phi))\leq \lambda^2\int_{B_r}(\chi_{\{\phi<Q\}}-\chi_{\{\psi<Q\}}).
				\end{align*}
				It follows from the Taylor expansion,   \eqref{eq:convex0} and \eqref{eq:upper_pzzG} that
				\begin{equation*}
					%\label{eq4-3-1}
					\begin{aligned}
						&\quad \int_{B_r} (\mcG(\nabla\psi, \psi)-\mcG(\nabla\phi,\phi)) \\
						&\geq \int_{B_r}\left( \p_{p_i}\mcG(\nabla\phi,\phi)\partial_i(\psi-\phi)+\p_z \mcG(\nabla\phi, \phi)(\psi-\phi)+\frac{\mfb_*}{4}|\nabla\psi-\nabla\phi|^2-C(\psi-\phi)^2\right)\\
						&\stackrel{\eqref{eq:aux1}}{=} \int_{B_r}\left(\frac{\mfb_*}{4}|\nabla\psi-\nabla\phi|^2-C(\psi-\phi)^2\right)
					\end{aligned}
				\end{equation*}
				for some $C=C(\mfb_*,\delta)$. Combining the above estimates together one has
				\begin{equation*}
					\begin{split}
						\int_{B_r}|\nabla\psi-\nabla\phi|^2 &\leq  {C\int_{B_r}((\psi-\phi)^2}+\lambda^2(\chi_{\{\phi<Q\}}-\chi_{\{\psi<Q\}}))\leq C(Q^2+\lambda^2)r^2,
					\end{split}
				\end{equation*}
				where $C=C(\mfb_*,\delta)$. 
				%Note that $\mfb_*, \delta,\,\delta'$ depend on $\gamma$, $\bar u$, $\epsilon$ and $Q$ by the expressions in Proposition \ref{prop:assumption} and \eqref{label_7}.   
				%Since $\int_{B_r}\chi_{\{\phi<Q\}}\leq |B_r|$,
				%this together with the simple relation $1-\chi_{\{\psi<Q \}}=\chi_{\{\psi=Q\}}$ gives
				%\[\mfb_* \int_{B_r}|\nabla\psi-\nabla\phi|^2\leq  {C\int_{B_r}(\psi-\phi)^2} + \lambda^2 \chi_{\{\psi=Q \}}\leq Cr^2.\]
				The
				desired H\"older regularity then follows from the standard Morrey type estimates (cf.  \cite[the proof for Theorem 3.8]{HanLin}).
			\end{proof}
			
			\subsection{Lipschitz regularity and nondegeneracy}\label{sec:Lipz}
			In this subsection we establish the (optimal) Lipschitz regularity of the minimizers and the so-called nondegeneracy property, which play an important role in getting the measure theoretic properties of the free boundary.
			
			%Throughout this section we fix a mass flux $Q$ satisfying \eqref{def:Q_*}.
		 Let $\psi$ be a minimizer of the functional  \eqref{eq:energy_new} over the admissible set \eqref{admissibleK}. 
			We consider the rescaled and renormalized function:
			\begin{equation}\label{psi_rescaled}
				\psi^\ast_{\bar x,r}(x):=\frac{Q-\psi(\bar x+ rx)}{Qr},\quad \bar x\in\mathcal D, \ r\in (0,1).
			\end{equation}
			Then $0\leq \psi^\ast_{\bar x,r}\leq 1/r$.   Moreover, $\psi^\ast_{\bar x, r}$ is a minimizer for 
			\begin{equation*}%\label{G_rescaled}
				\tilde{\mathcal J}(\varphi)=\int_{\mathcal{D}_{\bar x,r}} \tilde{\mcG}_r(\nabla\varphi,\varphi) + \lambda^2 \chi_{\{\varphi>0\}} \ dx 
			\end{equation*}
			over the admissible set $\tilde{\mathcal K}=\{\varphi\in H^1(\mcD_{\bar x,r})| \varphi=\psi^\ast_{\bar x,r} \text{ on } \p \mcD_{\bar x,r}\}$,
			where 
			\begin{equation*}
				\tilde{\mcG}_r(\bp,z):=\mcG(-Q\bp, Q-Qrz)
				\quad \text{and}\quad \mcD_{\bar x, r}:=\left\{\frac{x-\bar x}{r}\Big| x\in \mcD\right\}.
			\end{equation*}
			The straightforward computations show that $\psi^*_{\bar x, r}$ solves
			\begin{align}\label{msG_eq}
				\tilde{\mathscr{G}}_r(\nabla \psi^*_{\bar x, r}, \psi^*_{\bar x, r})\geq 0 \quad \text{in } \mcD_{\bar x,r}
			\end{align}
			and 
			\begin{equation}\label{eq:scale_w}
				\tilde{\mathscr{G}}_r(\nabla \psi^*_{\bar x, r}, \psi^*_{\bar x, r})=0 \quad \text{in} \ \mcD_{\bar x,r}\cap\{\psi^*_{\bar x, r}>0\},
			\end{equation}
			where
			\begin{equation}\label{defmsG}
				\tilde\msG_r(\bp, z) :=\p_{x_i} (\p_{p_i}\tilde{\mcG}_r(\bp,z))-\p_z\tilde{\mcG}_r(\bp,z).
			\end{equation}
			The properties of $\mcG$ in Proposition \ref{prop:assumption} can be translated into the properties of $\tilde\mcG_{r}$ in an obvious fashion: 
			\begin{align*}
				&\tilde{\mfb}_\ast|\bp|^2 \leq p_i \p_{p_i} \tilde{\mcG}_r (\bp, z)\leq \tilde{\mfb}^\ast |\bp|^2,\\
				&\tilde{\mfb}_\ast |\xi|^2 \leq 
				{\xi_i \p_{p_ip_j}\tilde{\mcG}_r(\bp,z)\xi_j}
				\leq \tilde{\mfb}^\ast \epsilon^{-1}|\xi|^2,
			\end{align*}
			where  $\tilde{\mfb}_\ast := Q^2 \mfb_\ast$, $\tilde{\mfb}^\ast:= Q^2\mfb^\ast$ and where $\mfb_\ast$, $\mfb^\ast$ are defined in Proposition \ref{prop:assumption}, and
			\begin{align}
				&\epsilon^{-1}|\p_z \tilde\mcG_{r}(\bp,z)|+|\bp\cdot \p_{\bp z} \tilde{\mcG}_r(\bp,z)|\leq \tilde{\delta'} r,\label{eq_rescale1}\\
				&  |\p_{\bp z}\tilde\mcG_{r}(\bp,z)|\leq \tilde{\delta} r , \quad |\p_{zz}\tilde\mcG_{r}(\bp,z)|\leq \tilde{\delta}r^2 ,\label{eq_rescale0}
			\end{align}
			where $\tilde{\delta'}:=\delta' Q$,  $\tilde{\delta}:=\delta Q^2$ and where $\delta'$, $\delta$ are defined in \eqref{delta}.  Note that from the explicit expressions of $\mfb_\ast$, $\mfb^\ast$, $\delta'$ and $\delta$ in Proposition \ref{prop:assumption} and Remark \ref{rmk:Q} we have
			\begin{align}\label{eq:BQ2}
				\tilde\mfb_\ast, \tilde\mfb^\ast, \tilde\delta',\tilde\delta \sim Q,
			\end{align}
			where $A\sim B$ means that $C^{-1}B\leq A\leq CB$ for some $C=C(\gamma,\epsilon, \bar u)$. In particular, $\tilde{\delta}/\tilde{\mfb}_\ast=\delta/\mfb_\ast$. Using $\tilde\mcG_r(\mathbf{0},0)=0$, $\p_z\tilde\mcG_r(\bp, 0)\leq 0$ which follows from \eqref{supportG}, and \eqref{eq_rescale1}--\eqref{eq_rescale0} we get
			\begin{align}
				&\tilde\mcG_{r}(\bp,z)\geq\frac{\tilde\mfb_*}{2}|\bp|^2-{\tilde\delta' r\epsilon z_+,} \label{eq_rescale2}\\
				& \p_z\tilde\mcG_r (\bp, z)\leq \tilde\delta r^2  z_+. \label{eq_rescale3}
			\end{align}
			We note that after the renormalization $\psi^\ast:=\psi^\ast_{\bar x,r}$ satisfies the elliptic equation 
			\begin{align*}
				%\sum_{i,j}
				a^{ij}_r\p_{ij} \psi^\ast  = f_r
			\end{align*}
			in $\{\psi^\ast>0\}$, with
			\begin{align*}
				a^{ij}_r:=\p_{p_ip_j}\tilde{\mcG}_r(\nabla\psi^\ast, \psi^\ast), \quad f_r:=%\sum_i 
				-\p_i\psi^\ast \p_{p_iz}\tilde{\mcG}_r(\nabla\psi^\ast,\psi^\ast)+ \p_z\tilde{\mcG}_r(\nabla\psi^\ast, \psi^\ast).
			\end{align*}
			From \eqref{eq_rescale1} and \eqref{eq:BQ2}, we conclude that there exist $C_\gamma>1$ depending only on $\gamma$ and $C=C(\gamma,\epsilon,\bar u)$ such that 
			\begin{align}\label{eq:ellipticity}
				1\leq \frac{\tilde\mfb^\ast}{\tilde\mfb_\ast}\leq \frac{C_\gamma}{ \epsilon},\qquad {\left|\frac{\p_z\tilde{\mcG}_r}{\tilde\mfb_\ast}\right|+} 
				\left|\frac{f_r}{\tilde\mfb_\ast}\right|\leq Cr.
			\end{align}

			%Since the energy depends on $\psi$ as well, apart from \eqref{eq:convex}--\eqref{eq:bounded} we need further assumptions on $\p_{zp} G(p,z)=\p_z g(|p|^2,z)p$ and $\p_{zz}G(p,z)$: for some $A_0>0$
			%\begin{equation}\label{eq:partial_g}
			%|\p_{zp} G(p,z)|\leq A_0 |p|,\quad |\p_{zz} G(p,z)|\leq A_0.
			%\end{equation}
			%We note that by \eqref{eq:claim1}, \eqref{eq:convex0} and \eqref{eq:convex}, in our specific problem we actually have
			%\begin{equation*}
			%A_0\leq L\beta^{-2}.
			%\end{equation*}

			In the next lemma we show that a minimizer $\psi$ decays at least linearly away from the free boundary.
			
			\begin{lem}\label{lem:linear}
				Let $\psi$ be a minimizer of \eqref{eq:energy_new} over the admissible set \eqref{admissibleK}. Let $\bar x\in \{\psi<Q\}$ satisfy $\dist(\bar x, \Gamma_\psi)\leq \min\{r_0, \frac{1}{4}\dist(\bar x,\p \mcD)\}$,
				where $r_0=r_0(\gamma,\epsilon, \bar u,\frac{\Lambda}{Q})>0$ is a small constant. Then there exists a constant  $C=C(\gamma,\epsilon,\bar u)>0$ such that
				\begin{equation*}
					Q-\psi(\bar x)\leq C \Lambda \dist(\bar x,\Gamma_\psi).
				\end{equation*}
			\end{lem}
			\begin{rmk}
				We will show in Corollary \ref{cor:cont_fit} that $\Lambda$ and $Q$ are comparable, i.e. there is {a constant}  $C=C(\gamma,\epsilon,\bar u)>0$ such that $C^{-1}\leq \frac{\Lambda}{Q}\leq C$. Thus the dependence on the parameter $\frac{\Lambda}{Q}$ is not an issue. 
			\end{rmk}
			\begin{proof}
				Let $\psi^*_{\bar x,r}(x)$ be defined in \eqref{psi_rescaled} with $r:=\dist(\bar x, \Gamma_\psi)\in (0,\min\{r_0, \frac{1}{4}\dist(\bar x,\p \mcD)\})$, where $r_0$ is to be determined later. 
				Then as discussed at the beginning of Section \ref{sec:Lipz}, $\psi^*_{\bar x, r}$ 
				satisfies \eqref{msG_eq}-\eqref{eq:scale_w}. 
				%The aim is to show there exists a constant $C>0$ depending only on $\gamma$, $\epsilon$, and $\bar u$, such that 
				%$$m_0:=\psi^*_{\bar x, r}(0)\leq CQ^{-1}\Lambda.$$
				Suppose 
				\begin{align}\label{m_0}
					m_0:=\psi^*_{\bar x, r}(0)\geq C_mQ^{-1}\Lambda
				\end{align}
				for a sufficiently large constant $C_m=C_m(\gamma,\epsilon,\bar u)>0$. If we can derive a contradiction, then the desired conclusion follows.
				%It suffices to show that if
				%\begin{align}\label{m_0}
				%	m_0:=\psi^*_{\bar x, r}(0)\geq cQ^{-1}\Lambda,
				%\end{align}
				%then $c\leq C$ for some constant $C=C(\gamma,\epsilon,\bar u)$. 
				For notational convenience, we omit {the subscripts  $\bar x$ and $r$} of $\psi^*_{\bar x, r}$ in the rest of the proof. The proof is divided  into four steps.

				\emph{Step 1. Construction of the barrier function.}
				For any given {$y\in \Gamma_{\psi^*}\cap \overline B_1$,} let $\phi$ be the solution of
				\begin{equation}\label{eq:v}\left\{
					\begin{aligned}
						&\tilde{\mathscr{G}}_r(\nabla \phi, \phi)=0 \quad \text{ in } B_2(y),\\
						& \phi=\psi^* \qquad \qquad \text{ on }\p B_2(y),
					\end{aligned}
					\right.
				\end{equation}
				{where $\tilde{\mathscr G}_r$ is defined in \eqref{defmsG}.}
				By \eqref{eq_rescale0} and the proof of Lemma \ref{lem:comparison}, if $r\leq r_1$ for some sufficiently small constant $r_1=r_1(\gamma, \epsilon, \bar u)>0$, one has $\phi \geq \psi^*$ in $B_2(y)$.  In particular, $\phi(0)\geq \psi^*(0)=m_0$. In view of \eqref{eq:ellipticity} and \eqref{m_0}, the Harnack inequality (cf. \cite[Theorems 8.17 and 8.18]{GT}) implies that
				%The interior  H\"older regularity for \eqref{eq:v} yields that 
				for some constant $C_0>0$ depending on $\gamma$, $\epsilon$, and $\bar u$ one has
				\begin{equation}\label{eq:harnack}
					\phi(x)\geq C_0 m_0 \quad \text{for any }x\in {B_{\frac12}(y)},
				\end{equation}
				provided $r\leq r_1$ for a possibly smaller $r_1$ depending on $\gamma,\epsilon,\bar u$, and $\frac{\Lambda}{Q}$.  
				
				We claim that if $r\leq r_1$ for a possibly smaller $r_1$ depending on $\gamma,\epsilon,\bar u$, and $\frac{\Lambda}{Q}$, then there exists {an absolute constant} $\bar C>0$ such that 
				\begin{equation}\label{eq:lower_bound}
					\phi(x)\geq \bar CC_0 m_0 (2-|x-y|) \quad \text{for any } x\in B_2(y).
				\end{equation}
				Indeed, by virtue of \eqref{eq:harnack}, it suffices to prove \eqref{eq:lower_bound} in 
				$B_2(y)\backslash B_{\frac12}(y)$. %$B_2(y)\backslash B_\vartheta(y)$
				For a fixed constant $\mu>4$, set
				\[\Phi_0(x):=C_0m_0\left(e^{-\mu|x-y|^2}-e^{-4\mu}\right).\]
				{Then it follows from \eqref{eq_rescale0}} and \eqref{eq_rescale3} that
				\begin{align*}
					\tilde{\mathscr{G}}_r(\nabla \Phi_0, \Phi_0)=&%\sum_{i,j}
					\p_{p_ip_j}\tilde{\mcG}_r(\nabla \Phi_0,\Phi_0)\p_{ij}\Phi_0 + \p_{p_iz}\tilde{\mcG}_r(\nabla \Phi_0,\Phi_0)\partial_i\Phi_0-\p_z\tilde{\mcG}_r(\nabla \Phi_0,\Phi_0)\\
					\geq &  
					\p_{p_ip_j}\tilde{\mcG}_r(\nabla \Phi_0,\Phi_0)\left(2\mu C_0m_0 e^{-\mu|x-y|^2}(2\mu (x_i-y_i)(x_j-y_j)-\delta_{ij})\right) \\
					&- r\tilde\delta  (2\mu C_0m_0 e^{-\mu|x-y|^2}|x-y|)-r^2\tilde\delta (C_0m_0 e^{-\mu|x-y|^2}).
				\end{align*}
				Therefore, if $r$ is sufficiently small depending on $\gamma,\, \bar u$, and $\epsilon$, then {$\tilde{\mathscr{G}}_r(\nabla \Phi_0, \Phi_0)\geq 0$} 
				in $B_2(y)\backslash {B_{\frac12}(y)}$. By the comparison principle (cf. Lemma \ref{lem:comparison}), one has $\phi\geq \Phi_0$ in $B_2(y)\backslash {B_{\frac12}(y)}$. Since there exists a positive constant $\bar C=\bar C(\mu)$ such that 
				$$\Phi_0(x)\geq \bar CC_0m_0(2-|x-y|) \quad \text{for any } x\in B_2(y)\backslash B_{\frac12}(y),$$ 
				this finishes the proof of the claim \eqref{eq:lower_bound}.

				\emph{Step 2. Gradient estimate.} We  will show that if $r\leq r_2$ for some $r_2=r_2(\gamma,\epsilon,\bar u)$, then  $\psi^*=\psi^*_{\bar x, r}$ satisfies
				\begin{equation}\label{eq:minimality}
					\tilde\mfb_*\int_{B_2(y)}|\nabla (\phi-\psi^*)|^2 \leq C \int_{B_2(y)} \lambda^2 \chi_{\{\psi^*=0\}}
				\end{equation}
				for some absolute constant $C$.  Indeed, the minimality of $\psi^*$ yields
				\begin{equation*}
					\int_{B_2(y)} (\tilde{\mcG}_r(\nabla \psi^*, \psi^*)-\tilde{\mcG}_r(\nabla \phi, \phi))
					\leq \lambda^2\int_{B_2(y)}(\chi_{\{\phi>0\}}-\chi_{\{\psi^*>0\}})\leq \lambda^2\int_{B_2(y)}\chi_{\{\psi^*=0\}}.
				\end{equation*}
				Using the Taylor expansion, \eqref{eq:v}, \eqref{eq:convex0} and \eqref{eq_rescale0}, one has
				\begin{equation}\label{gradest}
					\begin{split}
						&\int_{B_2(y)} (\tilde{\mcG}_r(\nabla \psi^*, \psi^*)-\tilde{\mcG}_r(\nabla \phi, \phi))\\
						\geq&	\frac12\int_{B_2(y)}\left(\tilde\mfb_*|\nabla(\psi^*-\phi)|^2-r\tilde\delta|\nabla(\psi^*-\phi)|\cdot|\psi^*-\phi|
						-{r^2}\tilde\delta|\psi^*-\phi|^2\right).
				\end{split}\end{equation}
				By H\"{o}lder and Poincar\'{e} inequalities, the last two terms can be absorbed by the first term on the right-hand side of \eqref{gradest} by taking $r\leq r_2$ for some $r_2=r_2(\gamma,\epsilon,\bar u)$ small.
				
				\emph{Step 3. Poincar\'e type estimate.} We claim the following Poincar\'e type estimate: there is a constant $C=C(\gamma,\bar u, \epsilon)$ such that if $r\leq r_1$, 
				\begin{equation}\label{eq:P}
					m_0|\mcS|^{1/2}\leq C \|\nabla(\phi-\psi^*)\|_{L^2(B_2(y))}\quad\text{with}\quad  \mcS:=\{x\in B_2(y)| \psi^*(x)=0\}.
				\end{equation}
				The proof is similar as  \cite[Lemma 3.2]{AC81} or  \cite[Lemma 2.2]{ACF84}. For the completeness,  we provide the details here. For $w\in B_1(y)$, consider a transformation $\msA_w$ from $B_2(y)$ to itself which fixes $\p B_2(y)$ and maps $w$ to $y$, for instance, $\msA_w^{-1}(x)=\frac{2-|x-y|}{2}(w-y)+x$. Set
				$\psi^*_w(x):=\psi^*(\msA^{-1}_w(x))$ and $\phi_w(x):=\phi(\msA^{-1}_w(x))$. Given a direction $\xi\in S^{1}$, define $r_{\xi}:= \inf \msR_\xi$ if
				\begin{equation*}
					\msR_\xi:=\{r| 1/4\leq r\leq 2, \ \psi^*_w(y+r\xi)=0\}\neq \emptyset.
				\end{equation*}
				Hence
				\begin{equation*}
					\begin{split}
						\phi_w(y+ r_{\xi}\xi)&=\int_{2}^{r_\xi} \frac{d}{dr}(\phi_w-\psi^*_w)(y+r\xi) dr\leq \int_{r_\xi}^2|\nabla (\phi_w-\psi^*_w)
						(y+r\xi)|dr\\
						&\leq \sqrt{2-r_\xi}\left(\int_{r_\xi}^2|\nabla (\phi_w-\psi^*_w)(y+r\xi)|^2dr\right)^{1/2}.
					\end{split}
				\end{equation*}
				On the other hand, by \eqref{eq:lower_bound}, one has for $r\leq r_1$,
				\begin{equation*}
					\begin{split}
						\phi_w(y+r_{\xi}\xi)&\geq {\bar C}C_0m_0 \left( 2-\left|\frac{2-r_{\xi}}{2}(w-y)+ r_\xi \xi\right|\right)\geq {\frac12\bar C}C_0m_0(2-r_\xi).
					\end{split}
				\end{equation*}
				Combining the above two inequalities together gives
				\begin{equation}\label{eq100.1}
					m_0^2(2-r_\xi)\leq C \int_{r_\xi}^2|\nabla (\phi_w-\psi^*_w)(y+r\xi)|^2dr
				\end{equation}
				for some $C=C(\gamma, \bar u, \epsilon)$. An integration of $\xi$ over $\mathbb{S}^{1}$ yields
				\begin{equation}\label{eq:scale1}
					m_0^2\int_{B_2(y)\backslash B_{1/2}(w)} \chi_{\{\psi^*=0\}}dx  \leq C \int_{B_2(y)}|\nabla (\phi-\psi^*)|^2dx.
				\end{equation}
				{Indeed, noting that $|\msA^{-1}_w(x)-w|\geq \frac{1}{2}$ implies $|x-y|\geq \frac{1}{4}$, changing back to the original variables and using the polar coordinates we have
					\begin{align}\label{eq100.2}
						\int_{B_2(y)\backslash B_{\frac{1}{2}}(w)} \chi_{\{\psi^*=0\}}\ dx& \leq C\int_{B_2(y)\backslash B_{\frac{1}{4}}(y)} \chi_{\{\psi^*_w =0\}} \ dx\\
						&=C \int_{\mathbb{S}^1} \int_{\frac{1}{4}}^2  \chi_{\{\psi^*_w=0\}} r\ dr d\xi \leq C \int_{\mathbb{S}^1} (2-r_\xi) \ d\xi.
					\end{align}
					On the other hand, it holds that 
					\begin{align}\label{eq100.3}
						\int_{\mathbb{S}^1}\int_{r_\xi}^2|\nabla (\phi_w-\psi^*_w)(y+r\xi)|^2drd\xi \leq C \int_{B_2(y)} |\nabla (\phi-\psi^*)|^2 dx.
					\end{align}
					Combining  \eqref{eq100.1}, \eqref{eq100.2}, and \eqref{eq100.3} yields	 \eqref{eq:scale1}.}
				A further integration over $w\in B_1(y)$ yields
				\begin{equation*}
					m_0^2 |\mcS|\leq C\int_{B_2(y)}|\nabla (\phi-\psi^*)|^2dx.
				\end{equation*}
				This is the desired estimate \eqref{eq:P}.
				
				\emph{Step 4. Decay estimate.} Combining \eqref{eq:minimality} and \eqref{eq:P} from Steps 2 and 3 gives,  {if $r\leq r_0=\min\{r_1,r_2\}$, then  for some $C=C(\gamma, \bar u, \epsilon)$,}
				\begin{equation*}
					m_0|\mcS|^{1/2} \leq \frac{C\lambda}{(\tilde \mfb_*)^{1/2}} |\mcS|^{1/2}\leq \frac{C\lambda}{Q^{1/2}} |\mcS|^{1/2},
				\end{equation*}
				where \eqref{eq:BQ2} has been used in the second inequality.
				If $|\mcS|>0$, then $m_0\leq C\lambda Q^{-1/2}$. By \eqref{ld_Ld} and \eqref{m_0} (where the constant $C_m$ is sufficiently large), one gets a contradiction. If $|\mcS|=0$, then $\psi^*=\phi$ a.e. in $B_2(y)$. By the interior regularity theory for elliptic equations (\cite{GT}),  $\psi^*$ and $\phi$ are continuous. Thus $\psi^*=\phi$ pointwise in $B_2(y)$. This, however, contradicts with the fact that $y$ is a free boundary point. This completes the proof of the lemma. 
			\end{proof}
			
			Lemma \ref{lem:linear} together with the elliptic estimates away from the free boundary yields the Lipschitz regularity of the minimizers.
			%The proof is the same as for Theorem 2.3 in \cite{ACF84} and we do not repeat it here.
			
			\begin{prop}\label{thm:Lipschitz}
				{Let $\psi$ be a minimizer of \eqref{eq:energy_new} over the admissible set \eqref{admissibleK}, then $\psi\in C^{0,1}_{loc}(\mcD)$.} Moreover, for any connected domain $K\Subset \mcD$ containing a free boundary point, the Lipschitz constant of $\psi$ in $K$ is estimated by $C\Lambda$, where $C$ depends on $\gamma,\,\epsilon,\,\bar u,\, \frac{\Lambda}{Q},\, K$, and $\dist(K,\p\mathcal{D})$.
			\end{prop}
			\begin{proof}
				{\it Step 1.} {Let $x\in \{\psi< Q\}\cap \mcD$. If $d(x):= \dist(x,\Gamma_\psi)\geq \frac14\min\{r_0,\,\dist(x,\p \mcD)\}$ where $r_0$ is given in Lemma \ref{lem:linear}, then the standard interior elliptic estimate applied to $\psi$ gives that $|\nabla\psi(x)|\leq CQ\dist(x,\p\mathcal{D})^{-1}$ for some $C=C(\gamma,\bar u, \epsilon)$. 
					If $d(x)\leq \frac14\min\{r_0,\,\dist(x,\p \mcD)\}$}, then let
				\begin{equation*}
					\psi^{*}_{x,r}(y):=\frac{Q-\psi(x+ r y)}{rQ}, \quad r:=d(x).
				\end{equation*}
				By Lemma \ref{lem:linear}, there is a constant $C=C(\gamma,\epsilon,\bar u)>0$ such that
				$$0\leq \psi^{*}_{x,r}(y)\leq \frac{C\Lambda}{Q} \quad\text{for any}\,\, y\in B_1.$$
				Note that $\psi^{*}_{x,r}$ satisfies \eqref{eq:scale_w} in $B_1$.
				By virtue of \eqref{eq:convex0}, $\p_i\psi^{*}_{x,r}$ solves  a uniformly  elliptic equation of divergence form and thus  is $C^{0,\alpha}$ by the De Giorgi-Nash-Moser  estimate (\cite{GT}). Hence for some $C=C(\gamma,\bar u, \epsilon)$ we have
				\begin{equation*}
					|\nabla \psi^{*}_{x,r}(0)| \leq C\left(\|\psi^{*}_{x,r}\|_{L^\infty(B_1)}+
					\|\p_z\tilde{\mcG}_r/\tilde\mfb_\ast\|_{L^\infty(B_1)}
					\right)\leq C\left(\frac{\Lambda}{Q}+r\right).
				\end{equation*}
				Here the upper bound for $|\p_z\tilde{\mcG}_r|$ in \eqref{eq:ellipticity} has been used to get the last inequality. This shows that $|\nabla\psi(x)|\leq C\Lambda+CQ$. Combining the above two cases we conclude that $\psi\in C^{0,1}_{loc}(\mathcal{D})$. Furthermore, $|\nabla\psi(x)|\leq C\Lambda$, where $C=C(\gamma,\epsilon,\bar u, \frac{\Lambda}{Q})$. 
				
				{\it Step 2.} If $K\Subset \mcD$ is connected and $K$ contains a free boundary point, then it follows from the Harnack inequality (\cite[Theorem 8.17]{GT}), the connectedness of $K$, and Lemma \ref{lem:linear} that
				\begin{equation*}
					Q-\psi\leq C\Lambda \quad \text{ in } K
				\end{equation*}
				for some $C$ depending on $\gamma,\epsilon,\bar u, \frac{\Lambda}{Q}, K$, and $\dist(K,\partial \mathcal{D})$. Thus similar arguments as in the second case of (i) give that $|\nabla \psi(x)|\leq C \Lambda$. 
			\end{proof}
			
			\begin{rmk}\label{rmk:up_to_bdry}
				By the boundary estimate for the elliptic equations, $\psi$ is Lipschitz up to the $C^{1,\alpha}$ portion $\Sigma\subset \p \mcD$ as long as the boundary data $\psi^\sharp\in C^{0,1}(\mcD\cup \Sigma)$. Moreover, if $K$ is a subset of $\overline{\mcD}$ with $K\cap \p \mcD$ being $C^{1,\alpha}$, $K\cap \mcD$ is connected, and $K$ contains a free boundary point, then $|\nabla\psi|\leq C\Lambda$ in $K$, where $C$ depends on $\gamma,\epsilon,\bar u, \frac{\Lambda}{Q}, K$ and the $C^{1,\alpha}$ norm of $K\cap \p\mcD$.
			\end{rmk}
			
			The next lemma investigates the nondegeneracy of the free boundary, whose proof relies the minimality of the energy functional and the construction of suitable barrier functions. 
			This lemma combined with Lemma \ref{lem:linear} gives the exact linear decay of a minimizer away from the free boundary. Moreover, the nondegeneracy plays an important role in deriving the measure theoretic properties of the free boundary. 
			
			\begin{lem}\label{lem:nondeg}
				Let $\psi$ be a minimizer of \eqref{eq:energy_new} over the admissible set \eqref{admissibleK}. Then for any $p>1$ and any $0<r<1$, there exist {positive constants $c_r=c_r(\gamma,\epsilon,\bar u, r,p)$ small 
					and $R_0=R_0(\gamma,\epsilon,\bar u, r)\in (0,1)$, %c_r(r,p,\mfb_*),R_0(c_r) 
					such that for any $B_{R}(\bar x)\subset \mathcal D$ with $R\leq \min\{R_0, c_r\frac{\Lambda}{Q}\}$}, 
				if
				\begin{equation*}
					\frac{1}{R}\left( %\frac{1}{|B_R(\bar x)|} 
					\dashint_{B_R(\bar x)}|Q-\psi|^p\right)^{1/p} \leq c_r\Lambda,
				\end{equation*}
				then $\psi=Q$ in $B_{rR}(\bar x)$.
			\end{lem}
			\begin{proof}
				As in Lemma \ref{lem:linear}, let 
				$$\psi_{\bar x,R}^*(x):=\frac{Q-\psi(\bar x+Rx)}{QR}$$ 
				for $R\in(0,R_0)$, where $R_0\in(0,1)$ is a constant  to be determined later. 
				Then $\psi_{\bar x,R}^*$ satisfies \eqref{msG_eq}--\eqref{eq:scale_w} (with $r$ being replaced by $R$). Now in view of the scaling we need to prove that if 
				\begin{equation*}
					\left(%\frac{1}{|B_1(0)|}
					\dashint_{B_1(0)}|\psi_{\bar x,R}^*|^p\right)^{1/p} \leq c_r\frac{\Lambda}{Q},
					%\leq c_r Q^{-\frac{1}{2}}\lambda,
				\end{equation*}
				then $\psi_{\bar x,R}^*=0$ in $B_{r}$. From now on, we drop the subscripts $\bar x$ and $R$ of $\psi_{\bar x,R}^*$ without ambiguity. By the $L^\infty$-estimate for the subsolution of \eqref{eq:scale_w} (cf. \cite[Theorem 8.17]{GT}), for any $r\in (0,1)$ there is a constant $C_\ast>0$ depending on $\gamma,\,\epsilon, \,\bar u,\, r$, and $p$ such that
				\begin{equation}\label{def M_r}
					\begin{split}
					\sup_{B_{\sqrt{r}}} \psi^*&\leq \frac{C_\ast}4 \left(\|\psi^*\|_{L^p(B_1)}+\|
						\p_z\tilde{\mcG}_R
						/\tilde \mfb_\ast \|_{L^\infty(B_1)}\right)
						\stackrel{\eqref{eq:ellipticity}}{\leq} \frac{C_\ast}2 \left(c_r\frac{\Lambda}{Q} +  R\right)\\
						&\leq {C_\ast c_r\frac{\Lambda}{Q}:=M_r,}
					\end{split}
				\end{equation}
				where in the last inequality we have used that $R\leq c_r\frac{\Lambda}Q$ by assumption.
				It suffices to show that if $c_r$ is sufficiently small, then $\psi^*=0$ in $B_r$. This is proved in two steps.
				
				\emph{Step 1. Upper bound of the energy.}  {Let $C_*$ in \eqref{def M_r} be sufficiently large. We claim that if $R\leq \min\{R_0, c_r\frac{\Lambda}{Q}\}$} for some $R_0=R_0(\epsilon, \gamma,\bar u, r)$, then there exists a constant  $\tilde{C}=\tilde{C}({r})>0$ such that 
				\begin{equation}\label{eq:up0}
					\int_{B_{r}} (\tilde{\mcG}_R(\nabla \psi^*, \psi^*) + \lambda^2\chi_{\{\psi^*>0\}}) \leq \tilde{C}\tilde\mfb_\ast M_r \int_{\p B_r} \psi^*.
				\end{equation}
				%\tbl{where $\tilde{\mathscr{G}}_R$ is defined in \eqref{defmsG}.} 
				The proof is based on the construction of a suitable energy competitor. Let
				\begin{align*}
					\Psi(x):=\begin{cases}
						{M_r} \frac{e^{-\mu r^2}-e^{-\mu|x|^2}}{e^{-\mu r^2}-e^{-\mu r}} &\quad \text{for}\,\, x\in B_{\sqrt{r}}\backslash B_r,\\
						0 &\quad  \text{for}\,\, x\in B_r.
					\end{cases}
				\end{align*}
				Similar as in the proof for Lemma \ref{lem:linear}, {if $\mu=\mu(r)$ is sufficiently large and  $R\leq R_0$ for some $R_0=R_0(\epsilon, \gamma,\bar u, r)>0$ sufficiently small,} then there exists {a constant}  $C_1=C_1({r})>0$ such that
				\begin{equation}\label{eq:u0}
					\partial_{x_i}( \p_{p_i} \tilde{\mcG}_R(\nabla \Psi, \Psi))\leq 
					%- C_1\sqrt{r}M_r\textcolor{red}{ \frac{e^{-\mu|x|^2}}{e^{-\mu r^2}-e^{-\mu r}}} 
					{-C_1\tilde \mfb_\ast M_r} \quad
					\text{in }B_{\sqrt{r}}\backslash B_r.
				\end{equation}
				Let $\phi:=\min\{\psi^*,\Psi\}$. Note that $\phi=\psi^*$ on $\p B_{\sqrt{r}}$, $\phi \equiv 0$ in $B_r$, and $\{x\in B_{\sqrt{r}}|\phi(x)> 0\}\subset \{x\in B_{\sqrt{r}}|\psi^*(x)> 0\}$.  This combined with the fact that $\psi^*$ is an energy minimizer in $B_{\sqrt{r}}$ gives
				\begin{equation*}\label{eq:wu00}
					\int_{B_{r}} (\tilde{\mcG}_R(\nabla \psi^*, \psi^*) + \lambda^2\chi_{\{\psi^*>0\}}) \leq  \int_{B_{\sqrt{r}}\backslash B_{r}} (\tilde{\mcG}_R(\nabla \phi, \phi)- \tilde{\mcG}_R(\nabla \psi^*, \psi^*)).
				\end{equation*}
				By the convexity of $\bp\mapsto \tilde{\mcG}_R(\bp,z)$ one has
				\begin{equation*}\label{eq:wu0}
					\begin{split}
						&\int_{B_{\sqrt{r}}\backslash B_{r}} (\tilde{\mcG}_R(\nabla \phi, \phi)- \tilde{\mcG}_R(\nabla \psi^*, \psi^*))\\
						\leq &\int_{B_{\sqrt{r}}\backslash B_{r}} \left[-\p_{p_i} \tilde{\mcG}_R(\nabla \phi, \phi)\partial_i(\psi^*-\phi) - (\psi^*-\phi)\int_0^1\p_z \tilde{\mcG}_R(\nabla \psi^*, \phi+s(\psi^*-\phi)) ds \right]\\
						\leq & \int_{(B_{\sqrt{r}}\backslash B_{r})\cap \{\psi^*>\Psi\}}\left[ -\p_{p_i} \tilde{\mcG}_R(\nabla \Psi, \Psi)\partial_i(\psi^*-\Psi) + {\tilde\delta' R\epsilon} (\psi^*-\Psi)\right],
					\end{split}
				\end{equation*}
				where the last inequality follows from the bound for $\p_z\tilde{\mcG}_R$ in \eqref{eq_rescale1}. 
				%and $0\leq \phi+s(\psi^*-\phi)\leq M_r$} for $s\in [0,1]$.
				Furthermore, multiplying \eqref{eq:u0} by $(\psi^*-\Psi)_+$ and integrating by parts (noting that $(\psi^*-\Psi)_+=\psi^*$ on
				$\p B_r$ and $(\psi^*-\Psi)_+=0$ on $\p B_{\sqrt{r}}$) one has
				\begin{equation*}\label{label_4}
					\begin{split}
						&\int_{(B_{\sqrt{r}}\backslash B_{r})\cap \{\psi^*>\Psi\}}-\p_{p_i} \tilde{\mcG}_R(\nabla \Psi, \Psi)\partial_i (\psi^*-\Psi)\\
						\leq & -\int_{\p B_r}\psi^* \p_{p_i} \tilde{\mcG}_R(\nabla \Psi, \Psi)\cdot \nu_i
						-\int_{(B_{\sqrt{r}}\backslash B_r) \cap \{\psi^*>\Psi\}} {C_1\tilde\mfb_\ast M_r %\frac{e^{-\mu|x|^2}}{e^{-\mu r^2}-e^{-\mu r}}
						}(\psi^*-\Psi).
					\end{split}
				\end{equation*}
				{Since the constant $C_*$ in \eqref{def M_r} is sufficiently large, then we infer from $\tilde \delta',\,\tilde\mfb_\ast\sim Q$ (cf. \eqref{eq:BQ2}) and  $R\leq c_r\frac{\Lambda}{Q}$ that}
				\begin{equation*}
					\int_{B_{r}} (\tilde{\mcG}_R(\nabla \psi^*, \psi^*) + \lambda^2\chi_{\{\psi^*>0\}}) \leq -\int_{\p B_r}\psi^*\p_{p_i} \tilde{\mcG}_R(\nabla \Psi, \Psi) \nu_i.
				\end{equation*}
				In view of the expression $|\nabla\Psi|$ on $\p B_r$ as well as \eqref{eq:convex}, there is a constant  $\tilde{C}=\tilde{C}({r})>0$ such that
				\[
				-\p_{p_i} \tilde{\mcG}_R(\nabla \Psi,\Psi) \nu_i \leq \tilde{C} \tilde\mfb_\ast M_r.
				\]
				Combining the above two estimates we obtain the desired estimate \eqref{eq:up0}.
				
				\emph{Step 2. Lower bound of the energy and conclusion.} 
				First, by the trace estimate there exists {a constant} $C=C(r)>0$ such that
				\begin{equation*}\label{traceineq}
					\int_{\p B_r} \psi^*\leq C\left(\int_{B_r} \psi^* + \int_{B_r}|\nabla \psi^*|\right).
				\end{equation*}
				The fact $0\leq \psi^* \leq M_r$ in $B_r$ and the  Cauchy-Schwarz inequality  yield
				\begin{align*}\label{Cauchyineq}
					\int_{B_r}\psi^* \leq M_r \int_{B_r}\chi_{\{\psi^*>0\}}
					\quad \text{and}\quad 
					\int_{B_r}|\nabla\psi^*|\leq \int_{B_r}\left(\frac{Q^{1/2}}{\lambda}|\nabla\psi^*|^2 + \frac{\lambda}{Q^{1/2}}\chi_{\{\psi^*>0\}}\right).
				\end{align*}
				We infer from the above two estimates that
				\begin{equation}\label{label_5}
					\int_{\p B_r}\psi^* \leq \frac{CQ^{1/2}}{\lambda}\int_{B_r}|\nabla \psi^*|^2 + \frac{C}{\lambda}\left(\frac{M_r}{\lambda}+\frac{1}{Q^{1/2}}\right)\int_{B_r} \lambda^2\chi_{\{\psi^*>0\}}.
				\end{equation}
				{On the other hand, since by \eqref{def M_r} and \eqref{ld_Ld} 
				\begin{equation}\label{label_6}
					R\leq M_r=C_\ast c_r\frac{\Lambda}{Q}\leq Cc_r\frac{\lambda}{Q^{1/2}}
				\end{equation} 
				and since $\tilde\delta',\,\tilde \mfb_\ast\sim Q$,} it follows from  \eqref{eq_rescale2} that
				\begin{equation}\label{eq:1}\begin{split}
						0\leq\int_{B_{r}}|\nabla\psi^*|^2
						\leq& \frac2{\tilde\mfb_*}\int_{B_{r}}(\tilde{\mcG}_R(\nabla\psi^*,\psi^*)+{\tilde\delta' R\epsilon} M_r\chi_{\{\psi^*>0\}})\\
						\leq&\frac{C(1+c_r^2)}{Q}\int_{B_{r}}(\tilde{\mathcal{G}}_R(\nabla\psi^*,\psi^*)+\lambda^2 \chi_{\{\psi^*>0\}}),
					\end{split}
				\end{equation} 
				where $C=C(\gamma,\epsilon, \bar u, r,p)$. 
				Combining the estimates \eqref{label_5} and \eqref{eq:1} together yields
				\begin{equation}\label{eq:up1}
					\begin{split}
						\int_{\p B_r}\psi^* \leq& \frac{C(1+c_r^2)}{\lambda Q^{1/2}}\int_{B_r} (\tilde{\mcG}_R(\nabla\psi^*,\psi^*)+\lambda^2\chi_{\{\psi^*>0\}})\\
						&+\frac{C}{\lambda}\left(\frac{M_r}{\lambda}+\frac{1}{Q^{1/2}}\right)\int_{B_r} \lambda^2\chi_{\{\psi^*>0\}}.
					\end{split}
				\end{equation}
				From \eqref{eq:up0}, \eqref{eq:up1} and $\tilde \mfb_\ast\sim Q$, there exists a constant $C=C(r)$  such that
				\begin{equation*}
					\begin{split}
						\int_{B_{r}} (\tilde{\mcG}_R(\nabla \psi^*, \psi^*) + \lambda^2\chi_{\{\psi^*>0\}}) \leq& C(1+c_r^2)\frac{M_rQ^{1/2} }{\lambda}\int_{B_r}\left(\tilde{\mcG}_R(\nabla \psi^*, \psi^*) + \lambda^2\chi_{\{\psi^*>0\}}\right)\\
						&+CQ \frac{M_r}{\lambda}\left(\frac{M_r}{\lambda}+\frac{1}{Q^{1/2}}\right)\int_{B_r} \lambda^2\chi_{\{\psi^*>0\}}.
					\end{split}
				\end{equation*}
				Using \eqref{label_6} again, we can conclude that if $c_r\leq \bar c$ for some $\bar c$ sufficiently small depending on $\gamma,\epsilon, \bar u, r$, and $p$, then
				\[
				\int_{B_r} (\tilde{\mcG}_R(\nabla \psi^*, \psi^*) + \lambda^2 \chi_{\{\psi^*>0\}})=0.
				\]
				{In view of \eqref{eq:1}, one gets} $\psi^*\equiv 0$ in $B_r$. Hence the proof of the lemma is complete.
			\end{proof}

			\subsection{Measure theoretic properties and regularity of the free boundary}

			From the Lipschitz regularity (Proposition \ref{thm:Lipschitz}) and nondegeneracy of the solutions (Lemma \ref{lem:nondeg}), one has the following measure theoretic properties of the free boundary. We recall that the free boundary is 
			\begin{align*}
				\Gamma_{\psi}:=\p\{\psi<Q\}\cap \mathcal{D}.
			\end{align*}
			
			\begin{prop}\label{thm:meas_fb}
				Let $\psi$ be a minimizer of \eqref{eq:energy_new} over the admissible set \eqref{admissibleK}. Then the following statements hold. 
				\begin{itemize}
					\item [(1)] $\mathcal{H}^{1}(\Gamma_\psi\cap K)<\infty$ for any $K\Subset \mcD$.
					\item [(2)] There is a Borel {function} $\varsigma_\psi$ such that
					\begin{equation*}
						\p_{x_i}( \p_{p_i} \mcG(\nabla \psi,\psi)) - \p_z \mcG(\nabla\psi,\psi)= \varsigma_\psi {\mathcal{H}^{1}}\lfloor \Gamma_\psi.
					\end{equation*}
					\item [(3)] There exist positive constants $c$ and $C$ depending on $\gamma,\,\epsilon$, $\bar u$, and   small constant $r_0=r_0(\gamma,\epsilon,\bar u, \frac{\Lambda}{Q})>0$, such that for any $B_r(x)\subset \mcD$ with $x\in \Gamma_\psi$ and $r\leq r_0$, one has
					\begin{equation*}
						c \leq \frac{|B_r(x)\cap \{\psi<Q\}|}{|B_r(x)|} \leq C.
					\end{equation*}
					\item [(4)] For any $K\Subset \mathcal{D}$, there exist positive constants $c$ and $C$ depending on $\gamma,\epsilon,\bar u$, and $\dist(K,\mcD)$, such that for every ball $B_r(x)\subset K$ with $x\in \Gamma_\psi\cap K$ and $r\leq r_0$, where $r_0$ is the same constant as in (3), one has
					\begin{equation*}
						\begin{split}
							c \left(1+\frac{\Lambda}{Q}\right)\leq \varsigma_\psi\leq C\left(1+\frac{\Lambda}{Q}\right), \quad c\left(1+\frac{\Lambda}{Q}\right)r\leq \mathcal{H}^{1}(B_r(x)\cap \Gamma_\psi)\leq C\left(1+\frac{\Lambda}{Q}\right) r.
						\end{split}
					\end{equation*}
				\end{itemize}
			\end{prop}
			The proof is similar as that for \cite[Theorems 2.8 and  3.2]{ACF84}, so we omit the details here.
			
			In order to get the regularity of the free boundary, we need to derive some strengthened estimates for $|\nabla \psi|$ near the free boundary. For that we use the additional symmetry property that the  function $\mcG$ is radial in $\bp$, i.e. 
			\begin{align}\label{eq:radial}
				\mcG(\bp, z)=G_\epsilon(|\bp|^2,z).
			\end{align}
			We also recall from \eqref{Phiepsilon} that
			\begin{align}\label{eq:Phi_1}
				\Phi(|\bp|^2,z):=\Phi_\epsilon(|\bp|^2,z)=%\sum_{i}
				p_i\p_{p_i}\mcG(\bp,z)-\mcG(\bp,z),\quad \Phi(\Lambda^2, Q)=\lambda^2, \quad \Phi(0, Q)=0.
			\end{align}

			\begin{lem}\label{lem:gradu}
				Let $\psi$ be a minimizer of  \eqref{eq:energy_new} over the admissible set \eqref{admissibleK}, where the function $\mcG$ satisfies \eqref{eq:convex}--\eqref{eq:com_energy} in Proposition \ref{prop:assumption} and  \eqref{eq:radial}. Then for $K\Subset \mcD$, there exist constants $\alpha=\alpha(\gamma,\epsilon,\bar u)\in (0,1)$,  $r_0=r_0(\gamma,\epsilon,\bar u, \frac{\Lambda}{Q})\in (0,1)$,
				and $C=C(\gamma,\epsilon,\bar u, \frac{\Lambda}{Q}, K)>0$ such that for any ball $B_{2r}(x)\subset K$ with $x\in \Gamma_\psi$ and $r\in (0,r_0)$, one has
				\begin{align*}
					\sup_{B_r(x)}|\nabla \psi|^2
					\leq \Lambda^2 + C\Lambda^2 r^\alpha.
				\end{align*}
			\end{lem}
			\begin{proof}
				The proof is similar as \cite[Theorem 4.1]{ACF84}. It follows from the equation of $\psi$ and the Bernstein estimate (\cite[Chapter 15]{GT}) that $w:=|\nabla \psi|^2$ satisfies
				\begin{align*}
					-\mathfrak a^{ij}\p_{ki}\psi\p_{kj}\psi+ \frac{1}{2} \p_i(\mathfrak a^{ij}\p_j w) + \p_z\mathfrak a^{ij} \p_{ij}\psi w - \frac{1}{2} \p_i \psi\p_z \mathfrak a^{ij} \p_j w + \p_k (\mfb \p_k \psi) -\mfb\Delta \psi=0
				\end{align*}
				in $\{\psi<Q\}$, where 
				\begin{align*}
					\mathfrak a^{ij}:=\p_{p_ip_j}\mcG(\nabla \psi, \psi),\quad
					\mfb:=\p_{p_i z} \mcG (\nabla \psi,\psi)\p_i\psi - \p_z \mcG(\nabla \psi, \psi).
				\end{align*}
				We infer from the radial symmetry of $\mcG$ in \eqref{eq:radial} that
				\begin{align*}
					\p_z\mathfrak a^{ij}(\bp,z)=2\p_{tz}G_\epsilon(|\bp|^2,z) \delta_{ij}+ 4\p_{ttz}G_\epsilon(|\bp|^2,z) p_ip_j.
				\end{align*}
				This combined with $2|\bp|^2\p_{tz}G_\epsilon(|\bp|^2,z)=\bp\cdot \p_{\bp z}\mcG(\bp,z)$ gives that
				\begin{align*}
					&\quad\p_z \mathfrak a^{ij} \p_{ij}\psi w -  \frac{1}{2}\p_i \psi\p_z \mathfrak a^{ij} \p_j w-\mfb\Delta\psi\\
					&=2\p_{tz}G_\epsilon (|\nabla\psi|^2,\psi)\left(|\nabla\psi|^2\Delta\psi - \p_i\psi\p_k\psi\p_{ik}\psi\right)-\mfb\Delta\psi\\
					&=-\nabla\psi\cdot\p_{\bp z}\mcG (\nabla\psi,\psi)\frac{\nabla\psi}{|\nabla\psi|}D^2\psi\frac{(\nabla\psi)^t}{|\nabla\psi|}+\p_z\mcG(\nabla\psi,\psi)\Delta\psi.
				\end{align*}
				Thus by Young's inequality,
				\begin{align*}
					&\quad |\p_z \mathfrak a^{ij} \p_{ij}\psi w -  \frac{1}{2}\p_i \psi\p_z \mathfrak a^{ij} \p_j w-\mfb\Delta\psi |\\
					&\leq \frac{1}{\mfb_\ast}\left(|\p_z\mcG(\nabla\psi,\psi)|^2+ |\nabla\psi\cdot \p_{\bp z}\mcG(\nabla\psi,\psi)|^2\right)+ \frac{\mathfrak b_\ast}{2}|D^2\psi|^2.
				\end{align*}
				Since 
				\begin{align*}
					%\sum_{i,j,k}
					\mathfrak a^{ij} \p_{ki}\psi\p_{kj}\psi \geq \mathfrak b_\ast |D^2\psi|^2,
				\end{align*}
				which follows from the convexity of the functional, cf. \eqref{eq:convex0}, we conclude that $w$ is a subsolution to the uniformly elliptic linear equation
				\begin{align*}
					\mathcal{L}w:=\frac{1}{2} \p_i(\mathfrak a^{ij}\p_j w) \geq -\mathfrak f - \p_i \mathfrak g^i, 
				\end{align*}
				where 
				\begin{align*}
					\mathfrak f=\frac{1}{\mathfrak b_\ast}\left(|\p_z\mcG(\nabla\psi,\psi)|^2+ |\nabla\psi\cdot \p_{\bp z}\mcG(\nabla\psi,\psi)|^2\right),\quad \mathfrak g^i = \mathfrak b \p_i \psi.
				\end{align*}
				We infer from  \eqref{eq:upper_pzzG} that for some $C=C(\gamma,\epsilon,\bar u)$ the inhomogeneities $\mathfrak f$ and $\mathfrak g^i$ satisfy
				\begin{align*}
					\|\mathfrak f\|_{L^\infty(B_{2r}(x))}\leq C\mfb_\ast^{-1},\quad  \|\mathfrak g^i\|_{L^\infty(B_{2r}(x))}\leq C \|\nabla\psi\|_{L^\infty(B_{2r}(x))}.
				\end{align*}
				For $\vartheta{>0}$, let
				\begin{equation*}
					W_\vartheta := \left\{
					\begin{aligned}
						&(|\nabla\psi|^2-\Lambda^2-\vartheta)^+,\,\,&\text{  in }\{\psi<Q\},\\
						& 0,\,\,&\text{ in }\{\psi=Q\}.
					\end{aligned}
					\right.
				\end{equation*}
				Note that $W_\vartheta$ is supported away from the free boundary (the proof is similar to that in \cite[Lemma 3.4]{ACF84}). Thus from the interior regularity of $\psi$ and the expression of $W_\vartheta$ one immediately has $W_\vartheta\in C^{0,1}(B_r(x))$.  Denote $W_\vartheta^*(r):=\sup_{B_r(x)} W_\vartheta$. Then $W^*_\vartheta(r)-W_\vartheta$ satisfies
				\begin{equation*}
					\left\{
					\begin{aligned}
						&\mathcal{L}(W^*_\vartheta(r)-W_\vartheta)\leq  \mathfrak f+\p_i\mathfrak g^i  &&\text{in } B_{2r}(x),\\
						& W^*_\vartheta(r)-W_\vartheta=W^*_\vartheta(r) &&\text{in } B_{2r}(x) \cap \{\psi=Q\}.
					\end{aligned}
					\right.
				\end{equation*}
				%where $\mfb^i$ and $\mfc$ are both bounded. 
				Using \cite[Theorem 8.18]{GT} with $p\in [1,\infty)$ one has for some $C=C(\gamma,\epsilon,\bar u)>0$ and $k(r):=\mfb_\ast^{-2} r+\mfb_\ast^{-1}\|\nabla\psi\|_{L^\infty(B_{2r}(x))} r^2$,
				\begin{align*}
					C\left(\inf_{B_{r/2}(x)}[W^*_\vartheta(r)-W_\vartheta]+k(r)\right)\geq  r^{-\frac{2}{p}}\|W^*_\vartheta-W_\vartheta\|_{L^p(B_r(x))}. 
				\end{align*}
				From the positive density property of $B_r(x)\cap \{\psi=Q\}$ (cf.  Proposition \ref{thm:meas_fb}(3)) we infer that for some $C'=C'(\gamma,\epsilon,\bar u)>0$,
				\begin{align*}
					r^{-\frac{2}{p}}\|W^*_\vartheta-W_\vartheta\|_{L^p(B_r(x))}  \geq  C'W^*_\vartheta(r).
				\end{align*} 
				Taking $\vartheta\rightarrow 0$ and a rearrangement of the above two inequalities yield for any $r\in (0,r_0)$
				\begin{align*}
					W^*_0\left(\frac{r}{2}\right)=\sup_{B_{\frac{r}{2}}(x)} W_0\leq \eta W^*_0(r)+k(r),\quad \eta=1-\frac{C'}{C}\in (0,1).
				\end{align*}
				It then follows from the iteration lemma (cf. eg. Lemma 8.23 in \cite{GT}) that for any $\mu\in (0,1)$ 
				\begin{align*}
					W^*_0(r)\leq C W^*_0(r_0)\left(\frac{r}{r_0}\right)^\alpha+ k(r^\mu r_0^{1-\mu}),
				\end{align*}
				for some $C=C(\gamma,\epsilon,\bar u)$ and $\alpha=\alpha(\gamma,\epsilon,\bar u, \mu)\in (0,1)$. Recalling the Lipschitz bound for $\psi$ in Proposition \ref{thm:Lipschitz} and noting that $\mfb_\ast\sim Q^{-1}$ by the expression of $\mfb_\ast$ in Proposition \ref{prop:assumption} and \eqref{label_7}, we obtain the desired estimate.
			\end{proof}
			
			\begin{rmk}\label{rmk_gradu}
				In Lemma \ref{lem:gradu}, if $B_{2r}(x)\subset B_{2R}(x)\subset K$ with $R\in (0,r_0)$, then there exists {a constant} $C=C(\gamma,\epsilon,\bar u, \frac{\Lambda}{Q},K)$ such that
				$$\sup_{B_r(x)}|\nabla\psi|^2\leq \Lambda^2+C\Lambda^2\left(\frac rR\right)^{\alpha}.$$
				This follows from applying Lemma \ref{lem:gradu} to $\psi_{x,R}(y)=\psi(x+Ry)/R$.
			\end{rmk}
			
			{Similar arguments as in \cite[Theorem 4.3]{ACF84}} give the following gradient estimate.
			\begin{prop}\label{prop:gradu2}
				Let $\psi$ be a minimizer of   \eqref{eq:energy_new} over the admissible set \eqref{admissibleK}, where the function $\mcG$ satisfies \eqref{eq:convex}--\eqref{eq:com_energy} in Proposition \ref{prop:assumption} and  \eqref{eq:radial}.  Then for any $B_{2r}(x)\subset K \Subset \mcD$ with $x\in \Gamma_\psi$ and $r\in (0,r_0)$, where $r_0$ is the same as in Lemma \ref{lem:gradu}, one has
				\begin{align*}
					\dashint_{B_r(x)\cap \{\psi<Q\}} \left(\Lambda^2-|\nabla\psi|^2\right)^+\leq C\Lambda^2 |\ln r|^{-1},
				\end{align*}
				where $C=C(\gamma,\epsilon,\bar u, \frac{\Lambda}{Q}, K)>0$ is a constant.
			\end{prop}
			
			\begin{proof}
				Without loss of generality,
				assume that $x$ is the origin. For any constant $\vartheta>0$ and any $\eta\in C_0^\infty(\mcD)$ with $\eta\geq 0$, the function $\psi_{\vartheta,*} :=\min\{\psi+\vartheta \eta, Q\}$ is admissible, so that $\mcJ(\psi)\leq \mcJ(\psi_{\vartheta,*})$. 
				{Thus
					\begin{align}\label{eq:2d_lambda}
						\int_{\{Q-\vartheta\eta\leq\psi<Q\}} \lambda^2  \leq \int_{\mcD}  (\mcG(\nabla\psi_{\vartheta,*}, \psi_{\vartheta,*})-\mcG(\nabla\psi,\psi)).
					\end{align}
					Denote $\tilde\psi_{\vartheta}:=\psi_{\vartheta,*}-\psi$. Then the integrand on the right hand side of \eqref{eq:2d_lambda} can be written as
					\begin{align*}
						&\mcG(\nabla\psi_{\vartheta,*}, \psi_{\vartheta,*})-\mcG(\nabla\psi,\psi) \\
						=& \mcG(\nabla \psi_{\vartheta,*}, \psi)-\mcG(\nabla\psi,\psi) + \mcG(\nabla\psi_{\vartheta,*},\psi_{\vartheta,*})-\mcG(\nabla\psi_{\vartheta,*},\psi)\\
						=&\p_{p_i}\mcG(\nabla\psi,\psi)\partial_i \tilde\psi_{\vartheta} 
						+ \int_0^1 d \tau \int_0^\tau \partial_i\tilde\psi_{\vartheta}\p_{p_ip_j}\mcG(\nabla\psi+ s\nabla\tilde\psi_{\vartheta},\psi)\partial_j\tilde\psi_{\vartheta} ds\\
						&+\tilde\psi_{\vartheta}\int_0^1\p_z\mcG(\nabla\psi_{\vartheta,*}, \psi+s\tilde\psi_{\vartheta}) ds.
					\end{align*}
					In the set $\{\psi+\vartheta\eta<Q\}$ one has $\tilde\psi_{\vartheta}=\vartheta \eta$. Hence by \eqref{eq:convex0}, it holds that
					\begin{align*}
						&\int_{\{\psi+\vartheta\eta<Q\}}\int_0^1 d\tau \int_0^\tau \p_i\tilde\psi_{\vartheta}\p_{p_ip_j}\mcG(\nabla\psi+ s\nabla\tilde\psi_{\vartheta},\psi)\p_j\tilde\psi_{\vartheta}ds
						\leq  \mfb^\ast\vartheta^2\int_{\{\psi+\vartheta\eta<Q\}}|\nabla\eta|^2.
					\end{align*}
					In the set $\{\psi+\vartheta\eta\geq Q\}$ one has $\psi_{\vartheta,*}=Q$ so that $\tilde\psi_{\vartheta}=Q-\psi$, thus
					\begin{align*}
						&\int_{\{\psi+\vartheta\eta\geq Q\}}\int_0^1 d\tau \int_0^\tau \partial_i\tilde\psi_{\vartheta}\p_{p_ip_j}\mcG(\nabla\psi+ s\nabla\tilde\psi_{\vartheta},\psi)\partial_j\tilde\psi_{\vartheta} ds\\
						= & \int_{\{\psi+\vartheta\eta\geq Q\}}\int_0^1 d\tau \int_0^\tau \partial_i\psi\p_{p_ip_j}\mcG((1-s)\nabla\psi,\psi)\partial_j\psi ds.
					\end{align*}
					On the other hand,
					\begin{align*}
						\mcG(\mathbf0,\psi)= \mcG(\nabla\psi,\psi)-\p_{p_i}\mcG(\nabla\psi,\psi)\partial_i\psi +\int_0^1 d\tau\int_0^\tau \partial_i\psi\p_{p_ip_j}\mcG((1-s)\nabla\psi,\psi) \partial_j\psi ds,
					\end{align*}
					and $\mcG(\mathbf0,\psi)=-(Q-\psi)\int_0^1 \p_z\mcG(\mathbf0,\psi+s(Q-\psi)) ds$, which follows from  $\mcG(\mathbf0,Q)=0$. Thus 
					\begin{align*}
						&\int_{\{\psi+\vartheta\eta\geq Q\}}\int_0^1 d\tau \int_0^\tau \partial_i\psi\p_{p_ip_j}\mcG((1-s)\nabla\psi,\psi)\partial_j\psi ds \\
						=&\int_{\{\psi+\vartheta\eta\geq Q\}} (\p_{p_i}\mcG(\nabla\psi,\psi)\partial_i\psi
						-\mcG(\nabla\psi,\psi))\\ &-\int_{\{\psi+\vartheta\eta\geq Q\}}(Q-\psi)\int_0^1 \p_z\mcG(\mathbf 0,\psi+s(Q-\psi)) ds.
					\end{align*}
					Note that $\tilde\psi_{\vartheta}\leq \vartheta\eta$. Combining the above inequalities and using the $L^\infty$-bound for $\p_z\mcG$ in \eqref{eq:upper_pzzG} (where $\epsilon<1/4$) yield
					\begin{equation}\label{eq113.5}
						\begin{aligned}
							&\int_{\mathcal D}  (\mcG(\nabla\psi_{\vartheta,*}, \psi_{\vartheta,*})-\mcG(\nabla\psi,\psi)) \\
							\leq &\int_{\mcD} \p_{p_i}\mcG(\nabla\psi,\psi)\partial_i\tilde\psi_{\vartheta}
							+\mfb^*\vartheta^2\int_{\{\psi+\vartheta\eta<Q\}}|\nabla\eta|^2\\ 
							&+\int_{\{\psi+\vartheta\eta\geq Q\}} (\p_{p_i}\mcG(\nabla\psi,\psi)\partial_i\psi-\mcG(\nabla\psi,\psi))+2\delta'\vartheta\int_{\mcD}\eta.
						\end{aligned}
					\end{equation}
					Using the equation \eqref{basic_eq_psi} in $\mcD \cap \{\psi<Q\}$,  $\tilde\psi_{\vartheta}%=\psi_{\vartheta,*}-\psi
					=0$ in $\{\psi=Q\}$ and $\tilde\psi_{\vartheta}\leq \vartheta\eta$, one has
					\begin{equation*}
						\begin{aligned}
							\int_{\mcD} \p_{p_i}\mcG(\nabla\psi,\psi)\partial_i\tilde\psi_{\vartheta}= -\int_{\mcD} \p_z \mcG(\nabla\psi,\psi)\tilde\psi_{\vartheta}\leq \delta'\vartheta\int_{\mcD}\eta.
						\end{aligned}
					\end{equation*}
					Thus, in view of \eqref{eq:Phi_1}, the estimate \eqref{eq113.5} together with \eqref{eq:2d_lambda} gives 
					\begin{align*}
						\int_{\{Q-\vartheta\eta\leq\psi<Q\}}\Phi(\Lambda^2,Q) \leq \int_{\{Q-\vartheta\eta\leq\psi<Q\}} \Phi(|\nabla\psi|^2,\psi) + \mfb^* \vartheta^2\int_{\{\psi+\vartheta\eta<Q\}}|\nabla\eta|^2 + 3\delta'\vartheta\int_{\mcD}\eta.
				\end{align*}}
				Now for $0<r<\hat r<R$ with $R$ such that $B_{2R}\subset K$, let
				\begin{equation*}
					\eta(x):= 1\text{ in } B_{r},\quad \eta(x):=\frac{\ln(\frac{\hat r}{|x|})}{\ln (\frac{\hat r}{r})} \text{ in } B_{\hat r}\backslash B_{r},\quad \eta(x):=0 \text{ in } \R^2\backslash B_{\hat r}.
				\end{equation*}
				Let $\vartheta:=C_0\Lambda r$, where $C_0=C_0(\gamma,\epsilon,\bar u)>0$ is a constant such that $Q-\psi\leq C_0\Lambda r=\vartheta\eta$ in $B_{r}$ (cf. Lemma \ref{lem:linear}). Then one has 
				\begin{equation*}\label{eq:3}
					\begin{aligned}
						&\int_{B_{r}\cap \{\psi<Q\}}\left(\Phi(\Lambda^2,Q)-\Phi(|\nabla\psi|^2,\psi)\right)^+\\
						%\leq&\int_{B_r\cap \{\psi<Q\}}\left(\Phi(|\nabla\psi|^2,\psi)-\Phi(\Lambda^2,Q)\right)^+\\
						%&+\int_{(B_{\hat r}\backslash B_r)\cap \{Q-\vartheta\eta\leq\psi<Q\}}\left(\Phi(|\nabla\psi|^2,\psi)-\Phi(|\Lambda|^2,Q)\right)^+
						%+ \mfb^* \vartheta^2\int_{B_{\hat r}\backslash B_r}|\nabla\eta|^2 + 2\delta\vartheta\int_{B_{\hat r}}\eta\\
						\leq &\int_{B_{\hat r}\cap \{\psi+\vartheta\eta\geq Q\}}\left(\Phi(|\nabla\psi|^2,\psi)-\Phi(|\Lambda|^2,Q)\right)^+ 
						+ C\mfb^\ast\Lambda^2r^2\left(\ln(\hat r/r)\right)^{-1} + C\Lambda r\hat r^2,
					\end{aligned}
				\end{equation*}
				for some $C=C(\gamma,\epsilon,\bar u)>0$. Note that it follows from \eqref{eq:Phi_1} that 
				\begin{align*}
					\p_t \Phi(|\bp|^2, z)=\frac{1}{2} %\sum_{i,j}
					\frac{p_ip_j}{|\bp|^2}\p_{p_ip_j}\mcG(\bp,z),\quad \p_z\Phi(|\bp|^2,z)=%\sum_i
					p_i\p_{p_iz}\mcG(\bp,z)-\p_z\mcG(\bp,z).
				\end{align*} 
				Thus it follows from \eqref{eq:convex0} and \eqref{eq:upper_pzzG} that 
				$$\frac12\mfb_\ast\leq\p_t \Phi(|\bp|^2, z)\leq \frac12\mfb^\ast \quad\text{and}\quad |\p_z\Phi(|\bp|^2, z)|\leq 2\delta'.$$ 
				Hence one gets 
				\begin{equation*}
					\frac12\mfb_\ast(\Lambda^2-|\nabla\psi|^2)^+\leq
					\left(\Phi(\Lambda^2,Q)-\Phi(|\nabla\psi|^2,\psi)\right)^++2\delta'(Q-\psi),
				\end{equation*} 
				and 
				\begin{equation*}\label{eq:2}\begin{aligned}
						\left(\Phi(|\nabla\psi|^2,\psi)-\Phi(\Lambda^2,Q)\right)^+ &\leq \frac12\mfb^\ast \left(|\nabla\psi|^2-\Lambda^2\right)^+ + 2\delta' (Q-\psi)\\
						&\leq C\mfb^\ast\Lambda^2\left(\frac{\hat r}{R}\right)^\alpha+ 2\delta'(Q-\psi),
				\end{aligned}\end{equation*} 
				for some $C=C(\gamma,\epsilon,\bar u,\frac{\Lambda}{Q},K)>0$, where Remark \ref{rmk_gradu} is used to get the last inequality.
				%Applying Remark \ref{rmk_gradu} to the first term on the right side of \eqref{eq:2} and using $Q-\psi\leq \vartheta \eta\leq C_0r$, one has
				%\[\left(\Phi(|\nabla\psi|^2,\psi)-\Phi(|\Lambda|^2,Q)\right)^+ \leq  \bar CC\left(\frac{\hat r}{R}\right)^\alpha + \bar CC_0r,\]
				%for some $C=C(\lambda,\mfb^*,\mfb_*,K)>0$.
				Note that $\mfb_\ast,\,\mfb^\ast\sim Q^{-1}$ by the expressions of $\mfb_\ast$ and $\mfb^\ast$ in Proposition \ref{prop:assumption} and \eqref{label_7}. 
				Then combining the above inequalities together and using $Q-\psi\leq \vartheta \eta\leq C_0\Lambda r$ yield 
				\begin{align*}
					\frac{1}{r^2}\int\limits_{B_{r}\cap \{\psi<Q\}}(\Lambda^2-|\nabla\psi|^2)^+ \leq C\Lambda^2 \left(\frac{\hat r}{r}\right)^2 \left(\frac{\hat r}{R}\right)^\alpha +\frac{C\Lambda^2}{\ln(\hat r/r)}+\frac{C\Lambda^2\hat r^2}{r},
				\end{align*}
				where $C=C(\gamma,\epsilon,\bar u,\frac{\Lambda}{Q}, K)$. 
				Taking $R=\sqrt{\hat{r}}$ and $\hat{r}=r^{s}$ with some $s\in (\frac{2}{2+\frac{\alpha}{2}}, 1)$ we obtain the desired estimate.
			\end{proof}
			
			An immediate consequence of Proposition \ref{prop:gradu2} is the following corollary.
			\begin{cor}\label{cor:plane_blowup}
				Let $\psi$ be a minimizer of  \eqref{eq:energy_new} over the admissible set \eqref{admissibleK},  where the function $\mcG$ satisfies \eqref{eq:convex}--\eqref{eq:com_energy} in Proposition \ref{prop:assumption} and  \eqref{eq:radial}. Let $\bar x\in \Gamma_\psi$ be any free boundary point. Then every blowup limit of $\psi$ at $\bar x$ 
				{is a half plane solution with slope $\Lambda$ in a neighborhood of the origin.}
			\end{cor}
			The proof for Corollary \ref{cor:plane_blowup} is the same as that for \cite[Corollary 4.4]{ACF84}, which we do not repeat here.
			
			With Corollary \ref{cor:plane_blowup} at hand, one can make use of the improvement of flatness arguments, which is rather standard now,  to show that the free boundary is locally a $C^{1,\alpha}$ graph for some $\alpha\in (0,1)$. For the proof we refer to \cite[Section 5]{ACF84} for the case of minimizers of the energy $\int_{\mcD}\mcG(\nabla \psi) $, and \cite{D11, DFS15} for a different argument which does not rely on energy minimality and also allows to deal with the inhomogeneous terms.

			Higher regularity of the free boundary follows from \cite[Theorem 2]{KN}. For the completeness, we state the result here and omit the proof.
			\begin{prop}
				Let $\psi$ be a minimizer of \eqref{eq:energy_new} over the admissible set \eqref{admissibleK},  where the function $\mcG$ satisfies \eqref{eq:convex}--\eqref{eq:com_energy} in Proposition \ref{prop:assumption} and  \eqref{eq:radial}. Then the free boundary $\Gamma_\psi$ is locally $C^{k+1,\alpha}$ if $\mcG(\bp,z)$ is $C^{k, \alpha}$ in its components, and it is locally real analytic if $\mcG(\bp,z)$ is real analytic.
			\end{prop}
			
			\begin{rmk}
				Under our assumption that $\bar u$ is $C^{1,1}([0,\bar{H}])$, the function $\mathcal{G}(\bp,z)$ is $C^{1,1}$ in its components. Thus in our case we have that the free boundary is locally $C^{2,\alpha}$ for any $\alpha\in (0,1)$. 
			\end{rmk}
			
			%\begin{rmk}
			%	The results in this section hold for general variational problem as long as $\mcG$ satisfies the structure conditions in Proposition \ref{prop:assumption}.
			%\end{rmk}

			%\newpage
			\section{Fine Properties for the free boundary problem}\label{secprop}
			The major property of solutions obtained  in this section is that the streamlines, in particular the free boundary, are graphs. This is one of the key ingredients to justify the equivalence {of} Problem \ref{pb} and Problem \ref{Pb2}.
			
			The key step towards establishing the graph property for the streamlines is to show that the stream function is monotone in the $x_1$ variable. For that we take a specific boundary value on $\p\Omega_{\mu, R}$ for the truncated problem. 
			
			In order to state the dependence of solutions on the parameters clearly and make use of the special structure of the equation \eqref{ELpde}, we resume to use the notations introduced in Section \ref{sec:variational}.
			\begin{center}
				\includegraphics[height=5cm, width=10cm]{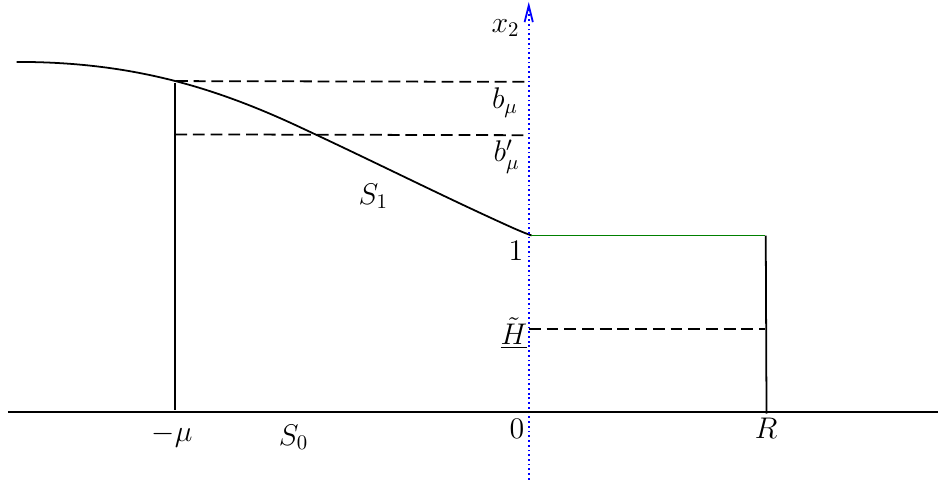}\\
				{\small Figure 3. The truncated domain $\Omega_{\mu, R}$.}
			\end{center}
			
			To describe the boundary data on $\p\Omega_{\mu,R}$, some notations are needed. {Let $s\in (\frac{1}{2},1)$ be a fixed constant and $b_\mu\in(1,\bar H)$ be such that $\Theta(b_\mu)=-\mu$, where $\Theta$ is defined in \eqref{eq:nozzle}. Choose a point $(-\mu, b'_\mu)$ with $1<b'_\mu< b_\mu$ and $k_\mu:=b_\mu-b'_\mu$ small (which is to be determined by Lemmas \ref{lem:bdry_sub_super} and  \ref{lem:linear_bd}). Define $\psi^\dag(x_2)$ as the minimizer of $$\int_{0}^{\tilde{\ubar H}}G_\epsilon(|v'|^2,v)$$ in the admissible set $K^\dag:=\{v\in C^{0,1}([0,\tilde{\ubar H}];\mathbb R)| v(0)=0, v(\tilde{\ubar H})=Q\}$, where $\tilde{\ubar H}\in(0,1)$ is small enough such that $(\psi^\dag)'(\tilde{\ubar H})>\Lambda$.  Properties of these energy minimizers are studied in Appendix (cf. Lemma \ref{lem:variation_sol}). 
				Set
				\begin{equation}\label{eq:bdry_datum}
					\psi^\sharp_{\mu, R}(x_1, x_2):=\left\{
					\begin{aligned}
						&0 && \hbox{if } x_1=-\mu, \,\, 0<x_2<b'_\mu,\\
						&Q\left(\frac{x_2-b'_\mu}{k_\mu}\right)^{1+s} && \hbox{if } x_1=-\mu, \,\, b'_\mu\leq x_2\leq b_\mu,\\
						&\psi^\dag(x_2) &&\hbox{if } x_1=R, \,\, 0<x_2<\tilde{\ubar H},\\
						&Q &&\hbox{if } x_1=R, \,\, x_2\geq \tilde{\ubar H} ,\\
						&0 &&\hbox{if } x_2=0,\\
						&Q &&\hbox{if } (x_1, x_2)\in S_1\cup \left([0,R]\times \{1\}\right).
					\end{aligned}
					\right.
				\end{equation}
				From the expression for $\psi_{\mu,R}^\sharp$ and the properties of $\psi^\dag$ in Lemma \ref{lem:variation_sol}, especially $0\leq \psi^\dag\leq Q$, it is not hard to see that  
				$\psi^\sharp_{\mu, R}$ is continuous and {satisfies} $0\leq \psi^\sharp_{\mu, R}\leq Q$.
				
				\begin{lem}\label{lem:bdry_sub_super}
					The function $\psi^\sharp_{\mu, R}(R,\cdot)$ is a supersolution to \eqref{ELpde} in $\Omega_{\mu, R}$.  Moreover, if $k_\mu$ is sufficiently small depending on $\gamma$ and $\bar u$, then $\psi^\sharp_{\mu, R}(-\mu, \cdot)$ is a subsolution to \eqref{ELpde} in $\Omega_{\mu, R}$. 
				\end{lem}
				\begin{proof}
					Let $\psi(x_1, x_2):=\psi_{\mu, R}^\sharp(R, x_2):\R\times \R_+\rightarrow \R$. {By} \eqref{supportG} $Q$ is a supersolution to \eqref{ELpde}, and by Lemma \ref{lem:variation_sol}(ii) $\psi^{\dag}(x_2)$ is a solution {to \eqref{ELpde}} for $x_2\in(0,\tilde{\ubar H})$ with  $\frac{d}{dx_2}\psi^{\dag}(\tilde{\ubar H})>0$. Then straightforward computations yield
					\[
					\int_{\Omega_{\mu,R}} g_\epsilon(|\nabla \psi|^2, \psi) \nabla \psi \cdot \nabla \eta +\p_z G_\epsilon(|\nabla\psi|^2, \psi)\eta\geq  0 \quad\text{for any } \eta\in C^{\infty}_0(\Omega_{\mu,R}),\, \eta\geq 0.
					\]
					This means that $\psi^\sharp_{\mu,R}(R,\cdot)$ is a supersolution in $\Omega_{\mu, R}$.
					
					Let $\varphi(x_2):=Q\left(\frac{x_2-b'_\mu}{k_\mu}\right)^{1+s}$. Note that $\psi^\sharp_{\mu,R}({-\mu}, x_2)=\varphi(x_2)$ when $x_2\in [b'_\mu,b_\mu]$ and $\psi^\sharp_{\mu,R}({-\mu},x_2)=0$ when $x_2\in [0,b'_\mu]$.  We claim that $\varphi$ is a subsolution to \eqref{ELpde} on $(b'_\mu,b_\mu)$, i.e.
					\begin{equation}\label{eq:subsol_bdry}
						(g_\epsilon(|\varphi'|^2, \varphi) \varphi')'-\p_z G_\epsilon(|\varphi'|^2, \varphi)> 0 \quad\text{on } (b'_\mu, b_\mu).
					\end{equation}
					Note that $\psi^\sharp_{\mu,R}(-\mu,\cdot)\in C^{1,s}_{loc}(\R)$ and $0$ is a solution 
					to \eqref{ELpde} by \eqref{supportG}. We 
					can conclude from \eqref{eq:subsol_bdry} that $\psi^\sharp_{\mu, R}(-\mu,\cdot)$ is a subsolution to \eqref{ELpde} in $\R\times (0,b_\mu)\supseteq \Omega_{\mu, R}$. Hence it remains to verify \eqref{eq:subsol_bdry}. 
					{The straightforward computations give
						\begin{equation*}
							\begin{split}
								(g_\epsilon(|\varphi'|^2, \varphi) \varphi')'=&\left(g_\epsilon + 2 \p_t g_\epsilon (\varphi')^2 \right)\varphi'' +\p_z g_\epsilon (\varphi')^2 \\
								=&\left(g_\epsilon + 2 \p_t g_\epsilon (\varphi')^2\right)\frac{(1+s)sQ}{k_\mu^2} \left(\frac{x_2-b'_\mu}{k_\mu}\right)^{-1+s} + \p_z g_\epsilon (\varphi')^2.
							\end{split}
						\end{equation*}
						It follows from \eqref{beta1eps}, \eqref{eq:dzgm_dtgm}--\eqref{lable_2} and \eqref{eq:upper_lower_g}, as well as \eqref{eq:upper_pzzG} that 
						\begin{align*}
							g_\epsilon+2\p_t g_\epsilon |\varphi'|^2\geq C_\gamma B_*^{-\frac1{\gamma-1}} >0,\quad |\p_zg_\epsilon|\leq C_\gamma\kappa_0 B_*^{\frac1{\gamma-1}}\p_t g_\epsilon,\quad 
							|\p_z G_\epsilon|\leq C_\gamma\kappa_0,
						\end{align*}
						where $C_\gamma$ is a constant  depending only on $\gamma$, and $\kappa_0$ is defined in \eqref{eq:u0_eps0} depending only on $\bar u$. 
						Therefore, if $k_\mu=k_\mu(\gamma,\bar u)$ is sufficiently small}, the function $\varphi$ satisfies \eqref{eq:subsol_bdry} on $(b'_\mu,b_\mu)$. 
				\end{proof}

				The next lemma shows that minimizers of \eqref{eq:mini_truncated} must stay between $\psi^\sharp_{\mu, R}(-\mu,\cdot)$ and $\psi^\sharp_{\mu, R}(R,\cdot)$. This in particular implies that for a minimizer $\psi$,  $\partial_{x_1} \psi$ is nonnegative at the two sides $x_1=-\mu$ and $x_1=R$ of $\Omega_{\mu,R}$.
				
				\begin{lem}\label{lem:linear_bd}
					Let $\psi$ be a minimizer of the truncated problem \eqref{eq:mini_truncated} in $\Omega_{\mu, R}$ with the boundary value $\psi^\sharp_{\mu, R}$ defined in \eqref{eq:bdry_datum}. 
					If $k_\mu$ is sufficiently small depending on $\gamma,\,\bar u$, and $\epsilon$, then  
					we have
					\begin{equation*}
						{\psi^\sharp_{\mu, R} (-\mu, x_2) < \psi(x_1,x_2) 
							\leq \psi^\sharp_{\mu, R}(R, x_2)} 
						\quad\text{for all } (x_1,x_2)\in \Omega_{\mu, R}.
					\end{equation*}
				\end{lem}
				\begin{proof}
					The proof is divided into two steps.
					
					{\it Step 1. Lower bound.} Let $k_\mu$ be  sufficiently small such that in $D:=\Omega_{\mu, R}\cap \{b'_\mu< x_2< b_\mu\}$, $\psi^\sharp_{\mu,R}(-\mu, \cdot)$ is a strict subsolution (cf. \eqref{eq:subsol_bdry}) and the comparison principle (cf. Lemma \ref{lem:comparison}) holds. Note that $\psi^\sharp_{\mu,R}(-\mu, \cdot)\leq \psi$ on $\p D$. 
					%and $\psi^\sharp_{\mu,R}(-\mu, \cdot)<\psi$ on $S_1\cap\{b'_\mu< x_2< b_\mu\}\subset\p D$. 
					Thus we infer from the comparison principle and the strong maximum principle (cf. \cite[Theorem 3.5]{GT}) that $\psi^\sharp_{\mu, R}(-\mu, x_2)< \psi(x_1,x_2)$ in $D$. Since $\psi^\sharp_{\mu, R}(-\mu,x_2)\equiv 0$ if $x_2\in [0,b'_\mu]$, one has $\psi^\sharp_{\mu, R}(-\mu, x_2)< \psi(x_1,x_2)$ in $\Omega_{\mu,R}$.
					
					{\it Step 2. Upper bound.} {Denote $\Psi^{(\tau)}(x_2):= \psi^\sharp_{\mu, R}(R,x_2+\tau)$ for $\tau>0$. It follows from the proof of Lemma \ref{lem:bdry_sub_super} that $\Psi^{(\tau)}$ is also a supersolution in $\Omega_{\mu,R}$.
						Furthermore, we have $\psi\leq Q=\Psi^{(\tau)}$ if $\tau>\tilde{\ubar H}$.
						Let
						\[
						\tau_*:=\inf\{ \tau| \psi\leq \Psi^{(\tau)}\}
						\]
						be the smallest constant such that the graph of $\psi$ touches that of $\Psi^{(\tau)}$. We claim $\tau_*=0$. Suppose by contradiction that $\tau_*>0$. By the boundary condition of $\psi$ and the strong maximum principle, the graphs of $\Psi^{(\tau_*)}$ and $\psi$ must touch at a free boundary point, say $(\bar{x}_1,\tilde{\ubar H}-\tau_*)$. Then by the Hopf lemma (cf. %\cite{R18}
						\cite[Lemma 3.4]{GT}), one has
						\begin{equation*}
							\Lambda = |\nabla \psi (\bar x_1,\tilde{\ubar H}-\tau_*)| > |(\Psi^{(\tau_*)})^\prime(\tilde{\ubar H}-\tau_*)|%=|(\psi^\dag)^\prime(\tilde{\ubar H})|
							>\Lambda,
						\end{equation*}
						where the last inequality is due to the choice of $\tilde{\ubar H}$. This contradiction implies that $\tau_*=0$.} Hence the proof of the lemma is completed.
					%On the other hand,
					%\begin{equation*}
					%\rho_1u_1((x_0)_2)\geq \rho_1 u_1(H_1)=\Lambda,
					%\end{equation*}
					%a contradiction. To see the last inequality, we make use of the monotonicity assumption on $\bar u$. Indeed, the conservation of $\omega/\rho$ along the stream line yields
					%\begin{align*}
					%\frac{u'_0(\mfh(\psi_{\mu,R}(R,x_2)))}{\bar\rho}=\frac{u'_1(x_2)}{\rho_1}.
					%\end{align*}
					%Thus the assumption $u'_0\leq 0$ implies $u'_1\leq 0$. Moreover, at the touching point $\psi(x_0)=Q=\psi_{\mu, R}(R,(x_0)_2)+\tau$ with $\tau>0$, which implies that $(x_0)_2<H_1$ (since $\p_{x_2}\psi_{\mu,R}(R,x_2)=\rho_1 u_1(x_2)>0$ and $\psi_{\mu,R}(R,H_1)=Q$). The inequality thus follows.
				\end{proof}

				The next proposition concerns the uniqueness of the minimizer and the monotonicity of the minimizer with respect to $x_1$.
				\begin{prop}\label{prop:mono}
					Under the same assumptions as in Lemma \ref{lem:linear_bd}, $\psi$ is the unique minimizer of \eqref{eq:mini_truncated}. Furthermore, $\p_{x_1}\psi\geq 0$ in $\Omega_{\mu,R}$.
				\end{prop}
				\begin{proof}
					{\it Step 1. Nonnegativity of $\partial_{x_1}\psi$.}
					%The proof is similar as Theorem 4.2 in \cite{ACF85}. We only point out the differences.
					For $k>0$ small, denote
					$$\psi^{(k)}(x_1,x_2):=\psi(x_1-k, x_2), \quad \text{for}\,\, (x_1,x_2)\in \Omega^{k}_{\mu,R}:=\{(x_1+k,x_2)| (x_1,x_2)\in \Omega_{\mu,R}\}.$$
					Since $\psi$ is a minimizer for $J^\epsilon_{\mu,R,\Lambda}$ over $\mathcal{K}_{\psi^\sharp_{\mu,R}}$, then $\psi^{(k)}$ is a minimizer for $J^{\epsilon,k}_{\mu, R,\Lambda}$ in $\mathcal{K}_{\psi^{\sharp,(k)}_{\mu,R}}$. Here $J^{\epsilon, k}_{\mu, R,\Lambda}$ is defined by performing the corresponding translation for $J_{\mu, R,\Lambda}^\epsilon$. For ease of notations, in the rest of the proof we denote
					$$J:=J^{\epsilon}_{\mu, R,\Lambda},\quad 
					J^k:=J^{\epsilon, k}_{\mu, R,\Lambda}\quad\text{and}\quad \mathcal{K}:=\mathcal{K}_{\psi^\sharp_{\mu,R}},\quad\mathcal{K}^k:= \mathcal{K}_{\psi^{\sharp,(k)}_{\mu,R}}.$$  
					%$J^{\epsilon, k}_{\mu, R,\Lambda}$ and $J^{\epsilon}_{\mu, R,\Lambda}$ by $J^k $ and $J$, respectively, in the rest of the proof.  Below we drop the dependence on $\epsilon,\mu,R,\Lambda$ in the energy functional and write $\mathcal{K}:=\mathcal{K}_{\psi^\sharp_{\mu,R}}$, $\mathcal{K}^k:= \mathcal{K}_{\psi^{\sharp,(k)}_{\mu,R}}$ for simplicity. 
					By Lemma \ref{lem:linear_bd}, one has
					\begin{equation*}
						\min\{\psi^{(k)}, \psi\} \in \mathcal{K}^{k}\quad\text{and}\quad  \max\{\psi^{(k)}, \psi\}\in \mathcal{K}.
					\end{equation*}
					Since the energy functional  $J$ depends only on $\nabla\psi$ and $\psi$, it is easy to verify that
					\begin{equation*}
						J^k(\psi^{(k)}) + J(\psi) = J^{k}(\min\{\psi^{(k)}, \psi\}) + J(\max\{\psi^{(k)}, \psi\}).
					\end{equation*}
					This, together with the minimality of $\psi^{(k)}$ and $\psi$, gives
					\begin{equation}\label{eq:maxmin}
						J(\psi)=J(\max\{\psi^{(k)},\psi\}).
					\end{equation}

					{Our aim is to show that $\psi^{(k)}\leq\psi$ in} $\Omega_{\mu,R}\cap \Omega_{\mu,R}^{(k)}$. By Lemma \ref{lem:linear_bd} and the continuity of $\psi$, one has $\psi^{(k)}< \psi$ near $\{x_1=-\mu+k\}\cap \Omega_{\mu,R}$. If the strict inequality was not true in the connected component of $\Omega_{\mu,R}\cap \{\psi<Q\}$ containing $\{x_1=-\mu+k\}$, then one would find a ball $B\Subset \Omega_{\mu,R}$ and $\bar x\in \p B$ with $\psi(\bar x)<Q$ such that
					\begin{equation*}
						\psi^{(k)}<\psi \ \text{ in } B, \quad \psi^{(k)}(\bar x)=\psi(\bar x).
					\end{equation*}
					By the Hopf lemma
					%(cf. \cite[Lemma 3.4]{GT}), 
					one has
					\begin{equation*}
						\lim_{x\rightarrow \bar x, \ x\in B}\p_\nu (\psi^{(k)}-\psi)(x)>0,
					\end{equation*}
					where $\nu $ is the outer unit normal of $\p B$ at $\bar x$. Thus there exists a curve through $\bar x$ such that along this curve $\max\{\psi^{(k)},\psi\}$ is not $C^1$ at $\bar x$. On the other side, since $\max\{\psi^{(k)},\psi\}$ is a minimizer by \eqref{eq:maxmin} and $\max\{\psi^{(k)},\psi\}(\bar x)<Q$, then by the interior regularity for the minimizers, {$\max\{\psi^{(k)},\psi\}$ is $C^{2,\alpha}$ around} $\bar x$. This is however a contradiction.
					
					Note that every component $\{\psi<Q\}$ has to touch $\p \Omega_{\mu,R}$ (otherwise it would violate the strong maximum principle). Moreover, from the construction of the boundary datum $\psi^\sharp_{\mu,R}$ one has that $\{\psi<Q \}\cap \p\Omega_{\mu,R}$ is a connected arc.  Hence $\{\psi<Q\}\cap \Omega_{\mu,R}$ consists only one component. Consequently, $\psi^{(k)}<\psi$ in $\Omega_{\mu,R}\cap \{\psi< Q\}$. Taking $k\rightarrow 0+$ yields $\p_{x_1}\psi\geq 0$.
					
					{\it  Step 2.  Uniqueness.} Assume that $\psi_1$ and $\psi_2$ are two minimizers with the same boundary data. Let
					$ \psi_1^{(k)}(x_1, x_2):=\psi_1(x_1-k,x_2)$ for $k>0$.
					The same argument as above gives that $\psi_1^{(k)}\leq \psi_2$ in $\Omega_{\mu,R}$. Taking $k\rightarrow 0+$ yields $\psi_1\leq \psi_2$. Similarly, it holds that $\psi_2\leq \psi_1$. Thus one has $\psi_1=\psi_2$.
				\end{proof}
				
				An important consequence of the property $\p_{x_1}\psi\geq 0$ obtained from Proposition \ref{prop:mono} is that, 
				{the free boundary is an $x_2$-graph, which is proved in Proposition \ref{thm:graph} below. 
					Let
					\begin{align}\label{eq:free_bdry_defi}
						\Gamma_{\mu,R,\Lambda}:=\p\{\psi <Q\}\cap \{(x_1,x_2)| (x_1,x_2)\in (-\mu, R]\times (0,1)\}
					\end{align}
					denote those free boundary points which lie strictly below $\{x_2=1\}$. %Note that if $\tilde{\ubar H}\geq 1$, then $\Gamma_{\mu, R,\Lambda}=\emptyset$.
					In view of the boundary value $\psi_{\mu,R}^\sharp(R,\cdot)$ in \eqref{eq:bdry_datum}, $\Gamma_{\mu, R,\Lambda}$ is not empty.}
				
				Before we prove that the free boundary is a graph, let us digress for the proof of the following non-oscillation property.

				\begin{lem}[Non-oscillation]
					\label{lem:nonoscillation}
					Let $D$ be a domain in $\Omega_{\mu, R}\cap \{\psi<Q\}$ bounded by two disjoint arcs $\Gamma_1$ and $\Gamma_1'$ of the free boundary and by $\{x_1=k_1\}$ and $\{x_1=k_2\}$. Suppose that $\Gamma_1$ ($\Gamma_1'$) lies in $\{k_1<x_1<k_2\}$ with endpoints $(k_1, l_1)$ and $(k_2,l_2)$ ($(k_1, l_1')$ and $(k_2,l_2')$). Suppose that $\dist((0,1), D)\geq c_0$ for some $c_0>0$ and $|k_1-k_2|\leq 1$.  Then
					\begin{equation}\label{eq:osc0}
						|k_1-k_2|\leq C \max\{|l_1-l_1'|, |l_2-l_2'|\}
					\end{equation}
					for some $C$ depending on $\gamma,\epsilon, \bar u, \frac{\Lambda}{Q},c_0$.
				\end{lem}
				\begin{center}
					\includegraphics[height=5cm, width=10cm]{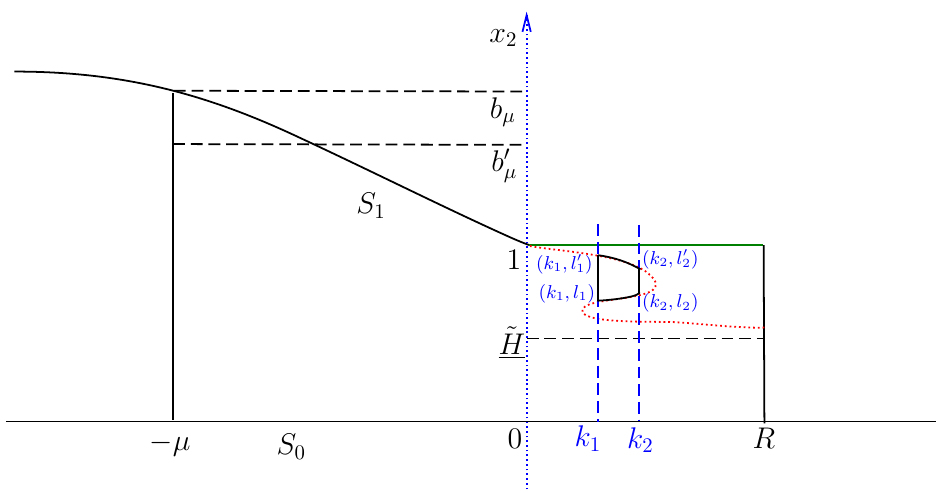}\\
					{\small Figure 4.}
				\end{center}
				\begin{proof}
					The proof is inspired by  \cite[Lemma 4.4]{ACF85}. Note that we can assume $\max\{|l_1-l'_1|, |l_2-l'_2|\}\leq 1$, because otherwise the conclusion is trivial. Integrating the equation
					\eqref{ELpde} in $D$ gives
					\begin{equation*}
						%\label{eq:osc1}
						\begin{aligned}
							&\,\, \int_{\Gamma_1\cup \Gamma_1'} g_\epsilon(|\nabla\psi|^2, \psi)\p_\nu \psi + \int_{\p D\cap (\{x_1=k_1\}\cup \{x_1=k_2\})} g_\epsilon(|\nabla\psi|^2, \psi)\p_\nu \psi \\
							=&\,\, \int_{D} \p_z G_\epsilon(|\nabla\psi|^2,\psi).
						\end{aligned}
					\end{equation*}
					Since $\p_\nu\psi= \Lambda$ along $\Gamma_1$ and $\Gamma'_1$, and since \eqref{eq:g} and Remark \ref{rmk:Q}, there is a constant  $\bar C=\bar C(\gamma, \bar u)>0$ such that
					\begin{align*}
						%\label{eq:osc2}
						\quad \int_{\Gamma_1\cup \Gamma_1'} g_\epsilon(|\nabla\psi|^2, \psi)\p_\nu \psi \geq \bar C \frac{\Lambda}{Q} (\mathcal{H}^1(\Gamma_1)+\mathcal{H}^1(\Gamma_1')).
					\end{align*}
					Applying the interior Lipschitz regularity (Proposition \ref{thm:Lipschitz}) and the boundary estimate for $\psi$ (cf. Remark \ref{rmk:up_to_bdry}) yields
						\begin{equation*}
							%\label{eq:osc3}
							\left |\int_{\p D\cap (\{x_1=k_1\}\cup \{x_1=k_2\})} g_\epsilon (|\nabla\psi|^2, \psi)\p_\nu \psi\right|\leq C\frac{\Lambda}{Q} \max\{|l_1-l_1'|, |l_2-l_2'|\},
						\end{equation*}
						where $C$ depends on $\gamma,\epsilon, \bar u, \frac{\Lambda}{Q}, c_0$. 
						By \eqref{eq:upper_pzzG} and Young's inequality one has 
						%, Remark \ref{rmk:up_to_bdry}, and Young's inequality 
						\begin{equation*}
							\left|\int_D\p_z G_\epsilon(|\nabla\psi|^2,\psi)\right|\leq \delta'|D|\leq \frac{2(\delta')^2}{\bar C}\frac{Q}{\Lambda} \max\{|l_1-l_1'|, |l_2-l_2'|\}+\frac{\bar C}{2}\frac{\Lambda}{Q}|k_1-k_2|.
						\end{equation*}
						Combining the above estimates together and using the simple relation $|k_1-k_2|\leq \mathcal{H}^1(\Gamma_1)+\mathcal{H}^1(\Gamma'_1)$ {we get \eqref{eq:osc0}.} 
						%\begin{align*}
						%\tre{|k_1-k_2|\leq C\max\{|l_1-l'_1|,|l_2-l'_2|\},}
						%\end{align*}
						%where $C=C(\gamma,\epsilon, \bar u,\frac{\Lambda}{Q},c_0)$.
				\end{proof}
				
				The following result on the uniqueness of the Cauchy problem for the elliptic equations,  plays a crucial role in the proof for the graph property of the free boundary as well as the study of the various blowup limits.
				
				\begin{lem}\label{lem:uni}
					Let $D$ be a bounded connected piecewise $C^{2,\alpha}$ domain in $\R^2$, and let $\Sigma$ be a nonempty open set of $\p D$ which is contained in a $C^{2,\alpha}$ boundary piece. Let $\psi, \tilde{\psi}\in H^{1}_{loc}(D)$ be two solutions to the Cauchy problem
					\begin{equation*}
						\left\{
						\begin{aligned}
							&\partial_{x_i}( \p_{p_i}\mathcal{G}(\nabla\psi, \psi))-\mathcal{W}(\nabla\psi, \psi)=0 &\text{ in } D,\\
							&\psi=Q, \ \p_{\nu }\psi =\Lambda &\text{ on } \Sigma,
						\end{aligned}
						\right.
					\end{equation*}
					where $Q$ and $\Lambda$ are constants. Assume that $\mathcal{G}$ is uniformly elliptic (cf. \eqref{eq:convex}--\eqref{eq:convex0}) and that $\p_{p_ip_jp_k}\mathcal{G}$, $\p_{p_ip_jz} \mathcal{G}$, {$\p_{p_i z}\mathcal{G}$, 
					$\p_{p_i}\mathcal{W}$,} $\p_z\mathcal{W}\in L^\infty$. 
					Then $\tilde{\psi}=\psi$.
				\end{lem}
				\begin{proof}
					The proof relies on the unique continuation property for the uniformly elliptic equations with Lipschitz coefficients. First, it follows from  
					{the boundary regularity  for the elliptic equation (\cite[Lemma 1.7]{ACF84}) that $\psi, \tilde\psi\in C^{2,\alpha}(D\cup\Sigma)$.} 
					Let $\varphi:=\tilde\psi-\psi$. Then $\varphi$ solves
					\begin{equation*}
						\left\{
						\begin{aligned}
							&\partial_i (\mathfrak a^{ij} \partial_j \varphi)= \tilde{\mfb}^i \partial_i \varphi -\partial_i(\mfb^i\varphi)+\mfc \varphi
							\quad \text{in } D \\
							& \varphi=\partial_\nu\varphi=0\quad \text{on}\,\, \Sigma,
						\end{aligned}
						\right.
					\end{equation*}
					where
					\begin{align*}
						&\mathfrak a^{ij}(x):=\int_0^1 \partial_{p_ip_j} \mathcal{G}(\nabla\tilde\psi(x)+s\nabla(\psi-\tilde\psi)(x), \tilde\psi(x)+s(\psi-\tilde\psi)(x))ds,\\
						&\mfb^i(x):=\int_0^1 \partial_{p_i z} \mathcal{G}(\nabla\tilde\psi(x)+s\nabla(\psi-\tilde\psi)(x), \tilde\psi(x)+s(\psi-\tilde\psi)(x))ds,\\
						&\tilde{\mfb}_i(x):=\int_0^1 \partial_{p_i} \mathcal{W}(\nabla\tilde\psi(x)+ s\nabla(\psi-\tilde\psi)(x),\tilde\psi(x)+s(\psi-\tilde\psi)(x))ds,\\
						&\mfc(x):=\int_0^1 \partial_z \mathcal{W}(\nabla\tilde\psi(x)+ s\nabla(\psi-\tilde\psi)(x),\tilde\psi(x)+s(\psi-\tilde\psi)(x))ds.
					\end{align*}
					The Cauchy data on $\Sigma$ allow to extend $\varphi$ into a solution in a neighborhood $U\supset \Sigma$ by setting $\varphi=0$ outside of $U\cap D$. It follows from the assumptions on $\mathcal{G}$ and $\mathcal{W}$ that $\mfb^i$, $\tilde{\mfb}^i$ and $\mfc\in L^\infty$, 
					{and $\mathfrak a^{ij}$ is uniformly elliptic and $C^{1,\alpha}$ regular.} 
					Thus by the  unique continuation property, cf. for example  {\cite[Theorem 1.1]{KT01},} one has $\varphi \equiv 0$ in $U\cup D$. This means that $\psi=\tilde{\psi}$ in $U$. Since $D$ is connected, using chain of balls argument one has $\psi=\tilde\psi$ in $D$.
				\end{proof}

				Now we are ready to show that the free boundary is a graph.
				\begin{prop}\label{thm:graph}
					Let $\Gamma_{\mu,R,\Lambda}$ be as in \eqref{eq:free_bdry_defi}. 
					{%If $\Gamma_{\mu,R,\Lambda}\neq \emptyset$,
						Then it can be represented as a graph of a function of $x_2$, i.e., there exists a function $\Upsilon_{\mu, R, \Lambda}$ and %a constant 
						$l_{\mu,R,\Lambda}\in[\tilde{\ubar{H}},1)$} such that
					\begin{align*}
						\Gamma_{\mu, R, \Lambda}=\{(x_1, x_2)| x_1=\Upsilon_{\mu,R,\Lambda}(x_2), \,\, x_2\in (l_{\mu,R,\Lambda}, 1)\}.
					\end{align*}
					Furthermore, $\Upsilon_{\mu,R,\Lambda}$ is continuous, $-\mu<\Upsilon_{\mu,R,\Lambda}\leq R$ and $\lim_{x_2\rightarrow 1-}\Upsilon_{\mu,R,\Lambda}(x_2)$ exists.
				\end{prop}
				\begin{proof}
					It follows from
					Proposition \ref{prop:mono} that $\p_{x_1}\psi\geq 0$. Hence there is a function $\Upsilon_{\mu,R,\Lambda}(x_2)$ with values in $[-\mu, R]$ such that for any $x_2\in (0,1)$,
					\begin{equation*}
						\psi(x_1,x_2)<Q \text{ if and only if } x_1<\Upsilon_{\mu,R,\Lambda}(x_2).
					\end{equation*}
					Since $\psi(-\mu,x_2)=0$ for $x_2\in [0,b'_\mu]$ with $b'_\mu>1$, then $\Upsilon_{\mu,R,\Lambda}(x_2)>-\mu$.
					
					It follows from Lemma \ref{lem:linear_bd} that 
					{one has
						\begin{equation*}
							\psi(x_1, x_2)\leq \psi_{\mu, R}^{\sharp}(R,x_2)<Q \quad\text{in } \R\times [0,\tilde{\ubar{H}}).
						\end{equation*}
						This implies that there is no free boundary points in the strip $\R\times [0,\tilde{\ubar{H}})$. Thus there is a constant $l_{\mu,R,\Lambda}\in [\tilde{\ubar{H}},1)$} such that the free boundary is a graph $x_1=\Upsilon_{\mu,R,\Lambda}(x_2)$ for $x_2\in (l_{\mu,R,\Lambda},1)$.
					
					Finally, one can also show that $\Upsilon_{\mu,R,\Lambda}$ is continuous. Clearly, $\Upsilon_{\mu,R,\Lambda}$ is lower semi-continuous. Suppose there is a jump point $\bar x_2$ such that
					\[
					\lim_{x_2\rightarrow \bar x_2}\Upsilon_{\mu,R,\Lambda}(x_2)=\Upsilon_{\mu, R,\Lambda}(\bar x_2)+\tau
					\]
					for some $\tau>0$. Let $\bar x_1:=\Upsilon_{\mu,R,\Lambda}(\bar x_2)$. Since $\p_{x_1}\psi\geq 0$, the line segment $[\bar x_1,\bar x_1+\tau]\times \{\bar x_2\}$ belongs to the free boundary. Therefore, one can find an upper or lower neighborhood $U$ of 
					$[\bar x_1, \bar x_1+\tau]\times \{\bar x_2\}$,
					such that $\psi$ solves the quasilinear equation \eqref{ELpde} in $U$, and $\psi=Q$, $\p_{\nu}\psi =\Lambda$ on 
					$[\bar x_1,\bar x_1+\tau]\times \{\bar x_2\}$,  
					where $\nu$ is the {outer} unit normal. Note the Cauchy problem
					\begin{equation}
						\left\{
						\begin{aligned}
							&\nabla\cdot(g_\epsilon(|\nabla \psi|^2, \psi)\nabla \psi) - \p_zG_\epsilon (|\nabla\psi|^2,\psi)=0 &&\text{in } \R\times [0,\bar x_2],\\
							&\psi=Q, \quad \p_{x_2}\psi=\Lambda  &&\text{on } x_2=\bar x_2
						\end{aligned}
						\right.
					\end{equation}
					admits a one-dimensional solution $\psi_{1d}(x)=\psi_{1d}(x_2)$, which is the solution of ODE 
					\begin{equation*}
						\left\{
						\begin{aligned}
							&(g_\epsilon(|\psi_{1d}'|^2, \psi_{1d})\psi_{1d}')' - \p_zG_\epsilon (|\psi_{1d}'|^2,\psi_{1d})=0 &&\text{on } [0,\bar x_2],\\
							&\psi_{1d}=Q, \quad \psi_{1d}'=\Lambda &&\text{at } x_2=\bar x_2.
						\end{aligned}
						\right.
					\end{equation*}
					The existence of the one-dimensional solution follows from the standard Picard-Lindelöf theorem using that $g_\epsilon(t,z)+2t\p_tg_\epsilon(t,z)>0$ (the ellipticity assumption). By the uniqueness of the Cauchy problem (cf. Lemma \ref{lem:uni}), one has $\psi=\psi_{1d}$ in $\Omega_{\mu,R}$. This is however impossible. Thus $\Upsilon_{\mu,R,\Lambda}$ is continuous. The existence of $\lim_{x_2\rightarrow 1-}\Upsilon_{\mu,R,\Lambda}(x_2)$ follows from the non-oscillation lemma, cf. Lemma \ref{lem:nonoscillation}.
				\end{proof}

				\section{Continuous fit and smooth fit}\label{seccont}
				In the first part of this section, we show the continuous fit property: given a nozzle and incoming data $(\bar\rho, \bar u)$ at the upstream, there exists a momentum  $\Lambda>0$ along the free boundary, such that the associated 
				{free boundary attaches the orifice} 
				of the nozzle. Subsequently, the free boundary joins the nozzle wall in a continuous differentiable fashion via a blowup argument. 
				
				\subsection{Continuous fit}\label{sec:cont_fit}
				In this subsection, $\mu, R$, and $\epsilon$ are fixed. 
				%\tre{so we drop the dependence on them}.
				To emphasize the dependence on the parameter $\Lambda$, we let $\psi_\Lambda:=\psi_{\mu, R,\Lambda}$ denote the minimizer of the truncated problem \eqref{eq:mini_truncated}, and  $\Gamma_\Lambda:=\Gamma_{\mu, R, \Lambda}$ be the free boundary defined in \eqref{eq:free_bdry_defi} with
				\begin{equation*}
					\Gamma_\Lambda= \{(x_1, x_2)| x_1=\Upsilon_\Lambda(x_2):=\Upsilon_{ \mu, R,\Lambda}(x_2) \text{ for } x_2\in (l_\Lambda, 1)\text{ with } l_\Lambda: =l_{\mu, R, \Lambda}\}.
				\end{equation*}

				The following lemma shows that $\psi_\Lambda$ and $\Upsilon_\Lambda$ depend on $\Lambda$ continuously.
				
				\begin{lem}\label{lem:conv}
					If $\Lambda_n\rightarrow \Lambda$, then $\psi_{\Lambda_n}\rightarrow \psi_\Lambda$ uniformly in $\Omega_{\mu, R}$ and $\Upsilon_{\Lambda_n}(x_2)\rightarrow \Upsilon_{\Lambda}(x_2)$ for each $x_2\in (l_{\Lambda}, 1]$.
				\end{lem}
				\begin{proof} The proof consists of three steps.

					{\it Step 1. Convergence of $\psi_{\Lambda_n}$}. This is similar to \cite{Friedman82}, but for completeness we sketch the proof here. For any subsequence $\psi_{\Lambda_n}\rightarrow \psi_{\Lambda}$ weakly in $H^{1}$ and strongly in $C^{0,\alpha}_{loc}$, where $\psi_{\Lambda_n}$ is a minimizer of the truncated problem \eqref{eq:mini_truncated} with $\Lambda$ substituted by $\Lambda_n$, one has $\p\{\psi_{\Lambda_n}<Q\}\rightarrow \p\{\psi_\Lambda<Q\}$ locally in the Hausdorff distance, $\chi_{\{\psi_{\Lambda_n}<Q\}}\rightarrow \chi_{\{\psi_{\Lambda}<Q\}}$ in $L^1_{loc}$ and $\nabla \psi_{\Lambda_n}\rightarrow \nabla \psi_\Lambda$ a.e. (\cite[Lemma 3.6 in Chapter 3]{Friedman82}). Thus $\psi_\Lambda$ is a minimizer of  \eqref{eq:mini_truncated}.  Then it follows from the uniqueness of the minimizer (cf. Proposition \ref{prop:mono}) that $\psi_{\Lambda_n}\to \psi_{\Lambda}$.
					
					{\it Step 2. Convergence of $\Upsilon_{\Lambda_n}$.} For each $x_2\in (l_{\Lambda}, 1)$, it follows from the uniform convergence of $\psi_{\Lambda_n}$ to $\psi_{\Lambda}$ and Proposition \ref{thm:Lipschitz} that $\Upsilon_{\Lambda_n}(x_2)\to \Upsilon_\Lambda(x_2)$. It remains to show $\Upsilon_{\Lambda_n}\rightarrow \Upsilon_\Lambda$ at $x_2=1$, which will be shown by contradiction. Assume that $\lim_{n\rightarrow \infty}\Upsilon_{\Lambda_n}(1)\neq \Upsilon_{\Lambda}(1)$. Let 
					$$\iota:=\min\left\{\max\{\lim_{n\rightarrow\infty}\Upsilon_{\Lambda_n}(1)-\Upsilon_{\Lambda}(1), -1\}, 1\right\}.$$Now there are three cases to be analyzed in detail.
					
					\emph{Case (i): $\iota<0$.}  First, we claim that $\psi_\Lambda$ solves
					\begin{equation}\label{eq:cauchy}
						\begin{cases}
							\nabla\cdot (g(|\nabla \psi|^2,\psi)\nabla \psi) - \p_z G(|\nabla \psi|^2,\psi)=0 &\text{ in } D_{\iota,\sigma_0},\\
							\psi =Q,\quad \lim_{x_2\rightarrow 1-}\p_{x_2} \psi =\Lambda & \text{ on } \overline{D_{\iota,\sigma_0}}\cap \{x_2=1\},
						\end{cases}
					\end{equation}
					where $D_{\iota,\sigma_0}:=(\Upsilon_\Lambda(1)+\frac{3}{4}\iota, \Upsilon_\Lambda(1)+\frac{1}{4}\iota)\times (1-\sigma_0,1)$ for some small $\sigma_0>0$.
					The proof for the claim is given in Step 3. Assume that $\psi_\Lambda$ solves \eqref{eq:cauchy}, then it follows from  Lemma \ref{lem:uni} (cf. also the proof for Proposition \ref{thm:graph}) that  $\psi_\Lambda$ must be  the one-dimensional solution, i.e.  $\psi_\Lambda(x)=\psi_{1d}(x_2)$. This is however impossible, since it contradicts to the fact that $\psi_\Lambda$ solves the jet problem, for which the domain depends on $x_1$ as well.
					
					\emph{Case (ii): $\iota>0$ and $\Upsilon_\Lambda(1)<0$.} Similar as for case (i),  $\psi_\Lambda$ solves
					\begin{equation*}
						\begin{cases}
							\nabla\cdot (g(|\nabla \psi|^2,\psi)\nabla \psi) - \p_z G(|\nabla \psi|^2,\psi)=0 \,\, &\text{ in } D^+_{\iota,\sigma_0}\\
							\psi =Q,\quad \lim_{x_2\rightarrow 1+}\p_{x_2} \psi =\Lambda \,\, &\text{ on } \overline{D^+_{\iota,\sigma_0}}\cap \{x_2=1\},
						\end{cases}
					\end{equation*}
					where $D^+_{\hat\iota,\sigma_0}:=
					{(\Upsilon_\Lambda(1)+\frac{1}{4}\hat\iota,  \Upsilon_\Lambda(1)+\frac{3}{4}\hat\iota)}\times (1,1+\sigma_0)$, $\hat\iota:=\min\{\iota, -\Upsilon_\Lambda(1)\}$,  and $\sigma_0>0$ is a small constant depending on $\iota$ and the nozzle. Using Lemma \ref{lem:uni} again, we have $\psi_\Lambda$ is a 
					one-dimensional solution. This leads to a contradiction.
					
					\emph{Case (iii): $\iota>0$ and $\Upsilon_\Lambda(1)\geq0$.} In this case we will use the non-oscillation lemma (cf. Lemma \ref{lem:nonoscillation}) to get a  contradiction. First, note that Lemma \ref{lem:nonoscillation} also holds if one of the arcs $\Gamma_1$ or $\Gamma'_1$ is on the horizontal boundary $(0,R)\times\{1\}$, and $\lim_{x_2\rightarrow 1-}\p_{x_2}\psi \geq \Lambda$ along $\Gamma_1$ (or $\Gamma'_1$).
					Next, denote $I_\iota:=(\Upsilon_\Lambda(1)+\frac{\iota}{4}, \Upsilon_\Lambda(1)+\frac{3\iota}{4})\times \{1\}$. By considering  the variation 
					$\tau_\vartheta (x) = x + \vartheta \eta(x)$ in the proof for Lemma \ref{lem:fbcondition}, where $\vartheta\geq 0$ and the diffeomorphism $\eta$ satisfies $\eta\cdot e_2\leq 0$ on $I_\iota$ and $\eta =0$ on $\p\Omega_{\mu,R}\backslash \left((\Upsilon_\Lambda(1), \Upsilon_\Lambda(1)+\iota)\times \{1\}\right)$, and using similar computations as in the proof for Lemma \ref{lem:fbcondition}, we have
					% $\Upsilon_{\Lambda_n}(1)\geq \Upsilon_{\Lambda}(1)+3\iota/4$ and
					\begin{equation}\label{eq:cont_fit_grad}
						\lim_{x_2\rightarrow 1-} \p_{x_2}\psi_{n} \geq \Lambda\quad \text{on } I_\iota,
					\end{equation}
					provided that $n$ is sufficiently large. Since $\Upsilon_{\Lambda_n}(1)\rightarrow \Upsilon_{\Lambda}(1)+\iota$ for $\iota>0$, it follows from Proposition \ref{thm:graph} that given any neighborhood of $I_\iota$, there is a sufficiently large $n$ such that the graph of the function $x_1=\Upsilon_{\Lambda_n}(x_2)$ crosses the neighborhood. Applying Lemma \ref{lem:nonoscillation} to the region bounded
					by $x_1=\Upsilon_\Lambda+\frac{\iota}{4}$, $x_1=\Upsilon_\Lambda(1)+\frac{3\iota}{4}$,  $I_\iota$,  and $\{x_1=\Upsilon_{\Lambda_n}(x_2)\}$ then leads to a contradiction to the fact that the neighborhood of $I_\iota$ can be chosen arbitrarily small. The proof for case (iii) is complete.
					
					\emph{Step 3. Proof for the claim that $\psi_\Lambda$ solves \eqref{eq:cauchy} if $\iota<0$}.  It follows from Lemma \ref{lem:nonoscillation} that there exists a $\sigma_0=\sigma_0(\iota)>0$, such that for any open subset $\mcU \Subset D_{\iota,\sigma_0}$ one has $\psi_{\Lambda_n}<Q$ in $\mcU$ for sufficiently large $n$. Thus as $n\to \infty$ the limit  $\psi_\Lambda$ satisfies the  equation in \eqref{eq:cauchy} in $D_{\iota,\sigma_0}$. The Dirichlet boundary condition $\psi_\Lambda=Q$ follows from $\p_{x_1}\psi_\Lambda\geq 0$ and the uniform convergence of $\psi_{\Lambda_n}$ to $\psi_\Lambda$ in $D_{\iota,\sigma_0}$. Thus one  needs only  to show that the Neumann condition holds.
					
					\emph{(a).} We first claim that $\lim_{x_2\rightarrow 1-}\p_{x_2}\psi_\Lambda\geq \Lambda$. For simplicity, denote $\psi_n:=\psi_{\Lambda_n}$ and $\psi:=\psi_{\Lambda}$. Fix $x^{(0)}\in\{x_2=1\}\cap \overline{D_{\iota,\sigma_0}}$ and $r>0$ such that $B_r(x^{(0)})\cap \{x_2<1\}\subset D_{\iota,\sigma_0}$. For any nonnegative function $\eta \in C^\infty_{0}(B_r(x^{(0)}))$, one has
					\begin{equation*}
						%\label{eq:etaLam_n}
						\begin{split}
							\Lambda_n g(\Lambda_n^2,Q)\int_{\p\{\psi_{n}<Q\}} \eta &= \int_{\{\psi_n<Q\}} g(|\nabla\psi_n|^2,\psi_n)\nabla\psi_n\cdot\nabla\eta + \p_z G(|\nabla\psi_n|^2,\psi_n)\eta \\
							&\rightarrow \int_{\{\psi<Q\}} g(|\nabla\psi|^2,\psi)\nabla\psi\cdot\nabla\eta + \p_z G(|\nabla\psi|^2,\psi)\eta,
						\end{split}
					\end{equation*}
					where $\chi_{\{\psi_n<Q\}}\rightarrow \chi_{\{\psi<Q\}}$ in $L^1(B_r(x^{(0)}))$,  $\psi_n\to \psi$ in $C^{0,\alpha}(B_r(x^{(0)}))$, and $\nabla\psi_n\to \nabla\psi$ a.e. have been used. It follows from the lower semi-continuity of the BV-norm that
					\begin{equation*}
						%\label{eq:semiBV}
						\int_{\p\{\psi<Q\}} \eta \leq \liminf_{n\rightarrow \infty}\int_{\p\{\psi_n<Q\}} \eta.
					\end{equation*}
					Combining the above two inequalities together we obtain
					\begin{equation*}
						\Lambda g(\Lambda^2,Q)\int_{\p\{\psi<Q\}} \eta \leq  \int_{\{\psi<Q\}} g(|\nabla\psi|^2,\psi)\nabla\psi\cdot\nabla\eta + \p_z G(|\nabla\psi|^2,\psi)\eta.
					\end{equation*}
					Since $B_r(x^{(0)})\cap \p\{\psi<Q\} =B_r(x^{(0)})\cap \{x_2=1\}$, the solution $\psi$ is $C^{1}$ up to the boundary. Therefore, an  integration by parts  gives
					\begin{equation*}
						\int_{\{\psi<Q\}} g(|\nabla\psi|^2,\psi)\nabla\psi\cdot\nabla\eta + \p_z G(|\nabla\psi|^2,\psi)\eta= \int_{\p\{\psi<Q\}}\p_{x_2}\psi g(|\p_{x_2}\psi|^2,Q) \eta.
					\end{equation*}
					Thus
					\begin{equation*}
						\Lambda g(\Lambda^2,Q)\int_{\p\{\psi<Q\}} \eta \leq \int_{\p\{\psi<Q\}}\p_{x_2}\psi g(|\p_{x_2}\psi|^2,Q) \eta.
					\end{equation*}
					Since $\eta$ is arbitrary and $s\mapsto s g(s^2, Q)$ is monotone increasing for $s\geq 0$,  we have
					\begin{equation*}
						\p_{x_2}\psi \geq \Lambda \quad \text{on } \p\{\psi<Q\}\cap B_r(x^{(0)}).
					\end{equation*}
					
					\emph{(b).} The arguments for $\lim_{x_2\rightarrow 1-}\p_{x_2}\psi \leq \Lambda$ on $\overline{D_{\iota,\sigma_0}}\cap \{x_2=1\}$ are similar as those in \cite[Theorem 6.1 in Chapter 3]{Friedman82}. We only outline the proof here for completeness.
					For each $\bar{x}\in \overline{D_{\iota,\sigma_0}}\cap \{x_2=1\}$ and $r>0$ sufficiently small, by the uniform convergence of $\psi_n$ to $\psi$ as well as the convergence of the free boundaries, one can find a sequence of domains $\{K_n\}$ such that: (i) $\psi_n<Q$  in  $K_n\cap B_r(\bar{x})$,
					and $\p K_n\cap B_r(\bar{x})$  is given by a $C^{1,\alpha}$  graph, which converges to the free boundary $\p\{\psi<Q\}\cap B_r(\bar{x})=\{x_2=1\}\cap B_r(\bar{x})$ in $C^{1,\alpha}$ as $n\rightarrow \infty$; (ii) there is a sequence of points $\{x^{(n)}\}\subset \p K_n\cap B_r(\bar{x})$, such that $\psi_n(x^{(n)})=Q$ for each $n$ and $x^{(n)}\rightarrow \bar{x}$ as $n\rightarrow \infty$.
					Then let $\phi_n$ be the solution to the Dirichlet problem
					\begin{equation*}
						\begin{cases}
							\nabla\cdot (g(|\nabla \phi_n|^2,\phi_n)\nabla \phi_n) - \p_z G(|\nabla \phi_n|^2,\phi_n)=0 \quad\text{in }  K_n\cap B_r(\bar{x}),\\
							\phi_n=  Q \quad \text{on } \p K_n\cap B_r(\bar x),\quad \phi_n=\psi_n \quad \text{ on } \p B_r(\bar{x}) \cap K_n.
						\end{cases}
					\end{equation*}
					By the comparison principle ({Lemma \ref{lem:comparison}}), $\phi_n\geq \psi_n$ in $K_n\cap B_r(\bar{x})$ and thus 
					$$\p_{\nu^{(n)}} \phi_n(x^{(n)})\leq \p_{\nu^{(n)}} \psi_n(x^{(n)})=\Lambda_n,$$ 
					where $\nu^{(n)}:=\nu(x^{(n)})$ is the unit outer normal of $\partial K_n$ at $x^{(n)}$. By the  boundary  estimate for elliptic equations one has $\phi_n\rightarrow \psi$ in $C^{1,\alpha}(B_{r/2}(\bar{x})\cap \overline{K_n})$. This gives 
					$$\p_{x_2} \psi ( \bar{x})=\lim_{n\rightarrow \infty} \p_{\nu^{(n)}} \phi_n(x^{(n)})\leq \lim_{n\rightarrow \infty} \Lambda_n=\Lambda.$$ 
					The proof for \eqref{eq:cauchy} is thus complete.
				\end{proof}
				
				The following lemma is a consequence of the nondegeneracy property of the solution (Lemma \ref{lem:nondeg}) as well as the graph property of the free boundary. 
				\begin{lem}\label{lem:continuous}
					There exists a constant $C=C(\gamma,\epsilon,\bar u)>0$  such that the following statements hold:
					
					\rm{(i)} If $\Lambda\geq {CQ}$, then {%the free boundary $\Gamma_{\Lambda}$ is nonempty and satisfies 
						$\Upsilon_\Lambda(1)<0$.}
					
					\rm{(ii)} If $\Lambda\leq C^{-1}Q$, then $\Upsilon_\Lambda(1)>0$.
				\end{lem}
				\begin{proof}
					The proof is inspired by  \cite[Lemma 5.2]{ACF85}.
					
					\emph{(i).} We claim that $\Lambda/Q$ cannot be too large if $\Upsilon_\Lambda(1)\geq0$. Since $\Gamma_{\Lambda}$ connects $(\Upsilon_\Lambda(1), 1)$ to $(R,l_\Lambda)$ (note that $l_\Lambda<1$), there exist  $x^{(0)}\in \Gamma_{\Lambda}$ and $r'\in (0,R_0)$ ($R_0=R_0(\gamma,\epsilon,\bar u)$ is the constant in Lemma \ref{lem:nondeg} with $r=1/2$) independent of $\Lambda$ or $Q$ such that $B_{r'}(x^{(0)})\subset \Omega_{\mu, R}\cap \{x_2 >\hat\sigma_0/2\}$ or $B_{r'}(x^{(0)})\subset \{x_1>0, x_2>\hat\sigma_0/2\}$ for some small $\hat\sigma_0>0$. If $c_{1/2}\Lambda/Q\leq R_0$,  where $c_{1/2}$ is the constant in Lemma \ref{lem:nondeg} depending on $\gamma,\,\epsilon$ and $\bar u$, then we are done. If $c_{1/2}\Lambda/Q>R_0$, then it follows from the nondegeneracy lemma (Lemma \ref{lem:nondeg}) that
					\begin{align*}
						\frac{Q}{r'}\geq \frac{1}{r'}\left(\frac{1}{|B_{r'}(x^{(0)})|}\int_{B_{r'}(x^{(0)})}|Q-\psi|^2\right)^{1/2}\geq c_{1/2}\Lambda.
					\end{align*}
					%{where $C>0$ depends only on $\gamma,\epsilon$ and $\bar u$.} 
					Since $r'$ is fixed, there is a contradiction if ${\Lambda}/Q$ is large.
					
					%Next if $\Gamma_{\Lambda}=\emptyset$ (i.e. $l_\Lambda\geq 1$), then one can derive a contradiction by taking any $x^{(0)}=(x_1,1)$ with $x_1>0$ and using the linear nondegeneracy for nonnegative solutions along flat boundaries.
					
					\emph{(ii).}  Suppose that $\Upsilon_\Lambda(1)\leq 0$, we will prove that ${\Lambda}/Q$ cannot be too small. Similar to that in \cite{ACF85}, {let 
						\begin{equation*}
							\eta_\vartheta(x_1):=\begin{cases}
								\vartheta \exp\left(\frac{(x_1-\frac{R}{2})^2}{(x_1-\frac{R}{2})^2-\hat\sigma^2}\right) &\text{ if } \left|x_1-\frac{R}{2}\right|< \hat\sigma,\\
								0 &\text{ otherwise},
							\end{cases}
						\end{equation*}
						where $\vartheta>0$, and $\hat\sigma=\hat \sigma(\gamma,\epsilon,\bar u)>0$ small such that the comparison principle holds in $\R\times(0,2\hat\sigma)$ (cf. Lemma \ref{lem:comparison}).}  
					One can increase $\vartheta$ till the first time $\vartheta_0$ such that the graph of $x_2=\eta_{\vartheta_0}(x_1)$ touches the free boundary at some $x^{(0)}\in \Gamma_{\Lambda}$.
					%Note that $(x^{(0)})_2=x^{(0)}\cdot e_2 >\epsilon_0$.
					Let $K:=\Omega_{\mu, R}\cap \{(x_1,x_2)| x_2< \eta_{\vartheta_0}(x_1)\}$ and $\phi$ be a solution to the following problem
					\begin{equation*}
						\begin{cases}
							\nabla\cdot ( g(|\nabla \phi|^2, \phi)\nabla\phi)- \p_z G(|\nabla \phi|^2, \phi)=0 \quad \text{in } K,\\
							\phi =Q \quad \text{on } \{x_2=\eta_{\vartheta_0}(x_1)\},\quad \phi =0 \quad \text{on } \{x_2=0\}.
							%\phi &=\psi_\Lambda \text{ on } \{x_1=-\mu\}\cup \{x_1=R\}.
						\end{cases}
					\end{equation*}
					It follows from {the comparison principle}    that $\phi\geq \psi_\Lambda$ in $K$. Hence
					\begin{align*}
						%\label{estphinu1}
						\p_\nu \phi(x^{(0)})\leq \p_\nu \psi_\Lambda(x^{(0)})=\Lambda.
					\end{align*}
					On the other side, using the linear asymptotics of $\phi$ at $x^{(0)}$, which can be obtained by solving the problem for the blowup limit of $\phi/Q$ at $x^{(0)}$, one has
					\begin{align}\label{estphinu2}
						\p_\nu \phi (x^{(0)})\geq C Q
					\end{align}
					for some $C=C(\gamma,\epsilon,\bar u)>0$. 
					Combining the above two inequalities together gives ${\Lambda}/Q \geq C$. This means ${\Lambda}/Q$ cannot be too small.
				\end{proof}
				
				As a consequence of Lemmas \ref{lem:conv} and \ref{lem:continuous}, one has the following continuous fit property for the free boundary. 
				\begin{cor}\label{cor:cont_fit}
					There exists a $\Lambda=\Lambda(\mu, R)>0$ such that $\Upsilon_\Lambda(1)=0$. Furthermore, 
					$\Lambda$ satisfies $C^{-1} Q \leq \Lambda(\mu, R)\leq C Q$ for positive constant $C$ depending on $\gamma$, $\epsilon$, and $\bar u$, %and $\barH$
					but independent of $\mu$ and $R$.
				\end{cor}

				\subsection{Smooth fit}\label{secsmooth}
				
				With the continuous fit property at hand, one can argue along the same lines as  \cite[Theorem 6.1]{ACF85} to conclude that the free boundary {$\Gamma_\psi:=\partial\{\psi<Q\}\cap\Omega_{\mu,R}$} 
				fits the outlet of the nozzle in a $C^1$ fashion.

				\begin{prop}\label{prop:smooth_fit}
					Let $\psi:=\psi_{\mu,R,\Lambda}$ be the minimizer of the truncated problem \eqref{eq:mini_truncated} with $\Lambda$ as in Corollary \ref{cor:cont_fit}.
					Then $S_1\cup \Gamma_\psi$ is $C^1$ in a neighbourhood of the point $A=(0,1)$, and $\nabla\psi$ is continuous in a $\{\psi<Q\}$-neighborhood of $A$.
				\end{prop}

				\begin{proof}
					%Let $\psi$ be a solution to the free boundary problem. 
					Suppose that $\{x_m\}$ is a sequence of points in $\{\psi<Q\}$ which converges to $A$. Define the rescaled functions
					\begin{align*}
						\psi_{r,x_m}(x):=Q-\frac{Q-\psi(x_m+r x)}{r}, \quad r\in (0,1).
					\end{align*}
					Then $\psi_{r, x_m}$ satisfies the quasilinear equation 
					\begin{align*}
						\nabla\cdot \left(g_\epsilon(|\nabla\psi_{r, x_m}|^2, Q-r(Q-\psi_{r, x_m})) \nabla \psi_{r, x_m}\right)- r\p_z G_\epsilon(|\nabla\psi_{r, x_m}|^2, Q-r(Q-\psi_{r, x_m}))=0 
					\end{align*}
					in $\{\psi_{r,x_m}<Q\}$. 
					%Then there exists a sequence $\{r_m\}$ such that $\psi_{r_m, x_m}\to \psi^\diamondsuit$.
					Since $\p_zG_\epsilon\in L^\infty$, then any blowup limit along a subsequence $\psi_{r_m, x_m}$ satisfies
					\begin{align}\label{eq:limit_A}
						\nabla \cdot (g_\epsilon(|\nabla\psi|^2 , Q)) \nabla \psi)=0.
					\end{align}
					Note that \eqref{eq:limit_A} is the same as the governing equation for the irrotational flows. 
					Thus with this rescaling property at hand, one can use the compactness arguments in \cite[Theorem 6.1]{ACF85} together with the unique continuation property for the limiting equation \eqref{eq:limit_A} to conclude Proposition \ref{prop:smooth_fit}.
				\end{proof}
				
				\begin{comment}
					More precisely, let $\ell$ be the straight line passing through the origin with the direction orthogonal to the tangent to $S_1$ at $A$. Let $C_\Lambda:=\{x\in \R^2: |x|=\Lambda\}$. Clearly, any blow up $\psi_{0,A}$ satisfies: $x\mapsto \nabla \psi_{0,A}(x)$ maps the limiting nozzle into $\ell$, and maps the free boundary into $C_\Lambda$. Let $S$ be the set of all limit points $\nabla \psi (x_m)$ with $\psi(x_m)<Q$ and $x_m\rightarrow A$. In the first step, one shows that $S\subset \ell\cup C_\Lambda$. The proof is based on a compactness argument. The key step is to show that any blow up limit $w$ satisfies: either $\nabla w$ is a constant or $x\mapsto \nabla w(x)$ is an open mapping in a neighborhood of the origin. To show this one exploits the unique continuiation property for the limit equation \eqref{eq:limit_A}. In the next step, one proves that any blow up limit $\psi_{0,A}$ satisfies $\nabla\psi_{0,A}=const$, where the constant is the unique intersection point of $\ell$ with $C_\Lambda$ in the upper half plane. This implies that $S_1\cap \Gamma_{\psi}$ has a tangent at $A$.
					Finally, one shows that $\nabla\psi$ is uniformly continuous in a $\{\psi<Q\}$ neighborhood of $A$, which completes the proof.
				\end{comment}

						\section{Removal of the truncations  and the far fields behavior}\label{secremove}
						
						In this section, we first remove the domain and subsonic truncations, and then {study the far fields behavior of the jet flows at the upstream and downstream. Consequently, we get subsonic jet flows which satisfy all properties} in Problem \ref{Pb2}.
						
						\subsection{Remove the domain truncations}
						The main goal of this subsection is to remove the truncations of the domain by letting $\mu, R\to \infty$, which is a consequence of the uniform estimates (with respect to $\mu$ and $R$) of the solutions in Section \ref{secreg} and Section \ref{sec:cont_fit}. Hence one gets a solution in $\Omega$, which is the domain bounded by $S_0$ and $S_1\cup ([0,\infty)\times \{1\})$. Moreover, the solution inherits the properties of solutions in the truncated domains. The convergence  and the properties of the limit solution are summarized in the following proposition.
						
						\begin{prop}\label{prop:domain_truncation}
							Let the nozzle boundary $S_1$ defined in \eqref{eq:nozzle} satisfy \eqref{eq:nozzle1}. Given an incoming  horizontal velocity  $\bar u\in C^{1,1}([0,\bar H])$ satisfying \eqref{cond:u0_eps0} and a mass flux $Q>\tilde Q$ where $\tilde Q$ is defined in \eqref{def:Q_*}.
							Let $\psi_{\mu,R,\Lambda}$ be the minimizer of the problem \eqref{eq:mini_truncated} 
							%$J^\epsilon_{\mu,R,\Lambda}(\psi)$ for $\psi\in \mcK_{\psi^\sharp_{\mu, R}}$ 
							with  $\psi^\sharp_{\mu,R}$  defined in \eqref{eq:bdry_datum}. Then for any $\mu_j, R_j\rightarrow \infty$, there is a subsequence (still labelled by $\mu_j$ and $R_j$) such that
							$\Lambda_j:=\Lambda(\mu_j, R_j)\rightarrow \Lambda_\infty$ for some $\Lambda_\infty\in (0,\infty)$ and $\psi_{\mu_j, R_j, \Lambda_j}\rightarrow \psi_\infty$ in $C^{0,\alpha}_{loc}(\Omega)$ for any $\alpha\in (0,1)$.
							Furthermore, the following properties hold.
							\begin{itemize}
								\item [(i)] The function {$\psi:=\psi_\infty$} is a local minimizer for the energy functional, i.e., for any $D\Subset \Omega$,  one has  $J^\epsilon(\psi)\leq J^\epsilon(\varphi)$ for all  $\varphi=\psi$ on {$\partial D$}, where
								\begin{align*}
									J^\epsilon(\varphi):=\int_{D} G_\epsilon(|\nabla\varphi|^2,\varphi)+\lambda_{\epsilon,\infty}^2 \chi_{\{\varphi<Q\}}\ dx 
									\quad\text{and} \quad \lambda_{\epsilon,\infty}:=\sqrt{\Phi_\epsilon(\Lambda_\infty^{{2}},Q)}.
								\end{align*}
								{In particular, $\psi$ solves 
								\begin{equation}\label{eq_limiting_sol}
									\left\{
									\begin{aligned}
										&\nabla\cdot\left(g_\epsilon(|\nabla \psi|^2,\psi)\nabla \psi\right)-\p_z G_\epsilon (|\nabla \psi|^2,\psi)=0 &&\text{ in } \mathcal{O},\\
										&\psi =0 &&\text{ on } S_0,\\
										&\psi =Q &&\text{ on } S_1 \cup \Gamma_{\psi},\\
										&|\nabla \psi| =\Lambda_\infty &&\text{ on } \Gamma_{\psi},
									\end{aligned}
									\right.
								\end{equation}
								where $\mathcal{O}:=\Omega\cap \{\psi<Q\}$ is the flow region, $\Gamma_{\psi}:=\p\{\psi<Q\}\backslash S_1$ is the free boundary, and  $\p_zG_\epsilon(|\nabla\psi|^2,\psi)$ satisfies \eqref{dzG_expression} in the subsonic region $|\nabla\psi|^2\leq (1-\epsilon)\mft_c(\mathcal{B}(\psi))$.}
								%and it satisfies the free boundary condition
								%\begin{align*}
								%	|\nabla\psi|=\Lambda_\infty \quad \text{ on } \Gamma_\psi.
								%\end{align*}
								\item [(ii)] {$\psi$ is in $C^{2,\alpha}(\mathcal{O})\cap C^1(\overline{\mathcal{O}})$ for any $\alpha\in(0,1)$.} Furthermore, it satisfies $\p_{x_1}\psi\geq 0$ in $\Omega$.
								% and $\p_{x_1}\psi_\infty>0$ in the flow region $\mathcal{O}:=\Omega\cap \{\psi_\infty<Q\}$.
								\item[(iii)] The free boundary $\Gamma_{\psi}$ is given by the graph $x_1=\Upsilon(x_2)$ for some function $\Upsilon$, where $\Upsilon$ is $C^{2,\alpha}$ for any $\alpha\in (0,1)$ as long as it is finite.
								\item[(iv)] At the orifice $A=(0,1)$ one has $\lim_{x_2\rightarrow 1-}\Upsilon(x_2)=0$. Furthermore, $S_1\cup\Gamma_{\psi}$ is $C^1$ around $A$, in particular, $\Theta'(1)=\lim_{x_2\rightarrow 1-}\Upsilon'(x_2)$.
								\item[(v)] There is a constant $\ubar H\in (0,1)$ such that $\Upsilon(x_2)$ is finite if and only if $x_2\in (\ubar H, 1]$, and $\lim_{x_2\rightarrow \ubar H+} \Upsilon(x_2)=\infty$. Furthermore, there exists an $\bar R>0$ sufficiently large, such that $\Gamma_{\psi}\cap \{x_1>\bar R\}= \{(x_1, f(x_1)) | \bar R<x_1<\infty\}$ for some $C^{2,\alpha}$ function $f$ and $\lim_{x_1\rightarrow \infty}f'(x_1)=0$.
							\end{itemize}
						\end{prop}
						%The regularity of the free boundary follows from the same arguments as in \cite{ACF84}, and we do not repeat here.
						
						%The asymptotic behavior of the solution at $x_1=\pm\infty$ and the free boundary at $x_1=+\infty$ is established in the following proposition.
						
						%\begin{prop}\label{prop:asymp} The following properties of the free boundary hold.
						%\begin{enumerate}
						%\item[(i)] $\lim_{x_2\rightarrow 1-}\Upsilon(x_2)=0$;
						%\item[(ii)] $\Upsilon(x_2)$ is finite if and only if $\underline{H}<x_2 \leq 1$, and $%\lim_{x_2\rightarrow\underline{H}+} \Upsilon(x_2)=\infty$. Furthermore, there exists $\bar R>0$ sufficiently large, such that $\Gamma\cap \{x_1>\bar R\}= \{(x_1, f(x_1)) : \bar R<x_1<\infty\}$ for some smooth function $f$ and $\lim_{x_1\rightarrow +\infty}f'(x_1)=0$.
						%\end{enumerate}
						%\end{prop}
						\begin{proof}
							The proof for (i)--(iii) is based on standard compactness arguments (cf. Step 1 in the proof for Lemma \ref{lem:conv}), the regularity of the minimizer and the free boundary  in Section \ref{secreg} and Proposition \ref{prop:smooth_fit}, as well as the properties of $\psi_{\mu_j, R_j, \Lambda_j}$ in 
							Section \ref{secprop} and Section \ref{sec:cont_fit}. We do not repeat it here. 
							
							\emph{(iv).} The continuous fit property follows from Proposition \ref{thm:graph},  Lemma \ref{lem:conv}, and Corollary \ref{cor:cont_fit}. The smooth fit property follows from Proposition \ref{prop:smooth_fit}.
							
							\emph{(v).} Let $I:=\{x_2\in (0,1)| 0<\Upsilon(x_2)<\infty\}$. Since $\Upsilon$ is continuous, $I$ is open. For each $\bar x_2\in I$ one has $\Upsilon(x_2)<\infty$ for all $x_2\in (\bar x_2, 1)$ ({cf. \cite[Lemma 5.4]{ACF85} and Lemma \ref{lem_nondegeneracy}). Let $\ubar H:=\inf\{x_2| x_2\in I\}$. It follows from Proposition \ref{thm:graph} that $\ubar H<1$.} We claim that
							\[
							\Upsilon(x_2)=\infty\quad \text{for all }0<x_2\leq \ubar H.
							\] Indeed, if $\Upsilon(\ubar H)<\infty$, there must be an $\tilde{x}_2<\ubar H$ such that $\Upsilon(\tilde{x}_2)<\infty$. Otherwise, $\psi:=\psi_\infty$ would solve the following Cauchy problem
							\begin{equation*}
								\begin{cases}
									\nabla \cdot(g_\epsilon(|\nabla\psi|^2,\psi)\nabla\psi)-\p_zG_\epsilon(|\nabla\psi|^2,\psi)=0  &\text{ in }  (\Upsilon(\ubar H), \infty)\times (0, \ubar H),\\
									\psi=Q, \quad \p_{x_2}\psi = \Lambda &\text{ on } (\Upsilon(\ubar H),\infty)\times \{\ubar H\}.\\
									%&\psi=0 \quad  \text{ on }\R\times \{0\}.
								\end{cases}
							\end{equation*}
							It follows from  Lemma \ref{lem:uni} that $\psi=\psi_{1d}$ is a one-dimensional solution and thus the free boundary is the horizontal line $\R\times\{\ubar H\}$. This leads to a contradiction. The fact that $\Upsilon(\tilde{x}_2)<\infty$ implies that $\Upsilon(x_2)<\infty$ for all $x_2\in (\tilde{x}_2,1)$ including $\ubar H$. Thus $\ubar H$ is an interior point of $I$, which is a contradiction. Hence  $\Upsilon(x_2)<\infty$ for all $x_2\in (\ubar H,1)$ and $\Upsilon(x_2)=\infty$ for all $x_2\in (0,\ubar H]$. Since $\lim_{x_2\rightarrow \ubar H+}\Upsilon(x_2)=\infty$, the flatness condition (cf. p. 43 in \cite{ACF84}) is satisfied when $x_1>\bar R$ for sufficiently large $\bar R$. This implies that the free boundary is an $x_1$-graph {so that it can be expressed as} $x_2=f(x_1)$ for  $x_1>\bar R$, where $f$ is a $C^{2,\alpha}$ function. In addition, one has $f'(x_1)\rightarrow 0$ as $x_1\rightarrow \infty$.
						\end{proof}
						
						{The next proposition shows the positivity of  $\partial_{x_1}\psi_\infty$, which plays an important role in the equivalence between the stream function formulation and the original Euler system.}
						%\tre{In view of Proposition \ref{Elemmaequivalent} and Lemma \ref{lem:stream_equiv}, with the regularity and asymptotic behavior of solutions at hand, the equivalence between the stream function formulation and the original Euler system  is true as long as $\partial_{x_1}\psi_\infty>0$ inside the flow region. This fact is guaranteed by the following proposition.}
					
					\begin{prop}\label{prop:positive_psi}
						Let $\psi:=\psi_\infty$ be a solution obtained in Proposition \ref{prop:domain_truncation}. Then $\p_{x_1}\psi>0$ in $\mcO:= \Omega\cap \{\psi<Q\}$.
					\end{prop}
					\begin{proof}
						It follows from part (ii) of Proposition \ref{prop:domain_truncation}  that $\p_{x_1}\psi\geq 0$ in $\mcO$.  The strict inequality can be proved by the strong maximum principle. Indeed,
						assume that $\p_{x_1}\psi(x^{(0)})=0$ for some 
						$x^{(0)}\in \mcO$. 
						%Let $\mathcal{U}\Subset \Omega\cap \{\psi<Q\}$ be a neighborhood of $x^{(0)}$. 
						Let $V:=\p_{x_1}\psi$. Then $V \geq 0$ in $\mcO$ and it {satisfies
							\begin{equation*}\label{Weq}
								\p_i(\mathfrak a^{ij}\p_jV)+ \partial_{i} (\mfb^i\p_iV)-\mfb^i \p_i V -\mfc V=0 \quad\text{in }\mcO,
							\end{equation*}
							where
							\begin{align*}
								\mathfrak a^{ij}:=\p_{p_ip_j}G_\epsilon(|\nabla\psi|^2,\psi) ,\quad \mfb^i:=\p_{p_iz} G_{\epsilon}(|\nabla\psi|^2,\psi),\quad \mfc:= \p_{zz}{G}_\epsilon(|\nabla\psi|^2,\psi).
							\end{align*}
							The strong maximum principle (cf. \cite[Theorem 3.5]{GT})} gives that  $V=0$ in $\mcO$, and hence all streamlines are horizontal.  This is however a contradiction. Hence $V>0$ in $\mcO$ and thus the proof of the proposition is completed.
					\end{proof}

						\subsection{Remove the subsonic truncation when the mass flux is large}
						This subsection devotes to the removal of the subsonic truncation. 
						We aim to show that for any given $\epsilon\in (0,1/4)$, if $Q=Q(\epsilon)$ is suitably large, then one always has $|\nabla\psi_\infty|^2\leq(1-\epsilon)\mft_{c}(\mcB(\psi_\infty))$ in the flow region, where $\mft_{c}$ defined in \eqref{defF} is the square of critical momentum. This means that  the subsonic truncation introduced in Section \ref{sec:truncation} can be removed.
						
						\begin{prop}\label{thm:subsonic}
							Let $\psi:=\psi_\infty$ be a limiting solution in Proposition \ref{prop:domain_truncation} and $\tilde Q$ be as in  \eqref{def:Q_*}. 
							Suppose $Q>Q^*$ for some ${Q}^*\geq \tilde Q$ sufficiently large depending on $\bar u$, $\gamma$, $\epsilon$ and the nozzle, then
							\begin{equation}\label{removetrun}
								|\nabla \psi|^2\leq(1-\epsilon)\mft_c(\mcB(\psi)).
							\end{equation}
						\end{prop}
						\begin{proof}
							In the flow region $\mathcal O:=\Omega\cap\{\psi<Q\}$, $\psi$ satisfies the following equation of nondivergence form 
							\begin{equation*}\label{eqnondiv}
								a_\epsilon^{ij}(\nabla\psi,\psi) \p_{ij}\psi=F_\epsilon(|\nabla \psi|^2, \psi),
							\end{equation*}
							where 
							$$F_\epsilon(|\nabla \psi|^2, \psi)=-\p_z g_\epsilon(|\nabla \psi|^2, \psi) |\nabla\psi|^2+\p_zG_\epsilon(|\nabla \psi|^2, \psi),$$ 
							and the matrix
							$$(a_\epsilon^{ij})=g_\epsilon(|\nabla \psi|^2, \psi)I_2+2\p_t g_\epsilon(|\nabla \psi|^2, \psi)\nabla\psi\otimes \nabla \psi$$
							is symmetric with the eigenvalues
							\begin{align*}
								\beta_{0,\epsilon}  =g_\epsilon(|\nabla \psi|^2, \psi)\quad \text{and}\quad
								\beta_{1,\epsilon} =g_\epsilon(|\nabla \psi|^2, \psi)+2\p_t g_\epsilon(|\nabla \psi|^2, \psi)|\nabla\psi|^2.
							\end{align*}
							
							It follows from Lemma \ref{lem:truncation_g}, \eqref{eq:upper_pzzG} and \eqref{label_7} that there exists a constant $C_\gamma>1$ depending only on $\gamma$ such that 
							\begin{equation}\label{subsonic_eq}
								1\leq \frac{\beta_{1,\epsilon}}{\beta_{0,\epsilon}}\leq C_\gamma \epsilon^{-1} 
								\quad\text{and}\quad
								%\tilde C_\gamma^{-1}\bar\rho\leq
								\frac{|F_\epsilon|}{\beta_{0,\epsilon}}\leq \frac{C_\gamma\kappa_0 Q}{\epsilon \|\bar u\|_{L^1([0,\bar H])}}.
							\end{equation}
							With \eqref{subsonic_eq} at hand, one can use similar arguments as in \cite[Propisition 3]{XX3} to get   
							\begin{equation}\label{eq:subsonic}
								\|\psi\|_{C^1(\mathcal O)}\leq C\bigg(1+Q+\Big\|\frac{F_\epsilon}{\beta_{0,\epsilon}}\Big\|_{L^{\infty}(\mathbb R\times[0,Q])}\bigg)\leq C(1+Q)
							\end{equation}
							where $C$ depends on $\gamma,\epsilon, \bar u$ and the nozzle. 
							Note that the right-hand side of \eqref{eq:subsonic} is $O(Q)$ for $Q\gg 1$,  and $\mft_c(\mathcal B(\psi))\geq C(\gamma,\bar u) Q^{\gamma+1}$, which simply follows from the definition of $\mft_c$ in \eqref{defF} and \eqref{label_7}. Since $\gamma>1$, we obtain the desired conclusion. 
						\end{proof}
						\begin{rmk}\label{rmk_remove_subsonic}
						{Given $\epsilon\in(0,1/4)$, let $\psi$ be a solution of the elliptic equation in \eqref{eq_limiting_sol} which satisfies the estimate \eqref{removetrun}. In view of \eqref{gm_g} and the expression of $\partial_zG_\epsilon$ in the subsonic region, cf. \eqref{dzG_expression}, the equation in \eqref{eq_limiting_sol} is exactly the equation in \eqref{eq_pb2}.}
						%the definition of $G_\epsilon$ in \eqref{Gepsilon} that $\psi$ satisfies the elliptic equation
						%\begin{equation}\label{eq:g_G}
						%\nabla\cdot\left(g(|\nabla \psi|^2,\psi)\nabla \psi\right)-\p_z G (|\nabla \psi|^2,\psi)=0 \quad \text{ in } \Omega\cap\{\psi<Q\},
						%\end{equation}
						%where 
						%\begin{align}\label{G}
						%	G(t,z):=\frac{1}{2}\int_0^t g(\tau, z) d\tau +\frac{1}{\gamma}\left( g(0,z)^{-\gamma}- g(0,Q)^{-\gamma}\right).
						%\end{align}
						%In view of the proof for  \eqref{dzG_expression}, one has $\p_zG(|\nabla\psi|^2,\psi)=\frac{\mathcal{B}'(\psi)}{g(|\nabla\psi|^2, \psi)}$. Thus the equation \eqref{eq:g_G} is exactly the elliptic equation in \eqref{eq_pb2}.}
						\end{rmk}
						\subsection{Far fields behavior of the jet}
						%\tre{This subsection is about the asymptotic behavior of \tre{a uniformly subsonic solution $\psi$ for \eqref{eq_limiting_sol}} at $x_1=\pm\infty$, which plays an important role in proving the equivalence of Problem \ref{pb} and Problem \ref{Pb2}. }
						{This subsection is about the asymptotic behavior of {a uniformly subsonic jet flow} as $x_1\to\pm\infty$.}

							\begin{prop}\label{prop:asymptotic}
							 {Let the nozzle boundary $S_1$ defined in \eqref{eq:nozzle} satisfy \eqref{eq:nozzle1}. 
							 Given an incoming  horizontal velocity  $\bar u\in C^{1,1}([0,\bar H])$ satisfying \eqref{cond:u0_eps0}. 
							 Let $\psi$ be a solution of \eqref{eq_limiting_sol} with positive constants $Q>Q_*$ and $\Lambda:=\Lambda_\infty$, where $Q_*$ is defined in \eqref{def:Q_*}. Assume that $\psi$ and the free boundary $\Gamma_\psi$ satisfy the properties (ii)-(v) in Proposition \ref{prop:domain_truncation} and  \eqref{removetrun}.} Then for any $\alpha\in(0,1)$, as $x_1\rightarrow -\infty$,
								\begin{equation}\label{eqbarpsi}
								\psi(x_1,x_2)\to\bar \psi(x_2):=\bar\rho\int_0^{x_2}\bar u(s)ds \quad \text{in } C_{loc}^{2,\alpha}([0,\bar H)),
								\end{equation}
							 where $\bar\rho$ is defined in \eqref{eq:rhobar}; as $x_1\rightarrow \infty$, 
								%is the stream function associated with the incoming flow $\bar\rho, \bar u$; 
								\begin{equation}\label{ubarpsi}
									\psi(x_1,x_2)\to\ubar\psi(x_2):=\ubar\rho\int_0^{x_2}\ubar u(s)ds \quad \text{in } C_{loc}^{2,\alpha}([0,\ubar H)),
								\end{equation}  
								where $\ubar\rho>0$ is the downstream density, $\ubar u\in C^{1,\alpha}([0,\ubar H])$ is the downstream (positive) horizontal velocity, and $\ubar H>0$ is the downstream asymptotic height. Furthermore, the downstream states $(\ubar{\rho},\ubar{u},\ubar H)$ are uniquely determined by $\gamma$, $\epsilon$, $\bar u$, $Q$ and $\Lambda$.
							\end{prop}
							\begin{proof}
								\emph{Step 1. Upstream asymptotic behavior.} Let $\psi^{(-n)}(x_1,x_2):=\psi(x_1-n, x_2)$, $n\in\mathbb Z$. Since the nozzle is asymptotically horizontal with the height $\bar H$, there exists a subsequence  $\psi^{(-n)}$ (relabeled) converging to a function $\hat\psi$ in $C_{loc}^{2,\alpha}(\R\times[0,\bar H))$ for any $\alpha\in(0,1)$, where $\hat\psi$ {satisfies \eqref{removetrun} and} solves the Dirichlet problem in the infinite strip
								\begin{equation}\label{blowuplimitpb}
									\left\{
									\begin{aligned}
										&\nabla \cdot(g_\epsilon(|\nabla\hat\psi|^2,\hat\psi)\nabla\hat\psi)-\p_zG_\epsilon(|\nabla\hat\psi|^2,\hat\psi)=0
										\quad  \text{ in }  \R\times (0, \bar H),\\
										&\hat\psi=0  \text{ on }\R\times \{0\},\quad \hat\psi=Q\text{ on } \R\times \{\bar H\}.
										%\quad 0<\hat\psi<Q \text{ in } \R\times (0,\bar H).
									\end{aligned}
									\right.
								\end{equation}
								Using a standard shifting argument (cf. Lemma \ref{lem:variation_sol} for the one-dimensional case and \cite[Lemma 7.3]{ACF85} for the two-dimensional unbounded case), one can show that \eqref{blowuplimitpb} has a unique solution. More precisely, a similar argument as in Lemma \ref{lem:variation_sol}(i) gives  $0<\hat\psi<Q$ in $\R\times (0,\bar H)$. Assume $\hat\psi_1$ and $\hat\psi_2$ are two solutions of \eqref{blowuplimitpb}. Define $\hat\psi_{2,k}(x_1,x_2):=\hat\psi_2(x_1,x_2+k)$ for $k\in[0,\bar H)$. Then $\hat\psi_1$ and $\hat\psi_{2,k}$ are solutions of \eqref{blowuplimitpb} in $\R\times(0,\bar H-k)$. Thus with the shifting method as in \cite[Lemma 7.3]{ACF85} one has $\hat\psi_1\leq\hat\psi_{2,0}=\hat\psi_2$, and vice versa. This means \eqref{blowuplimitpb} has a unique solution. 
								{Since $\hat\psi$ satisfies \eqref{removetrun} and the function $\bar\psi$ defined in \eqref{eqbarpsi} satisfies the  equation in \eqref{eq_pb2}, by Remark \ref{rmk_remove_subsonic} one has $\hat\psi(x_1,x_2)=\bar \psi(x_2)$ in $\R\times[0,\bar H]$.} This proves the asymptotic behavior of the flows at the upstream.
								
								\emph{Step 2. Downstream asymptotic behavior.}
								Let $\psi^{(n)}(x_1, x_2):=\psi(x_1+n, x_2)$, $n\in\mathbb{Z}$.
									By the $C^{2,\alpha}$ regularity of the free boundary and  the boundary  regularity for elliptic equations, there exists a subsequence $\psi^{(n)}$ (relabeled) converging to a function $\ubar{\psi}$ in $C_{loc}^{2,\alpha}(\R\times[0,\ubar H))$, where $\ubar\psi$ {satisfies \eqref{removetrun} and solves} 
									\begin{equation}\label{pb_downstream}
										\begin{cases}
											\nabla \cdot(g_\epsilon(|\nabla\ubar{\psi}|^2,\ubar{\psi})\nabla\ubar{\psi})-\p_zG_\epsilon(|\nabla\ubar{\psi}|^2,\ubar{\psi})=0  &\text{ in }  \R\times (0, \ubar H),\\
											\ubar{\psi}=0 &\text{ on } \R\times \{0\},\\
											\ubar{\psi}=Q, \quad \p_{x_2}\ubar{\psi} = \Lambda &\text{ on } \R\times \{\ubar H\}.
										\end{cases}
									\end{equation}
									For given $Q$ and $\Lambda>0$, there is a unique one-dimensional solution to the above problem. Furthermore, the height $\ubar H$ is uniquely determined by $Q$ and $\Lambda$. Indeed, for $k>0$, consider the variational problem
									\[
									\inf_{\varphi\in \mcK^\flat_k}\int_0^k G_\epsilon(|\varphi'|^2,\varphi),
									\]
									where 
									$$\mcK^\flat_k =\{\varphi \in C^{0,1}([0,k];\R)| \varphi(0)=0, \varphi(k)=Q\}.$$ 
									It follows from the direct method in the calculus of variations that there exists a minimizer ${\varphi}_k$. Moreover, by the Hopf lemma  $\Lambda(k):=\lim_{s\rightarrow k-}{\varphi}'_k(s)$ is strictly monotone decreasing in $k$ (cf. Lemma \ref{lem:variation_sol}(iii)). Thus there exists a unique $k_{\Lambda}>0$ such that $\Lambda(k_{\Lambda})=\Lambda$ and necessarily $k_{\Lambda}=\ubar H$ (otherwise by the Hopf lemma $\p_{x_2}\ubar\psi(x_1,\ubar H)\neq \p_{x_2}\ubar\varphi(x_1,k_{\Lambda})$ for $x_1\in\R$, where $\ubar\varphi(x_1,x_2):=\varphi_{k_{\Lambda}}(x_2)$, then there is a contradiction). 
									As the unique continuation (cf. Lemma \ref{lem:uni}) implies that the Cauchy problem \eqref{pb_downstream} has a unique solution, one has $\ubar\psi(x_1,x_2)=\varphi_{\ubar H}(x_2)$ in $\R\times[0,\ubar H]$. Thus one can conclude that  $\psi$ converges to a unique one-dimensional function $\ubar\psi$ as $x_1\rightarrow \infty$ with a unique asymptotic height {$\ubar H>0$}. 
								
									Since $\ubar\psi=\ubar\psi(x_2)$ and since $\ubar\psi'>0$ on $[0,\ubar H]$ by Lemma \ref{lem:variation_sol}(ii), using \eqref{eq:psi_uv} one has $u_2\to0$ as $x_1\to\infty$. Consequently the downstream density 
								$$\ubar \rho:=\frac1{g(|\ubar\psi'|^2,\ubar\psi)}$$ 
								must be a constant (cf. \cite[Remark 1.1]{CDXX}), where $g$ is the function defined in \eqref{eq:branch}. Then using \eqref{eq:psi_uv} again yields the downstream horizontal velocity $\ubar u(x_2)=\ubar \psi'(x_2)/\ubar\rho$ for $x_2\in[0,\ubar H]$, which is always positive since $\ubar\psi'>0$ on $[0,\ubar H]$. Hence one has the expression of $\ubar\psi$ in \eqref{ubarpsi}. The regularity of $\ubar u$ follows from the regularity of $\ubar \psi$.  
								This finishes the proof of the proposition.
							\end{proof}
					{In view of Propositions \ref{prop:domain_truncation}--\ref{prop:asymptotic}, we have proved the existence of solutions to {Problem} \ref{Pb2} when the mass flux $Q$ is sufficiently large. Then the existence of solutions to Problem \ref{pb} follows from Proposition \ref{prop:equiv_sol}.}
					
					\section{Uniqueness}\label{secunique}
					In this section we prove the uniqueness of the solution to Problem \ref{Pb2}. 
					\begin{prop}\label{propunique}
					{Let the nozzle boundary $S_1$ defined in \eqref{eq:nozzle} satisfy \eqref{eq:nozzle1} and \eqref{eq:nozzle2}. Given an incoming  horizontal velocity  $\bar u\in C^{1,1}([0,\bar H])$ satisfying \eqref{cond:u0_eps0} and a mass flux $Q>Q_*$, where $Q_*$ is defined in \eqref{def:Q_*}. 
					If $(\psi_i, \Gamma_i, \Lambda_i)$ ($i=1$, $2$) are two  uniformly subsonic solutions to Problem \ref{Pb2}, then they must be the same.}
					\end{prop} 
					
					\begin{proof} Let $(\unrho_i, \unu_i, \unH_i)$, $i=1,2$, be the associated downstream asymptotic states. The proof is divided into two steps.
						
						\textit{Step 1. Proof for $\Lambda_1=\Lambda_2$.} Without loss of generality, assume that $\Lambda_1 <\Lambda_2$. Since $\bar u$ and $Q$ are fixed, in view of Step 2 in Proposition \ref{prop:asymptotic} and Lemma \ref{lem:variation_sol}(iii) one has $\ubar H_1>\ubar H_2$. 
						Let $\mcO_i$ be the domain bounded by $S_0$, $S_1$, and $\Gamma_{i}$.
						Since $\unH_1>\unH_2$ and that $S_1\cup \Gamma_i$ is an $x_2$-graph, one can find a $k\geq0$ such that the domain $\mcO_1^{(k)}=\{(x_1, x_2)|(x_1-k, x_2)\in \mcO_1\}$  contains $\mcO_2$. Let $k_*$ be the smallest number such that $\mcO_1^{(k_*)}$ contains $\mcO_2$ and they touch at some point $(x_1^\ast, x_2^\ast)\in \Gamma_2$. Let $\psi_1^{(k)}(x_1,x_2):= \psi_1(x_1-k, x_2)$. 
						{We claim $\psi_1^{(k_*)}\leq \psi_2$  in $\mcO_2$. In fact, define $\psi_{2,\tau}(x_1,x_2):=\psi_2(x_1,x_2+\tau)$ for $\tau\geq0$. Let  $\tau_*$ be the smallest constant such that $\psi_1^{(k_*)}\leq \psi_{2,\tau_*}$ in $\mathcal O_2$. Then $\tau_*\in[0,\bar H)$. Suppose $\tau_*>0$. Note that $\psi_1^{(k_*)}$ and $\psi_2$ have the same asymptotic behavior as $x_1\to-\infty$, and $\psi_1^{(k_*)}<\psi_2$ as $x_1\to\infty$ by the strong maximum principle (similarly as in Lemma \ref{lem:variation_sol}(iv)). Thus $\psi_1^{(k_*)}<\psi_{2,\tau_*}$ as $x_1\to\pm\infty$, so that there exists a point $(\hat x_1,\hat x_2)\in\mathcal O_2$ such that $\psi_1^{(k_*)}(\hat x_1,\hat x_2)=\psi_{2,\tau_*}(\hat x_1,\hat x_2)<Q$, which contradicts the strong  maximum principle. Thus $\tau_*=0$ and consequently $\psi_1^{(k_*)}\leq \psi_2$  in $\mcO_2$. 
							Note that 
							$\psi_1^{(k_*)}(x_1^\ast, x_2^\ast)= \psi_2(x_1^\ast, x_2^\ast)$. Then by the Hopf lemma we have
							\begin{equation*}
								\Lambda_1=\frac{\partial \psi^{(k_*)}_1}{\partial \nu}(x_1^\ast, x_2^\ast) >\frac{\partial\psi_2}{\partial \nu}(x_1^\ast, x_2^\ast) =\Lambda_2.
							\end{equation*}
							This leads to a contradiction. Hence one has $\Lambda_1=\Lambda_2$.}
						\begin{center}
							\includegraphics[height=5cm, width=10cm]{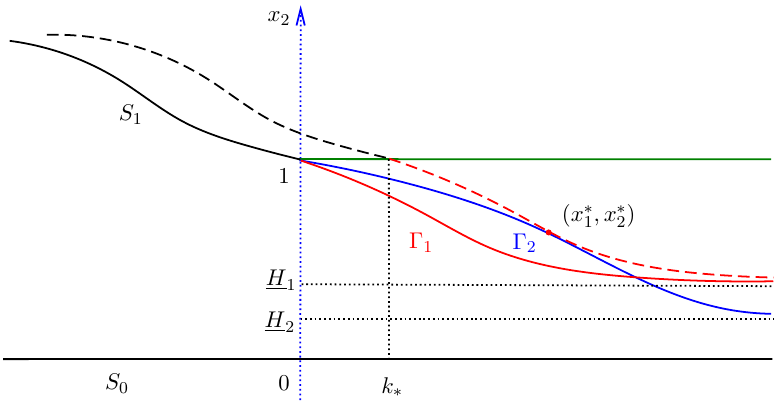}\\
							{\small Figure 5.}
						\end{center}

						\textit{Step 2. Proof for $\psi_1=\psi_2$ and $\Gamma_1=\Gamma_2$.} 
						Suppose that $\Gamma_1\neq \Gamma_2$,  where $\Gamma_i=(\Upsilon_i( x_2), x_2) (i=1,2)$. Without loss of generality, assume that there exists an $\tilde{x}_2$ such {that $\Upsilon_1(\tilde{x}_2)<\Upsilon_2(\tilde{x}_2)$. Since the nozzle boundary $S_1$ satisfies the monotone condition  \eqref{eq:nozzle2}, there exists a point $(\mu_*,0)$ ($\mu_*<0$) such that 
							\begin{itemize}
								\item [(i)] $S_1$ lies above the line $l_*$ connecting $(\mu_*,0)$ to $(0,1)$ and $S_1$ is starshaped with respect to $(\mu_*,0)$; 
								\item[(ii)] the free boundaries $\Gamma_1$ and $\Gamma_2$ lie below $l_*$.
							\end{itemize} 
							Without loss of generality, we assume $\mu_*=0$. Note that from Step 1 and Proposition \ref{prop:asymptotic} one has $\unH_1=\unH_2$. Similar as in Step 1, there exists $\tilde{k}_*\geq 0$ such that the free boundary of $\psi_{2,\tilde{k}_*}(x_1,x_2):=\psi_2(x_1,x_2+\tilde{k}_*)$ touches the free boundary of $\psi_1$  %$\mcO_1^{(\tilde{k}_*)}:=\{(x_1,x_2):(x_1,x_2-\tilde{k}_*)\in\mathcal O_1\}$ contains $\mcO_2$ and they touch 
							at some point $(\tilde{x}_1^\ast, \tilde{x}_2^\ast)\in \Gamma_1$. 
							%Let $\Psi_1^{(\tilde{k}_*)}(x_1,x_2):=\psi_1(x_1,x_2-\tilde{k}_*)$. 
							By the strong maximum principle one can deduce  $\psi_1\leq\psi_{2,\tilde{k}_*}$ in $\mathcal O_1$.  Moreover, by the Hopf lemma, one has
							\begin{equation*}
								\Lambda_1=\frac{\partial \psi_1}{\partial \nu}(\tilde{x}_1^\ast, \tilde{x}_2^\ast) >\frac{\partial\psi_{2,\tilde{k}_*}}{\partial \nu}(\tilde{x}_1^\ast, \tilde{x}_2^\ast) =\Lambda_2.
							\end{equation*}
							However, this leads to a contradiction. Therefore, one has  $\Gamma_1=\Gamma_2$ and thus $\mcO_1=\mcO_2$. Let $\Gamma:=\Gamma_1$ and $\mcO:=\mcO_1$. Now we prove $\psi_1=\psi_2$. Note that they are both solutions to the boundary value problem
							\begin{equation*}
								\left\{
								\begin{aligned}
									&\nabla \cdot(g(|\nabla\psi|^2,\psi)\nabla\psi)={\frac{\mathcal{B}'(\psi)}{ g(|\nabla\psi|^2,\psi)}} \quad  \text{ in }  \mcO,\\
									&\psi=Q\text{ on } S_1\cup \Gamma,\quad \psi=0  \text{ on } S_0
								\end{aligned}
								\right.
							\end{equation*}
							with the same asymptotic behavior as  $x_1\to\pm\infty$. Let $\psi_{2,k}(x_1,x_2)=\psi_2(x_1,x_2+k)$ for $k>0$. Then by similar arguments as in Step 1 one has  $\psi_1\leq\psi_{2,k}$ for all $k>0$, so that 
							$\psi_1\leq\psi_2$ in $\mcO$ and vice versa.} Hence $\psi_1=\psi_2$.
						
						This completes the proof of the lemma.
					\end{proof}
					%\newpage
					
					\section{Existence of Critical Mass Flux}\label{SEcritical}
					{For the given horizontal velocity at
						the upstream as in Theorem \ref{mainthm},  it has been proved in
						Section  \ref{secremove} that  there exist subsonic jet  flows as long as $Q>{Q}^*$ for some $Q^*$ sufficiently large. In this section, we 
						decrease $Q^*$ as small as possible until some critical mass flux $Q_c$, such that for $Q<Q_c$ either the flow fails to be global subsonic or the stream function $\psi$ fails to be a solution in the sense of Problem \ref{Pb2}. 
						\begin{prop}\label{Epropcritical}
							Let the nozzle boundary $S_1$ defined in \eqref{eq:nozzle} satisfy \eqref{eq:nozzle1} and \eqref{eq:nozzle2}. 
							Let $\bar u\in C^{1,1}([0,\bar H])$ be the incoming  horizontal velocity satisfying  \eqref{cond:u0_eps0} and $Q_*$ be as in \eqref{def:Q_*}. There exists a critical mass flux $Q_c\geq Q_*$ such that if $Q>Q_c$, then there exists a unique solution $\psi$ to Problem \ref{Pb2}, which satisfies
							\begin{equation}\label{Eellipticity}
								0< \psi< Q \ \text{ in }\mcO\quad\text{and}\quad
								\mfM(Q):=\sup_{\overline{\mcO}}\frac{|\nabla\psi|^2}{\mft_c(\mathcal{B}(\psi))}<1,
							\end{equation}
							where the domain $\mcO$ is bounded by $S_0$, $S_1$, and $\Gamma_\psi$. Furthermore, as $Q\rightarrow Q_c$ one has either
							$\mfM(Q)\rightarrow 1$, or there does not
							exist a $\sigma>0$ such that Problem \ref{Pb2} has solutions for all
							$Q\in (Q_c-\sigma,Q_c)$ and
							\begin{equation}\label{Ebifurcationestimate}
								\sup_{Q\in(Q_c-\sigma,Q_c)}\mfM(Q)<1.
							\end{equation}
					\end{prop}}
					\begin{proof}
						The proof for the proposition is inspired by \cite[Proposition 6]{XX3}, where the existence of critical flux for subsonic solutions with non-zero vorticity in fixed nozzles was established. 
						
						{For the given horizontal velocity $\bar u$ at the upstream %satisfying \eqref{cond:u0_eps0} as in Theorem \ref{mainthm} 	
							and any mass flux $Q>Q_*$, one can determine the density $\bar\rho$ at the upstream and the Bernoulli function $\mcB(\psi)$.} 
						Since $\barho$ and $\mcB(\psi)$ depend on $Q$ by definition, in this
						section, we denote them by $\barho(Q)$ and $\mcB(\psi;Q)$,
						respectively.
						
						Let $\{\epsilon_n\}_{n=1}^{\infty}$ be a strictly decreasing
						sequence of positive numbers such that $\epsilon_n\downarrow 0$ and $\epsilon_1\leq \epsilon$ where $\epsilon$ is the constant used in Section \ref{sec:truncation} for subsonic truncations. {Let $\psi_n(\cdot; Q)$ be the function (if it exists) such that
						\begin{itemize}
							\item [(i)] $\psi_n(x; Q)\in  C^{2,\alpha}(\{\psi_n(x;Q)<Q\})\cap C^{1}(\overline{\{\psi_n(x;Q)<Q\}})$ satisfies 
							\begin{equation}\label{eq_critical}
								\left\{
								\begin{aligned}
									&\nabla\cdot\left(g_{\epsilon_n}(|\nabla \psi|^2,\psi)\nabla \psi\right)- \partial_z G_{\epsilon_n}(|\nabla\psi|^2,\psi)=0 &&\text{in } \{\psi<Q\},\\
									&\psi =0 &&\text{on } S_0,\\
									&\psi =Q &&\text{on } S_1 \cup \Gamma_\psi,\\
									&|\nabla \psi| =\Lambda &&\text{on } \Gamma_\psi,
								\end{aligned}
								\right.
							\end{equation}
							where the functions $g_{\epsilon_n}$ and $G_{\epsilon_n}$ are defined in \eqref{eq:truncation_g} and \eqref{Gepsilon}, respectively;
							\item [(ii)] $\psi_n(x; Q)$ satisfies $\partial_{x_1}\psi_n(x;Q)>0$ in $\{\psi_n(x;Q)<Q\}$; 
							\item [(iii)]  the free boundary $\Gamma_{\psi_n}$ satisfies the properties (3)--(5) in Problem \ref{Pb2}. Furthermore, there exist constants $C_1(\epsilon_n, Q, \tau)$ and $C_2(\epsilon_n, Q, R)$ (with $\tau$, $R>0$) depending on $Q$ continuously such that
							\begin{equation}\label{eq9.3.5}
							\|\Upsilon\|_{C^{2, \alpha}([\underline{H}+\tau, 1-\tau ])}\leq C_1(\epsilon_n, Q, \tau) \quad \text{and}\quad \|f\|_{C^{2,\alpha}(R, \infty)}\leq C_2(\epsilon_n, Q, R),
							\end{equation}
							where $\Upsilon$ and $f$ are the functions appeared in the properties (3) and (5) in Problem \ref{Pb2}, respectively.
						\end{itemize}
						Moreover, if 
						\begin{equation}\label{Ecriticaluniformellipticity}
							{\frac{|\nabla \psi_{n}(x; Q)|^2}{ \mft_c({\mathcal{B}}(\psi_n;Q))}\leq
								1-\epsilon_n,}
						\end{equation}
						then $\psi_n(x;Q)$ solves the problem \eqref{eq_pb2}, and it satisfies the far fields behavior claimed in Proposition \ref{prop:asymptotic} and
						\begin{equation*}
							0\leq \psi_n(x;Q)\leq Q.
						\end{equation*}
						Set
						\begin{equation*}
							\mfS_n(Q):=\{\psi_n(x;Q)|\psi_n(x;Q)\,\,\text{satisfies the properties (i)--(iii)}\}.
						\end{equation*}
						Define}
							\begin{equation*}
								\mfM_n(Q):=\inf_{\psi_n\in \mfS_n(Q)}\sup_{\overline{\mcO}}
								\frac{|\nabla\psi_n(x;Q)|^2}{\mft_c(\mathcal{B}(\psi_n;Q))}
							\end{equation*}
							and
							\begin{equation*}
								\mfT_n:=\{s|s\geq Q_*,\,\, \mfM_n(Q)\leq
								1-\epsilon_n\,\,\text{if}\,\,Q\in (s,\infty)\}.
							\end{equation*}
							It follows from Section \ref{secremove} that 
							$[{Q}^*,\infty)\subset \mfT_n$ for sufficiently large $n$. Hence
							$\mfT_n$ is not an empty set. Define $Q_n:=\inf \mfT_n$.
						
						The sequence $\{Q_n\}$ has some nice properties.
						
						First, $\mfM_n(Q)$ is right continuous for $Q>Q_n$. Let $\{Q_{n,k}\}\subset(Q_n,\infty)$ and $Q_{n,k}\downarrow {Q}$ as $k\to \infty$. Since $\mfM_n(Q_{n,k})\leq 1-\epsilon_n$, one has
							\begin{equation*}
								\|\psi_n(x;Q_{n,k})\|_{C^{0,1}(\overline{\mcO_{n,k}})}\leq C\,\, \text{and}\,\, \|\psi_n(x;Q_{n,k})\|_{C^{2,\alpha}(\overline{\mcO_{n,k}}\backslash B_r(A))}\leq C(r)\,\, \text{for any}\,\,r>0,
							\end{equation*}
							where $\mcO_{n,k}$ is the associated domain bounded by $S_0$, $S_1$, and the free boundary $\Gamma_{n,k}$, and $B_r(A)$ is the disk centered at $A=(0,1)$ with radius $r$.  
							Therefore, there exists a subsequence $\psi_n(x;Q_{n,k_m})$
							such that $\psi_n(x;Q_{n,k_m})\rightarrow \psi$ and $\mcO_{n,k_m}\to \mcO$ as $m\to \infty$. 
							{Furthermore, together with the estimates \eqref{eq9.3.5}, 
							$\psi$ satisfies the properties (i)--(iii) mentioned above.} 
							Thus $\mfM_n(Q)\leq
							\lim_{m\rightarrow \infty} \mfM_n(Q_{n,k_m}) \leq 1-\epsilon_n$. It follows from Proposition \ref{propunique} that $\psi$ is unique. Hence $\mfM_n(Q)= \lim_{m\rightarrow \infty} \mfM_n(Q_{n,k_m})$.
						
						Second, $Q_n>Q_*$. Suppose not, by the definition of $Q_n$ one has 
						$Q_*\in \mfT_n$. It follows from the right continuity of $\mfM_n$ that
						$\mfM_n(Q_*)\leq 1-\epsilon_n$. Thus $\psi_n(x;{Q}_*)$ satisfies the far fields
						behavior claimed in Proposition  \ref{prop:asymptotic}.  However, it follows from
						the definition of $Q_*$ that
						$$\sup_{x\in \overline{\mcO}} \frac{|\nabla\psi_n(x; Q_*)|^2}{\mft_c({\mathcal{B}}(\psi_n(x; Q_*)))}
						\geq\sup_{x_2\in[0,\barH]}\frac{|\barho(Q_*)\baru(x_2; Q_*)|^2}{\mft_c(B(x_2))}=1.$$
						%\begin{eqnarray*}
						%	& &\sup_{x\in \overline{\mcO}}\left(|\nabla\psi_n(x; Q_*)|^2-
						%	\mft_c({\mathcal{B}}(\psi_n(x; Q_*)))\right)\\
						%	&\geq&\sup_{x_2\in[0,\barH]}\{|\barho(Q_*)\baru(x_2; Q_*)|^2-\mft_c(B(x_2))\}\\
						%	&=&0.
						%\end{eqnarray*}
						Thus $\mfM_n(Q_*)\geq 1$. This leads to a contradiction, which implies
						$Q_n>Q_*$.
						
						Finally, it follows from the definition of $\{Q_n\}$  that $\{Q_n\}$ is  {a decreasing} sequence. 	
						
						Define $Q_c:=\lim_{n\rightarrow \infty} Q_n$. Based on previous properties of $\{Q_n\}$, $Q_c$ is well-defined and $Q_c\geq  Q_*$. Note that for any $Q>Q_c$, there exists a $Q_n\in(Q_c,Q)$ and thus $\mfM_n(Q)\leq 1-\epsilon_n$.
						{Therefore, $\psi=\psi_n(x;Q)$ solves Problem \ref{Pb2}.}
						%satisfies the properties (i)--(iii) mentioned before and
						%	\begin{equation*}
						%		\sup_{\overline{\mcO}}
						%		\frac{|\nabla\psi|^2}{\mft_c(\mathcal{B}(\psi))}=\mfM_n(Q)\leq 1-\epsilon_n.
						%	\end{equation*}
						%Furthermore,$\partial_{x_1}\psi_n>0$ and the induced flow $(\rho, \mathbf{u})$ satisfies all other properties in part (i) of Theorem \ref{mainthm}.
						
						{If $\sup_{Q\in(Q_c,\infty)}\mfM(Q)<1$, then there exists an
							$n$ such that
							$\sup_{Q\in(Q_c,\infty)}\mfM(Q)<1-\epsilon_n$. As the
							same as the proof for the right continuity of $\mfM_n(Q)$ on
							$[Q_n,\infty)$, $\mfM(Q_c)\leq 1-\epsilon_n$. Suppose
							that there exists $\sigma>0$ such that Problem  \ref{Pb2} always has
							a solution $\psi$ for $Q\in(Q_c-\sigma,Q_c)$, and
							\begin{equation*}
								\sup_{Q\in(Q_c-\sigma,Q_c)}
								\mfM(Q)=\sup_{Q\in(Q_c-\sigma,Q_c)}\sup_{\overline{\mcO}}
								\frac{|\nabla\psi|^2}{\mft_c(\mathcal{B}(\psi))}<1.
							\end{equation*}
							Hence there exists a $k>0$ such that
							\begin{equation*}
								\sup_{Q\in(Q_c-\sigma,Q_c)}
								\mfM(Q)=\sup_{Q\in(Q_c-\sigma,Q_c)}\sup_{\overline{\mcO}}
								\frac{|\nabla\psi|^2}{\mft_c(\mathcal{B}(\psi))}\leq
								1-\epsilon_{n+k}.
							\end{equation*}
							This yields that $Q_{n+k}\leq Q_c-\sigma$ and leads to a contradiction. Therefore,  either
							$\mfM(Q)\rightarrow 1$, or there does not exist a $\sigma>0$ such that 
							Problem \ref{Pb2} has solutions for all $Q\in
							(Q_c-\sigma,Q_c)$} and the associated solutions satisfy (\ref{Ebifurcationestimate}).
						
						This completes the proof of the proposition.
					\end{proof}
					
					Combining all the results in previous sections, Theorem \ref{mainthm} is proved.

					\appendix
					\section{Some technical lemmas}
					In this appendix, we collect some lemmas which provide some technical details used in the paper. The first one devotes to the analysis on shear flows.
					
					\begin{lem}\label{lem:variation_sol}
						Let $h>0$ be a constant. Suppose $\phi_h$ is a minimizer of 
						\[
						\mathfrak{J}_h(\varphi)=\int_{0}^{h}G_\epsilon(|\varphi'|^2,\varphi)
						\] in the admissible set $\mathfrak{K}_h:=\{\varphi\in C^{0,1}([0,h];\mathbb R)| \varphi(0)=0, \varphi(h)=Q\}$, {where the function $G_\epsilon$ is defined in \eqref{Gepsilon} and satisfies \eqref{eq:convex}--\eqref{eq:com_energy} in Proposition \ref{prop:assumption}.} 
						Then 
						\begin{itemize}
							\item [(i)]  $0<\phi_h<Q$ on $(0,h)$; 
							\item [(ii)] $\phi_h'>0$ on $[0,h]$;
							\item [(iii)] the value of $\phi_h'(h)$ is strictly decreasing with respect to $h$;
							\item [(iv)] if $0<h_1<h_2$, one has $\phi_{h_1}>\phi_{h_2}$ on $(0,h_1)$.
						\end{itemize} 
					\end{lem}
					\begin{proof}
						By similar arguments as in the proof for \eqref{ELpde} and Lemma \ref{lem:upper_lower_bd}, one knows that $\phi_h$ satisfies 
						\begin{equation}\label{eq:compare3}
							\left\{
							\begin{aligned}
								(g_\epsilon(|\phi'_h|^2,\phi_h)\phi_h')'-\p_zG_\epsilon(|\phi_h'|^2,\phi_h)=0 \quad&\text{on }  (0, h),\\
								\phi_h(0)=0,\quad \phi_h(h)=Q,
								\quad 0\leq\phi_h\leq Q \quad &\text{on }  [0,h]. 
							\end{aligned}
							\right.
						\end{equation} 
						It suffices to prove (i)--(iv) for {the solutions of}  \eqref{eq:compare3}.
						
						\emph{(i).} For ease of notations, denote $\phi:=\phi_h$. Suppose that $\tilde{\phi}$ is also a solution to the equation in \eqref{eq:compare3}, then $w:=\phi-\tilde{\phi}$ satisfies 
						\begin{equation}\label{1d_equ1}
							(\tilde{\mathfrak a}w'+\tilde{\mfb}w)'-\tilde{\mfb}w' -\tilde{\mfc}w=0 \quad\text{on }  (0, h),
						\end{equation}
						where  
						\begin{align*}
							\tilde{\mathfrak a}:=\int_0^1\p_{pp}G_\epsilon(|\tilde\phi_\tau'|^2,\tilde\phi_\tau)d\tau,
							\quad \tilde{\mfb}:=\int_0^1\p_{pz} G_{\epsilon}(|\tilde\phi'_\tau|^2,\tilde\phi_\tau)d\tau,
							\quad \tilde{\mfc}:= \int_0^1\p_{zz}{G}_\epsilon(|\tilde\phi'_\tau|^2,\tilde\phi_\tau)d\tau,
						\end{align*}
						with $\tilde\phi_\tau=\tilde\phi+\tau(\phi-\tilde\phi)$. Notice $0$ and $Q$ are both solutions to the equation in \eqref{eq:compare3}. Let $w_1:=\phi$ and $w_2:=\phi-Q$. Since  
						$$w_1(0)=0=\min_{x_2\in[0,h]}w_1(x_2)
						\quad\text{and}\quad  w_2(h)=0=\max_{x_2\in[0,h]}w_2(x_2),$$ 
						using the strong maximum principle for the equation \eqref{1d_equ1} yields $w_1>0$ and $w_2<0$ on $(0,h)$. This finishes the proof for (i). 
						
						\emph{(ii).} Denote $\phi:=\phi_h$ and $v:=\phi'_h$. Note that by the Hopf lemma one has $v(0)=w_1'(0)>0$ and $v(h)=w_2'(h)>0$, where $w_1$ and $w_2$ are the same as that defined in (i). Then $v$ satisfies
						\begin{equation}\label{1d_deri_equ}
							\left\{
							\begin{aligned}
								&(\mathfrak a v'+\mfb v)'-\mfb v' -\mfc v=0 \quad\text{on }(0,h),\\
								&v(0)>0,\quad v(h)>0
							\end{aligned}
							\right.
						\end{equation} 
						%\begin{equation}
						%(\mathfrak a v'+\mfb v)'-\mfb v' -\mfc v=0 \quad\text{in }(0,h)
						%\end{equation}
						with 
						\begin{align*}
							\mathfrak a:=\p_{pp}G_\epsilon(|\phi'|^2,\phi) ,\quad \mfb:=\p_{pz} G_{\epsilon}(|\phi'|^2,\phi),\quad \mfc:= \p_{zz}{G}_\epsilon(|\phi'|^2,\phi). 
						\end{align*}
						We claim that $v\geq0$ on $(0,h)$. In fact, let $\phi^{(k)}(x_2):=\phi(x_2+k)$ for $k\geq0$, where $\phi(x_2)$ is extended to be $Q$ for $x_2>h$. Let $k_*:=\inf\{k:\phi\leq \phi^{(k)}\}$.  Suppose $k_*>0$. Since $0<\phi<Q$ on $(0,h)$ by (i), then by continuity necessarily $k_\ast<h$ and furthermore, one can find a point $x_2^*\in(0,h-k_*)$ such that $\phi-\phi^{(k_*)}$ attains its maximum $0$ at $x_2=x_2^*$. Since $\phi-\phi^{(k_*)}$ satisfies \eqref{1d_equ1} in $(0,h-k_*)$ and has a sign, by the strong maximum principle one gets a contradiction. Hence $k_*=0$ and $v=\lim_{k\rightarrow 0+}\frac{\phi(x_2+k)-\phi(x_2)}{k}\geq 0$. Then using the strong maximum principle for \eqref{1d_deri_equ} gives (ii). 
						
						\emph{(iii).} For $0<h_1<h_2$, let $\phi_{h_1}$ and $\phi_{h_2}$ be the corresponding solutions of \eqref{eq:compare3}. We need to show $\phi_{h_1}'(h_1)>\phi_{h_2}'(h_2)$. 
						Let $\bar\phi_{h_2}(x_2):=\phi_{h_2}(x_2+(h_2-h_1))$, then $\bar\phi_{h_2}$ satisfies the equation in \eqref{eq:compare3} for $x_2\in(0,h_1)$. Moreover, 
						$$\phi_{h_1}(0)=0<\bar\phi_{h_2}(0) \quad\text{ and }\quad
						\phi_{h_1}(h_1)=Q=\bar\phi_{h_2}(h_1).$$ 
						By considering $\bar\phi_{h_2}^{(k)}(x_2):=\bar\phi_{h_2}(x_2+k)$ and using similar arguments as in (ii), one can prove  $\phi_{h_1}<\bar\phi_{h_2}$ on $(0,h_1)$.
						%We claim $\phi_1<\bar\phi_2$ in $(0,h_1)$. Indeed, let $\bar\phi_2^{(k)}(x_2)=\bar\phi_2(x_2+k)$ for $k\geq0$, and  $\bar k:=\inf\{k:\phi_1\leq \bar\phi_2^{(k)}\}\in[0,h_1)$.  Suppose $\bar k>0$. Since $0<\phi_1<Q$ in $(0,h_1)$ by $(i)$, one can find a point $\bar x_2\in(0,h_1-\bar k)$ such that $\phi-\phi^{(\bar k)}$ attains its maximum $0$ at $x_2=\bar x_2$. Since $\phi_1-\phi_2^{(\bar k)}$ satisfies \eqref{1d_equ1} in $(0,h_1-\bar k)$, this contradicts the strong maximum principle. Hence $\bar k=0$ and the claim holds true. 
						Notice $(\phi_{h_1}-\bar\phi_{h_2})(h_1)=0$.  Using the Hopf lemma, one has $\phi_{h_1}'(h_1)>\bar\phi_{h_2}'(h_1)=\phi_{h_2}'(h_2)$.
						The proof for (iii) is finished. 
						
						\emph{(iv).} Note that $\phi_{h_1}(0)=0=\phi_{h_2}(0)$ and  $\phi_{h_1}(h_1)=Q>\phi_{h_2}(h_1)$. Using similar arguments as in (ii) and (iii), one has $\phi_{h_1}>\phi_{h_2}$ on $(0,h_1)$.
						
						This completes the proof for the lemma.
					\end{proof}
					
					%\begin{lem}\label{lem:Hopf monotone}
					%	Let $h>0$ be a constant. Assume $\phi_h$ is a solution to  
					%	\begin{align}\label{eq:compare3}\begin{cases}
							%			(g_\epsilon(|\phi'|^2,\phi)\phi')'-\p_z G_\epsilon(|\phi'|^2,\phi)=0 &\text{in } (0,h),\\
							%			\phi(0)=0,\quad 
							%			\phi(h)=Q,\quad 
							%			0\leq \phi\leq Q  &\text{in } [0,h].
							%	\end{cases}\end{align}
					%	Then the value of $\phi_h'(h)$ is strictly decreasing with respect to $h$.\end{lem}
				
				The next lemma gives another version of nondegeneracy.
				\begin{lem}\label{lem_nondegeneracy}
					Let $\bar I$ be a segment joining the points $(\bar x_1\pm a, \bar x_2)$ in $\Omega$ and $\psi:=\psi_\infty$, $\Lambda:=\Lambda_\infty$ be as in Proposition \ref{prop:domain_truncation}. If
					$$\psi=Q \quad \text { in a neighborhood of } \bar I,$$
					and
					$$(\bar x_1, \hat x_2) {\rm~is~a~free~boundary~point~with~}  \hat x_2>\bar x_2,$$
					then
					$$a\leq C,$$
					where $C$ is a constant depending on $\gamma$, $\epsilon$, $\bar u$, $Q$ and  $\Lambda$.
				\end{lem}
				\begin{proof}
					The proof is similar as \cite[Lemma 4.1]{ACF83}. For $\tau\geq0$ define
					$$\gamma_{\tau}:=\left\{(x_1,x_2)\Big| |x_1-\bar x_1|\leq a,~ x_2-\bar x_2=\tau\eta\left(\frac{x_1-\bar x_1}{\sqrt a}\right)\right\},$$
					where
					\begin{align*}\eta(x_1):=\begin{cases}
							e^{-\frac{x_1^2}{1-x_1^2}} &\text{ if } |x_1|<1,\\
							0 &\text{ if } |x_1|\geq1.
					\end{cases}\end{align*}
					Let $\tau_0$ be the smallest value of $\tau$ for which $\gamma_{\tau}$ touches the free boundary at some point, say, $(\tilde x_1,\tilde x_2)$. Notice that
					$$|\tilde x_1-\bar x_1|\leq\sqrt a \quad{\rm and}\quad \tau_0\leq C_0,$$
					where $C_0$ is a universal constant.
					
					For fixed large $m$ consider the domain $D$ bounded by $\gamma_{\tau_0}$, $x_2=\bar x_2+m$ and $x_1=\bar x_1\pm a$. {Let $\phi$ be the solution of  
						\begin{align*}\begin{cases}
								\nabla\cdot(g_\epsilon(|\nabla\phi|^2,\phi)\nabla\phi)-\p_z G_\epsilon(|\nabla\phi|^2,\phi)=0 & {\rm in~}  D,\\
								\phi=Q &{\rm on~} \gamma_{\tau_0},\\
								\phi=0 & {\rm on~} \partial D\backslash \gamma_{\tau_0},\\
								0\leq\phi\leq Q & {\rm in~} D.
						\end{cases}\end{align*}
						The existence of $\phi$ can be obtained by considering the variational problem 
						$$\inf_{K_\phi} \int_D G_\epsilon(|\nabla \varphi|^2,\varphi),$$
						where 
						$$K_\phi=\{\varphi\in H^{1}( D)|\varphi=Q \text{ on } \gamma_{\tau_0}, \ 
						\varphi=0 \text{ on } \p  D\backslash \gamma_{\tau_0}\},$$
						and $0\leq\phi\leq Q$ by similar arguments as in Lemma \ref{lem:upper_lower_bd}. Now we claim that $\psi\geq\phi$ in $D$. In fact, $\psi>\phi$ in a neighborhood of $\partial  D\cap\{x_1=\bar x_1-a\}$ and $\psi\geq\phi$ on the remaining part of $\partial( D\cap\{\psi<Q\})$. By similar arguments as in Lemma \ref{lem:variation_sol}(ii)-(iv) and the strong maximum principle, one gets
						%$\psi\geq\phi$ in $ D\cap\{\psi<Q\}$ and consequently
						$\psi\geq\phi$ in $ D$. Since $\psi=\phi=Q$ at $(\tilde{x}_1,\tilde{x}_2)$,  one has} 
					\begin{equation}\label{eq:compare1}
						\frac{\p \phi}{\p\nu}(\tilde x_1,\tilde x_2)\geq\frac{\p\psi}{\p\nu}(\tilde x_1,\tilde x_2)=\Lambda.
					\end{equation}
					
					Now, if the assertion of the lemma is not true then the above situation holds for a sequence of $\{a_j\}$  with $a_j\to\infty$. Choose a subsequence such that 
					$$(\bar x_1)_j,~(\bar x_2)_j,~(\tau_0)_j,~(\tilde x_2)_j,~ \frac{(\tilde x_1)_j-(\bar x_1)_j}{\sqrt {a_j}}$$
					converge, and let
					$$x_1^*=\lim_{j\to\infty}\frac{(\tilde x_1)_j-(\bar x_1)_j}{\sqrt {a_j}},$$
					then $x_1^*\in[-1,1]$.
					
					After a translation in the $x_1$-direction so that $(\tilde x_1,\tilde x_2)$ lies on the $x_2$-axis, the corresponding domain $D_j$ converges to
					$$ D^*:=\{(x_1,x_2)| x_1\in\mathbb R,~ \bar x^*_2+\tau^*\eta(x^*_1)<x_2<\bar x^*_2+m\}$$
					and $(\tilde x_1,\tilde x_2)$ converges to $(0,\bar x^*_2+\tau^*\eta_1(x_1^*))$, where $\bar x^*_2$ and $\tau^*$ denote the limit values $(\bar x_2)_j$ and $(\tau_0)_j$ in the above construction for each $a_j$. Furthermore, the corresponding solutions $\{\phi_j\}$ converge to a solution $\phi^*$ of
					\begin{align}\label{eq:compare2}\begin{cases}
							\nabla\cdot(g_\epsilon(|\nabla\phi^*|^2,\phi^*)\nabla\phi^*)-\p_z G_\epsilon(|\nabla\phi^*|^2,\phi^*)=0 & {\rm in~}  D^*,\\
							\phi^*=Q &{\rm on~} x_2=\bar x^*_2+\tau^*\eta(x^*_1),\\
							\phi^*=0 & {\rm on~} x_2=
							\bar x^*_2+m,\\
							0\leq \phi^*\leq Q & {\rm in~}  D^*.
					\end{cases}\end{align}
					One deduces from \eqref{eq:compare1} that
					$$\Lambda\leq\lim_{j\to\infty}\frac{\p\phi_j}{\p\nu}\bigg|_{((\tilde x_1)_j,(\tilde x_2)_j)}=\frac{\p\phi^*}{\p\nu}\bigg|_{(0,\bar x_2^*+\tau^*\eta(x_1^*))}\leq C_m$$
					for some constant $C_m$. 
					Now we prove $C_m<\Lambda$ if $m$ is large enough, then there is a contradiction and the lemma follows. 
					{Indeed, the value of $\p\phi^*/\p\nu$ on $\{x_2=\bar x_2^*+\tau^*\eta(x_1^*)\}$ is strictly decreasing with respect to $m$ (similarly as in Lemma \ref{lem:variation_sol}(iii)). 
						On the other hand, for the same $Q$ and $\Lambda$, we know that if $m-\tau^*\eta(x_1^*)=\unH$ ($\unH$ is the downstream asymptotic height of $\psi$ determined in Proposition \ref{prop:asymptotic}), then  $\p\phi^*/\p\nu=\Lambda$ on $\{x_2=x_2^*+\tau^*\eta(x_1^*)\}$. Therefore choosing $m>\unH+1$ one gets $C_m<\Lambda$. This finishes the proof.}
				\end{proof}

		\medskip
		{\bf Acknowledgement.}
		W.~Shi was partially supported by the German Research Foundation (DFG) by the project SH 1403/1-1. L. Tang was partially supported by  NNSFC  grant 11831109  of China and Fundamental Research Grant for Central  Universities (No. CCNU19TS032). The research of Xie was partially supported by NSFC grants {11971307 and 1221101620,  Fundamental Research Grants for Central Universities, Natural Science Foundation of Shanghai 21ZR1433300, and Program of Shanghai Academic Research Leader 22XD1421400.} Part of the work was done when Xie was visiting Institut des Hautes \'{E}tudes Scientifiques (IHES) and The Institute of Mathematical Sciences (IMS) at The Chinese University of Hong Kong. He would like to thank Professor Zhouping Xin for many helpful discussions over the years, the hospitality and support of both IHES and IMS. 
		
		\bibliographystyle{plain}

\begin{thebibliography}{10}
			
			
			
			
			
			
			
			
			
			
			\bibitem{AC81}
			H. W. Alt and L. A. Caffarelli,  { Existence and regularity for a minimum problem with free boundary}, {\it J. Reine Angew. Math.}, {\bf 325} (1981), 105--144.
			
			\bibitem{ACFCPAM}
			H. W. Alt, L. A. Caffarelli,  and A. Friedman, { Asymmetric jet flows}, {\it Comm. Pure Appl. Math.}, {\bf 35} (1982), 29--68.
			
			\bibitem{ACFgravity}
			H. W. Alt, L. A. Caffarelli, and A. Friedman,  Jet flows with gravity. {\it J. Reine Angew. Math.}, {\bf 331} (1982), 58--103.
			
			\bibitem{ACF83}
			H. W.  Alt, L. A. Caffarelli, and A. Friedman, Axially symmetric jet flows, {\it Arch. Rational Mech. Anal.}, {\bf 81} (1983), 97--149.
			
			\bibitem{ACF84}
			H. W. Alt, L. A. Caffarelli,  and A. Friedman,  A free boundary problem for quasilinear elliptic equations, {\it Ann. Scuola Norm. Sup. Pisa Cl. Sci.}, {\bf 11} (1984), no. 1, 1--44.
			
			
			\bibitem{ACFtwop}
			H. W. Alt, L. A. Caffarelli, and  A. Friedman, Variational problems with two phases and their free boundaries, {\it Trans. Amer. Math. Soc.} 282 (1984),  431--461.
			
			\bibitem{ACFtwop1}
			H. W. Alt, L. A. Caffarelli, and A. Friedman, Jets with two fluids. I. One free boundary, {\it Indiana Univ. Math. J.}, {\bf 33} (1984),  213--247.
			\bibitem{ACFtwop2}
			H. W. Alt, L. A.  Caffarelli, and A. Friedman,  Jets with two fluids. II. Two free boundaries, {\it Indiana Univ. Math. J.}, {\bf 33} (1984),  367--391.
			
			
			\bibitem{ACF85}
			H. W. Alt, L. A. Caffarelli, and A. Friedman, Compressible flows of jets and cavities. {\it J. Differential Equations},  {\bf 56} (1985),  82--141.
			
			\bibitem{Batchelor}
			G. K. Batchelor, {\it An introduction to fluid dynamics}, Second paperback edition,  Cambridge Mathematical Library. Cambridge University Press, Cambridge, 1999.
			
			
			\bibitem{Bersbook}
			L. Bers, {\it Mathematical aspects of subsonic and transonic gas
				dynamics}, Surveys in Applied Mathematics, Vol.3, John Wiley and
			Sons, Inc., New York, 1958.
			
			
			\bibitem{Birkhoff}
			G. Birkhoff and E. H. Zarantonello, {\it Jets, wakes, and cavities}, Academic Press Inc., Publishers, New York, 1957.
			
			
			\bibitem{Chaplygin}
			S. Chaplygin, {\it Gas jets}, Tech. Notes Nat. Adv. Comm. Aeronaut., (1944), no. 1063.
			
			\bibitem{CDXX}
			C. Chen, L. L. Du, C. J. Xie, and Z. P. Xin, Two dimensional subsonic Euler flow past a wall or a symmetric body, {\it Arch. Ration. Mech. Anal.}, {\bf 221} (2016),  559--602.
			
			\bibitem{ChenXie1}
			C. Chen and C. J. Xie, \newblock Existence of steady subsonic Euler
			flows through infinitely long periodic nozzles,  {\it J.
				Differential Equations}, {\bf 252} (2012), 4315--4331.
			
			\bibitem{ChenXie2}
			C. Chen and C. J. Xie, \newblock Three dimensional steady subsonic Euler flows in bounded nozzles,
			{\em Journal of Differential Equations}, {\bf 256} (2014), 3684--3708.
			
			\bibitem{CCS}
			{G. Q. Chen,} J. Chen, and K. Song, {Transonic nozzle flows and free boundary problems for the full Euler equations,} {\it J. Differential Equations}, {\bf 229} (2006), 92--120.
			
			
			\bibitem{CDSW}
			{G.~Q. Chen, C.~Dafermos, M.~Slemrod, and D.~H. Wang}, \newblock {On
				two-dimensional sonic-subsonic flow}, {\it Comm. Math. Phys.}, {\bf
				271} (2007), 635--647.
			
			
			\bibitem{CDX}
			{G. Q. Chen, X. M. Deng, and W. Xiang}, \newblock{Global steady
				subsonic flows through infinitely long nozzles for the full Euler
				equations}, {\em SIAM J. Math. Anal.}, {\bf 44} (2012), 2888--2919.
			
			\bibitem{Chen}
			G. Q. Chen and M. Feldman,
			\newblock  Multidimensional transonic shocks and free boundary problems for nonlinear equations of mixted type,
			\newblock {\em J. Amer. Math. Soc.}, {\bf 16} (2003), 461--494.
			
			\bibitem{CHW} G. Q. Chen, F. M. Huang, and T. Y. Wang,
			\newblock Sonic-subsonic limit of approximate solutions to
			multidimensional steady Euler equations,  {\it Arch. Rational Mech. Anal.}, {\bf 219} (2016), no. 2, 719--740.
			\bibitem{CHWX}
			G. Q. Chen, F. M. Huang, T. Y. Wang, and W. Xiang, Steady Euler flows with large vorticity and characteristic discontinuities in arbitrary infinitely long nozzles, {\it Adv. Math.}, {\bf 346} (2019), 946--1008.
			
			
			\bibitem{ChenSX}
			S. X. Chen,  Compressible flow and transonic shock in a diverging nozzle, {\em Comm. Math. Phys.},  {\bf 289} (2009),  75--106.
			
			\bibitem{CSX}
			S. X. Chen and  H. R. Yuan, \newblock Transonic shocks in compressible
			flow passing a duct for three-dimesional Euler systems,
			\newblock{\em Arch. Rational Mech. Anal.}, {\bf 187} (2008), 523--556.
			
			\bibitem{ChengDu}
			J. F.  Cheng and L. L.  Du, A free boundary problem for semi-linear elliptic equation and its applications, preprint, 
			arXiv:2006.02269
			
			
			\bibitem{CDXiang}
			J. F. Cheng, L. L. Du, and W. Xiang, Compressible subsonic jet flows issuing from a nozzle of arbitrary cross-section, {\it J. Differential Equations}, {\bf 266} (2019), 5318--5359.
			
			
			
			\bibitem{Cook}
			L. P. Cook, E.  Newman, S.  Rimbey, and G.  Schleiniger,  Sonic and subsonic axisymmetric nozzle flows, {\em SIAM J. Appl. Math.}, {\bf 59} (1999),  1812--1824.
			
			\bibitem{CF48}
			R. Courant and K. O. Friedrichs, {\it Supsonic flow and shock waves,} Interscience Publ., New York, 1948.
			
			
			\bibitem{D11}
			D.~De~Silva.
			\newblock Free boundary regularity for a problem with right hand side.
			\newblock {\em Interfaces Free Bound}, {\bf 13} (2011), 223--238.
			
			\bibitem{DFS15}
			D. De~Silva, F. Ferrari, and S. Salsa,
			\newblock Two-phase problems with distributed sources: regularity of the free boundary,
			\newblock {\em Analysis $\&$ PDE}, {\bf 7} (2014), 267--310.
			
			
			\bibitem{DD}
			L. L. Du and B. Duan, Global subsonic Euler flows in an infinitely long
			axisymmetric nozzle, {\it J. Differential Equations}, {\bf 250} (2011),
			813--847.
			
			\bibitem{DD1}
			L. L. Du and B. Duan, Note on the uniqueness of subsonic Euler flows in
			an axisymmetric nozzle, {\it Appl. Math. Letters}, {\bf 25} (2012),
			153--156.
			
			\bibitem{DWX}
			L. L. Du, S. K. Weng, and Z. P. Xin,
			\newblock Subsonic irrotational flows in a finitely long nozzle with variable end
			pressure, {\em Comm. Partial Differential Equations}, {\bf 39} (2014), 666--695.
			
			\bibitem{DX}
			L. L. Du and C. J. Xie, On subsonic Euler flows with stagnation points
			in two dimensional nozzles,  Indiana Univ. Math. J.,  {\bf 63} (2014),  1499--1523.
			
			\bibitem{DXX} L. L. Du, C. J. Xie, and Z. P. Xin, Steady subsonic ideal
			flows through an infinitely long nozzle with large vorticity, {\it
				Comm. in Math. Phy.}, {\bf 328} (2014), 327--354.
			
			\bibitem{DXY}
			L. L. Du, Z. P. Xin, and  W. Yan,
			\newblock Subsonic flows in a multi-dimensional
			nozzle, {\em Arch. Rational Mech. Anal.}, {\bf 201} (2011), 965--1012.
			
			
			\bibitem{DL}
			{B. Duan and Z. Luo}, {Three-dimensional full Euler flows in axisymmetric nozzles}, {\em J. Differential Equations},  {\bf 254} (2013),  2705--2731.
			
			\bibitem{Evans}
			L. C. Evans,  {\it Partial differential equations},
			{Graduate Studies in Mathematics},
			{\bf 19}, American Mathematical Society, Providence, 1998.
			
			\bibitem{Friedman82}
			A. Friedman, {\it Variational principles and free-boundary problems}, A Wiley-Interscience Publication, Pure and Applied Mathematics, John Wiley \& Sons, Inc., New York, 1982.
			
			\bibitem{Friedmanrot}
			A. Friedman, Axially symmetric cavities in rotational flows, {\it Comm. Partial Differential Equations}, {\bf 8} (1983),  949--997.
			
			\bibitem{Giaquinta83}
			M. Giaquinta, {\it Multiple integrals in the calculus of variations and nonlinear elliptic systems}, Annals of Mathematics Studies, {\bf 105}, Princeton University Press, Princeton, NJ, 1983.
			
			\bibitem{Gilbarg1}
			D. Gilbarg,  Comparison methods in the theory of subsonic flows, {\it J. Rational Mech. Anal.}, {\bf 2} (1953), 233--251.
			
			\bibitem{Gilbargjets}
			D. Gilbarg, {\it Jets and cavities}, 1960 Handbuch der Physik, Vol. 9,  311--445 Springer-Verlag, Berlin.
			
			
			
			
			\bibitem{GT}
			D. Gilbarg and N. ~S. Trudinger,
			\newblock Elliptic partial differential equations of second order,
			\newblock {\em Classics in Mathematics}, Springer-Verlag, Berlin, 2001.
			
			\bibitem{HanLin}
			Q. Han and F. Lin,
			{\it Elliptic partial differential equations}, Second edition,  Courant Institute of Math. Sci., NYU, American Mathematical Society, 2011.
			
			\bibitem{HWW}
			F. M. Huang, T. Y. Wang, and Y. Wang, \newblock On multi-dimensional
			sonic-subsonic flow, \newblock{\em Acta Mathematica Scientia}, {\bf
				31} (2011), 2131--2140.
			
			\bibitem{KN}
			D. Kinderlehrer and  L. Nirenberg,  Regularity in free boundary problems, {\it Ann. Scuola Norm. Sup. Pisa Cl. Sci.}, {\bf 4} (1977), 373--391.
			
			
			\bibitem{KT01}
			H. Koch and D. Tataru,  Carleman estimates and unique continuation for second-order elliptic equations with nonsmooth coefficients, {\it Comm. Pure Appl. Math.}, 54 (2001),  339--360.
			
			\bibitem{LXY}
			J. Li, Z. P. Xin, and H. C. Yin, \newblock On transonic shocks in a
			nozzle with variable end pressures, \newblock{\em Comm. Math. Phys.},
			{\bf 291} (2009), 111--150.
			
			\bibitem{LXY1}
			{J. Li, Z. P. Xin, and H. C. Yin,} Transonic shocks for the full compressible Euler system in a general two-dimensional de Laval
			nozzle, \newblock{\em Arch. Ration. Mech. Anal.},  {\bf 207} (2013), 533--581.
			
			\bibitem{LSX1}
			Y. Li, W. H. Shi, and  C. J. Xie, The cavity problem for two-dimensional compressible subsonic flows with nonzero vorticity, preprint, 2024.
			
			\bibitem{LSX2}
			Y. Li, W. H. Shi, and  C. J. Xie, The jet problem for three-dimensional axially symmetric full compressible subsonic flows with nonzero vorticity, preprint, 2024.
			
			
			
			\bibitem{LXYuan}
			L. Liu, G. Xu, and H. Yuan, {Stability of Spherically symmetric subsonic flows and transonic shocks under multidimensional perturbations,}  {\it Advance in Mathematics}, {\bf 291} (2016), 696--757.
			
			\bibitem{Majda}
			A. J. Majda and  A. L. Bertozzi,  {\it Vorticity and incompressible flow}, Cambridge Texts in Applied Mathematics, 27. Cambridge University Press, Cambridge, 2002.
			
			
			
			\bibitem{R18}
			L. Rosales,  A Hopf-type boundary point lemma for pairs of solutions to quasilinear equations, {\it Canad. Math. Bull.}, {\bf 62} (2019),  607--621.
			
			
			
			\bibitem{Strauss}
			W.  Strauss,  Steady water waves, {\it Bull. Amer. Math. Soc. },  {\bf 47} (2010),  671--694.
			
			\bibitem{WX2}
			C. P. Wang and  Z. P. Xin,
			\newblock On a degenerate free boundary problem and continuous subsonic-sonic flows
			in a convergent nozzle,
			\newblock {\it Arch. Rational Mech. Anal.}, {\bf 208} (2013), 911--975.
			
			
			\bibitem{WX1}
			{C. P. Wang and  Z. P. Xin,
			\newblock
			On an elliptic free boundary problem and subsonic jet flows for a given surrounding pressure, {\it SIAM J. Math. Anal.}, {\bf 51}  (2019), 1014--1045.}
			
			
			\bibitem{W2}
			S. K. Weng, Subsonic irrotational flows in a two-dimensional
			finitely long curved nozzle, {\em Zeitschrift f\"{u}r
				Angewandte Mathematik und Physik}, {\bf 65} (2014), 203--220.
			
			\bibitem{WXX}
			{S. K. Weng, C. J. Xie, and Z. P. Xin,
			Structural stability of the transonic shock problem in a divergent three dimensional axisymmetric perturbed nozzle, {\it SIAM J. Math. Anal.}, {\bf 53}  (2021), 279--308.}
			
			\bibitem{XX1}
			C. J. Xie and Z. P. Xin,
			\newblock Global subsonic and subsonic-sonic flows through infinitely long
			nozzles,
			\newblock {\em Indiana Univ. Math. J.}, {\bf56} (2007), 2991--3023.
			
			\bibitem{XX2}
			C. J. Xie and  Z. P. Xin,
			\newblock Global subsonic and subsonic-sonic flows through infinitely long axially symmetric nozzles,
			\newblock {\em J. Differential Equations}, {\bf 248} (2010), 2657--2683.
			
			\bibitem{XX3}
			C. J. Xie and Z. P. Xin, Existence of global steady subsonic Euler
			flows through infinitely long nozzle, {\it SIAM J. Math. Anal.},
			{\bf 42}  (2010), 751--784.
			
			\bibitem{XY1}
			Z. P. Xin and H. C. Yin, Transonic shock in a nozzle, I.
			Two-dimensional case, {\em Comm. Pure Appl. Math.}, {\bf 58} (2005),
			999--1050.
			
			\bibitem{XY2}
			Z. P. Xin and  H. C. Yin, Three-dimensional transonic shocks in a nozzle, {\it Pacific J. Math.}, {\bf 236} (2008),  139--193.
			
			\bibitem{XYJDE}
			Z. P. Xin and  H. C. Yin, The transonic shock in a nozzle, 2-D and 3-D
			complete Euler systems, \newblock{\em J. Differential Equations},
			{\bf 245} (2008), 1014--1085.
			
		\end{thebibliography}
		
	\end{document}